\documentclass[12 pt, a4paper]{article}
\pdfoutput=1
\usepackage{amsfonts}
\usepackage{amsmath}
\usepackage{amssymb}
\usepackage{amsthm}
\usepackage{subfigure}
\usepackage[pdftex]{graphicx}
\usepackage{natbib}
\usepackage[hypertexnames=false]{hyperref}
\usepackage{setspace}
\usepackage{lscape}
\usepackage[backgroundcolor=blue!40,linecolor=blue!40,disable]{todonotes}
\DeclareGraphicsExtensions{.jpg}

\newtheorem{theorem}{Theorem}[section]

\newtheorem{lemma}{Lemma}[section]
\newtheorem{definition}{Definition}[section]
\newtheorem{example}{Example}[section]

\newtheorem{remark}{Remark}[section]
\newtheorem{assumption}{Assumption}[section]

\newcommand{\citeposs}[1]{\citeauthor{#1}'s \citeyearpar{#1}}

\setdisplayskipstretch{0.8}
\author{Hiroaki Kaido\\
Boston University\thanks{The author acknowledges excellent research assistance from Michael Gechter. The author also thanks  Iv\'{a}n Fern\'{a}ndez-Val, Francesca Molinari, Pierre Perron, Zhongjun Qu, Andres Santos, J\"{o}rg Stoye, and seminar participants at BU, Cornell, Yale and participants of BU-BC Joint Econometrics Conference, the CEMMAP Workshop (New Developments in the Use of Random Sets in Economics), and the 2012 International Symposium on Econometric Theory and Applications (SETA) for helpful comments. Financial support from the NSF Grant SES-1230071 is gratefully acknowledged.}}
\title{Asymptotically Efficient Estimation of Weighted Average Derivatives with an Interval Censored Variable}
\date{December, 2013}
\begin{document}
	\maketitle

\begin{abstract}
This paper studies the identification and estimation of weighted average derivatives of conditional location functionals including conditional mean and conditional quantiles in settings where either the outcome variable or a regressor is interval-valued.
Building on  \cite{Manski_Tamer2002aE} who study nonparametric bounds for mean regression with interval data, we characterize the identified set of  weighted average derivatives of   regression functions. Since the weighted average derivatives do not rely on parametric specifications for the regression functions, the identified set is well-defined without any parametric assumptions.
Under general conditions, the identified set is compact and convex and hence admits  characterization by its support function.   
Using this characterization, we derive the semiparametric efficiency bound of the support function when the outcome variable is interval-valued. We illustrate efficient estimation by constructing an efficient estimator of the support function for the case of  mean regression with an interval censored outcome.
\end{abstract}	

\bigskip
\begin{center}
\end{center}	
	
\clearpage
\section{Introduction}
Interval censoring commonly occurs in various  economic data used in empirical studies. The Health and Retirement Study (HRS), for example, offers wealth brackets to respondents if they are not willing to provide point values for different components of wealth. In  real estate data, locations of houses are often recorded by zip codes, which makes the distance between any two locations  interval-valued. 
Analyzing regression models with such interval-valued data poses a challenge as the regression function is not generally point identified. 
This paper studies  the identification and estimation of weighted average derivatives of general regression functions when data include an interval-valued variable.

Let $Y\in\mathcal Y\subseteq \mathbb R$ denote an outcome variable and let  $Z\in\mathcal Z\subseteq\mathbb R^\ell$ be a vector of covariates.
The researcher's interest is often in the regression function defined by
\begin{align}
m(z)\equiv \text{argmin}_{\tilde m}E[\varrho(Y-\tilde m)|Z=z]~,\label{eq:reg}
\end{align}
for some loss function  $\varrho:\mathbb R\to\mathbb R$. 
For example, $m$ is the conditional mean function of $Y$ given $Z$ when $\varrho$ is the square loss, i.e., $\varrho(\epsilon)=\epsilon^2/2$, while $m$ is the the conditional quantile function when $\varrho(\epsilon)=\epsilon(\alpha-1\{\epsilon\le  0\})$. 
Our focus is on estimating the identified features of $m$ when either the outcome variable or one of the covariates is interval-valued. A variable of interest is \textit{interval-valued} when  the researcher does not observe the variable of interest $W$ but observes a pair $(W_L,W_U)$ of random variables such that
\begin{align}
W_L\le W\le W_U,~~\text{with probability 1.}\label{eq:interval}
\end{align}
In the presence of an interval-valued variable, data in general do not provide information sufficient for identifying $m$. Yet, they may provide informative bounds on $m$. In their pioneering work, \cite{Manski_Tamer2002aE} derive sharp nonparametric bounds on the conditional mean function when either an outcome or a regressor is interval-valued.
Suppose for example that the outcome variable $Y$ is interval-valued. Letting $m_L$ and $m_U$ denote the solutions to \eqref{eq:reg} with $Y_L$ and $Y_U$ in place of $Y$ respectively and letting $\varrho(\epsilon)=\epsilon^2/2$,  the bounds of \cite{Manski_Tamer2002aE} are given by
\begin{align}
m_{L}(Z)\le m(Z)\le m_{U}(Z),~ \text{with probability 1}.\label{eq:bounds}
\end{align}
 When $Y$ is observed  and a component $V$ of the vector of covariates $(Z',V)$ is interval-valued, similar nonparametric bounds can be obtained when the researcher can assume that the regression function  is weakly monotonic in $V$.

Recent developments in the partial identification literature  allow us to conduct inference for the identified features of the regression function when inequality restrictions such as \eqref{eq:bounds} are available.  For example, when the functional form of $m$ is known up to a finite dimensional parameter, one may construct a confidence set that covers either the identified set of parameters or points inside it with a prescribed probability: \citep{ChernozhukovHongTamer2007E,AndrewsShi2013E}. One may also conduct inference for the coefficients of the best linear approximation to the regression function \citep{BeresteanuMolinari2008E,ChandrasekharChernozhukovMolinari2011}.
This paper contributes to the literature by studying the estimation of another useful feature of the regression function: the weighted average derivative.

A motivation for studying the weighted average derivative is as follows.
A common way to make inference for $m$ is to  specify its functional form.
For example, one may assume $m(z)=g(z;\gamma_0)$ for some $\gamma_0$, where $g$ is a function known up to a finite dimensional parameter $\gamma$. The identified set   for $\gamma_0$ is then defined as the set of $\gamma$'s that satisfy the inequality restrictions: $m_L(Z)\le g(Z;\gamma)\le m_U(Z)$ with probability 1.
Existing estimation and inference methods for partially identified models  can be employed to construct  confidence sets for $\gamma_0$ or its identified region.
However,  such inference may be invalid if $g$ is misspecified, a point raised by \cite{PonomarevaTamer2011EJ}. 
In contrast, the weighted average derivative is well-defined without functional form assumptions.\footnote{Another parameter that is also robust to misspecification is the coefficients in the best linear approximation to $m$. Inference methods for this parameter are studied  in \cite{BeresteanuMolinari2008E}, \cite{PonomarevaTamer2011EJ}, and \cite{ChandrasekharChernozhukovMolinari2011}.}   Suppose $m$ is differentiable with respect to $z$ $a.e$.  Letting $w:\mathcal Z\to \mathbb R_+$ be a weight function,  the \textit{weighted average derivative} of $m$ is defined by
\begin{align}
	\theta\equiv E[w(Z)\nabla_zm(Z)].\label{eq:weightedAD1}
\end{align}	
\cite{Stoker1986EJES} first analyzed estimation of this parameter. It has also been 
studied in a variety of empirical studies, including \cite{HardleHildenbrandJerison1991E,DeatonNg1998JASA}, \cite{CoppejansSieg2005JBES}, \cite{CarneiroHeckmanVytlacil2010E}, and \cite{CrossleyPendakur2010JBES}.
This parameter allows a simple interpretation: the weighted average of marginal impacts of $Z$ on specific features (e.g., conditional quantiles) of the distribution of $Y$.  Further, under suitable assumptions on the data generating process, it can also serve as a structural parameter associated with the function of interest. For example, if  $Y$ is generated as $Y=G(Z)+\epsilon$ with $G$ being a structural function and  $\epsilon$ being mean independent of $Z$, the average derivative of the conditional mean summarizes the slope of the structural function $G$.

In the presence of interval-valued data, the weighted average derivative  is generally set identified.  This paper's first contribution is to characterize the identified set, the set  of weighted average derivatives compatible with the distribution of the observed variables. 
Specifically, we show that the identified set is compact and convex under mild assumptions. This  allows us to represent the identified  set by its \textit{support function}, a unique function on the unit sphere that characterizes the location of hyperplanes tangent to the identified set.  Support functions have recently been used  for making inference for various economic models that involve convex identified sets or have convex predictions \citep[See][]{BeresteanuMolinari2008E,Kaido2010aWP,BeresteanuMolchanovMolinari2011E, ChandrasekharChernozhukovMolinari2011,BontempsMagnacMaurin2012E}. Building on the aforementioned studies, we derive a closed form formula for the support function, which in turn gives an explicit characterization of extreme points of the identified set.
This characterization also gives closed-form bounds on the weighted average derivative with respect to each covariate.

This paper's second  contribution is to characterize the  semiparametric efficiency bound for estimating the identified set when the outcome variable is interval-valued. 
A key insight here is that the support function allows us to interpret the identified set as a parameter taking values in a normed vector space.
In recent work, using the theory of semiparametric efficiency for infinite dimensional parameters, \cite{KaidoSantos2011} characterize the semiparametric efficiency bound for estimating parameter sets defined by convex moment inequalities. Applying their framework, we  characterize the semiparametric efficiency bound for the support function of the identified set of the weighted average derivatives.  
Using  mean regression as an example, we further illustrate efficient estimation by showing that an estimator of the identified set for the density weighted average derivative, which builds on \cite{PowellStockStoker1989EJES}, is semiparametrically efficient.
When the interval censoring occurs on a covariate,
the nonparametric bounds on the regression function take the form of intersection bounds. We show that the support function of the identified set also depends on these bounds.
 As pointed out by \cite{Hirano_Porter2009aE}, intersection  bounds are not generally pathwise differentiable, which implies that the identified set does not generally admit regular estimation when a covariate is interval-valued. 
We then discuss a possibility of regular estimation of the support function of another parameter set, which conservatively approximates the true identified set.

This paper is related to the broad literature on semiparametric estimation of weighted average derivatives. For the mean regression function,
 \cite{Stoker1986EJES} and \cite{HardleStoker1989JASA} study estimation of unweighted average derivatives, while 
\cite{PowellStockStoker1989EJES} study estimation of the density weighted average derivative. \cite{ChaudhuriDoksumSamarov1997AS} study the weighted average derivative of the quantile regression function. Semiparametric efficiency bounds are shown to exist in these settings. See for example, \cite{Samarov1991NFERT},  \cite{ChaudhuriDoksumSamarov1997AS}, and \cite{SeveriniTripathi2001JE}.
This paper's efficiency results build on \cite{NeweyStoker1993EJES}, who characterize the efficiency bound for the average derivative of general regression functions that are defined through minimizations of various loss functions.  
 The optimal bandwidth for point identified average derivatives is studied in \cite{HardleTsybakov1993JE} and \cite{PowellStoker1996JE}.
Other work on semiparametric inference on average derivatives include
 \cite{NishiyamaRobinson2000E,NishiyamaRobinson2005E} who study the higher-order properties of \citeposs{PowellStockStoker1989EJES} estimator through Edgeworth expansions and \cite{Cattaneo2010SA,CattaneoCrumpJansson2013ET} who study small bandwidth asymptotics for the estimation of density weighted average derivatives and a robust data driven selector of bandwidth.

The rest of the paper is organized as follows. Section \ref{sec:setup} presents the model and characterizes the identified sets. Section \ref{sec:effbound} gives our main results on the efficiency bounds. Section \ref{sec:est} constructs an efficient estimator of $\Theta_0(P)$ for the mean regression example. We examine the finite sample performance of the estimator in Section \ref{sec:mc} and conclude in Section \ref{sec:conclusion}.

\section{General Setting}\label{sec:setup}
Throughout, we let $X\in\mathcal X\subset \mathbb R^{d_X}$ denote the vector of observables that follows a distribution $P$. 
We assume that the observable covariates $Z\in\mathcal Z\subseteq\mathbb R^\ell$ are  continuously distributed and let $f$ denote the probability density function of $Z$ with respect to Lebesgue measure. Suppose that $w(z)f(z)$ vanishes on the boundary of $\mathcal Z$.
By integration by parts, Eq. \eqref{eq:weightedAD1} can be equivalently written as
\begin{align}
	\theta=\int m(z)l(z)dP(x)~,~~l(z)\equiv -\nabla_zw(z)-w(z)\nabla_zf(z)/f(z).\label{eq:intparts}
\end{align}
This suggests that the weighted average derivative is a bounded (continuous) linear function of $m$ under mild moment conditions on $l$. Hence, bounds on $m$  can provide informative bounds on $\theta$. This  observation is especially useful when no \textit{a priori} bounds on $\nabla_z m$ are available. \footnote{On the other hand, the bounds on $m$ do not generally provide useful bounds on its derivative $\nabla_z m(z)$ evaluated at a point $z$.}

\subsection{Motivating examples}
To fix ideas, we briefly discuss examples of regression problems with interval censoring. The first example is based on  nonparametric demand analysis \citep[See e.g.][]{DeatonNg1998JASA}.
\begin{example}\rm
	Let $Y$ be expenditure on the good of interest. Let $Z$ be a vector of prices of $\ell$ goods. In survey data, expenditures may be reported as brackets, making $Y$  interval-valued. 
	A key element in the analysis  of demand is the effect of a marginal change in the price vector $Z$  on expenditure $Y$. For example, consider the conditional mean $m(z)\equiv E[Y|Z=z]$ of the demand.  The (weighted) average marginal impact of price changes  is then measured by  $\theta\equiv E[w(Z)\nabla_z m(Z)]$. Similarly, one may also study the average marginal impact of price changes on the conditional median or other conditional quantiles of the demand.
\end{example}

The second example is estimation of a hedonic price model using quantile regression. 
\begin{example}\rm\label{ex:house}
 Let $Y$ be the price of a house and $Z$ be a $\ell$-dimensional  vector of  house  characteristics. Let  $V$ be the distance between a house and another location relevant for the home value (e.g. a school or a factory causing air pollution).  If data  only record locations by  zip codes,  one may only obtain an interval-valued measurement $[V_L,V_U]$ of the distance, where $V_L$ and $V_U$ are the minimum and maximum distances between  two locations. 
The researcher's interest may be in the upper tail of the house price, in particular in  the weighted average effect  of the $j$-th house characteristic (e.g. square footage) on a specific quantile. Here, the weight function can be chosen so that it puts higher weights on the houses that have specific characteristics the researcher considers relevant.
The weighted average effect can be measured by $\theta^{(j)}$, the $j$-th coordinate of $\theta\equiv E[w(Z)\nabla_z g(Z,v)]$.
\end{example}

\subsection{Identification when $Y$ is interval-valued}
Suppose $Y$ is interval-valued. Throughout this section, we let $x\equiv(y_L,y_U,z')'.$
Given \eqref{eq:intparts}, the sharp identified set for $\theta$ is given by
\begin{align}
\Theta_0(P)\equiv \big\{\theta\in\Theta:\theta=\int m(z)l(z)dP(x),~m \text{ satisfies } \eqref{eq:reg} \text{ with } Y_L\le Y\le Y_U, P-a.s. \big\}.	\label{eq:idset}
\end{align}

In order to characterize the identified set, we make the following assumptions on the covariates $Z$, loss function $\varrho$, and weight function $w$.
	\begin{assumption}\label{as:onY}
	(i)  The distribution of $Z$ and $l(Z)$ are absolutely continuous with respect to the Lebesgue measure on $\mathbb R^\ell$. $\mathcal Z$ is compact and convex with nonempty interior;  (ii) $\varrho$ is convex and satisfies $\varrho(0)=0$, $\varrho(\epsilon)\ge 0, \forall\epsilon$, and $\varrho(\epsilon)\to\infty$ as $|\epsilon|\to\infty$. A measurable function $q(\epsilon)\equiv d\varrho(\epsilon)/d\epsilon$ exists $a.e.$, and $q$ is bounded and continuous $a.e.$ 
	\end{assumption}

	\begin{assumption}\label{as:onw}
 (i) $w(z)$ and $\nabla_z w(z)$ are bounded and continuous on $\mathcal Z$.
	\end{assumption}

As is standard in the literature,	Assumption \ref{as:onY} (i) requires $Z$ to be a continuous random vector, where no component of $Z$ can be functionally determined by other components of $Z$. We also assume that this holds for $l(Z)$.  
	Assumption \ref{as:onY} (ii) imposes regularity condition on the derivative $q$ of the loss function, which may require compactness of the supports of $\epsilon_j=Y_j-m_j(Z), j=L,U$  for some specification of $\varrho$ (e.g. square loss) while it
	is trivially satisfied for the conditional quantile, where $q(\epsilon)=\alpha-1\{\epsilon\le 0\}.$	
We further add the following assumption on $P$.

	\begin{assumption}\label{as:onP1}
			(i) There is a compact set $D\subset\mathbb R$ containing the support of $Y_j$  in its interior for $j=L,U$; (ii)  $w(z)f(z)=0$ on the boundary $\partial\mathcal Z$ of $\mathcal Z$, $\nabla_z f(z)/f(z)$ is continuous, and $E[\|l(Z)\|^{2}]<\infty$;
(iii) For each $z\in\mathcal Z$, $z\mapsto E[q(Y_j - \tilde m)|Z=z] = 0$ has a unique solution at $\tilde m=m_j(z)\in D$ for $j=L,U$;	  (iv) 	 	$m_L,m_U$ are continuously differentiable $a.e.$  with bounded derivatives.
	\end{assumption}
 
Assumption \ref{as:onP1} (ii) is a key condition, which together with other conditions, allows us to write the average derivative as in \eqref{eq:intparts}. This, in turn, enables us to obtain  bounds on $\theta$ from the nonparametric bounds on $m$. Assumption \ref{as:onP1} (ii)-(iii) then impose regularity conditions on the density function $f$ and require that the regression functions $m_j,j=L,U$ are well defined. When Assumption \ref{as:onP1} (iv) holds, we can use it to ensure the sharpness of the identified set.

Under Assumptions \ref{as:onY}-\ref{as:onP1},  $\Theta_0(P)$ is  a compact convex set. Hence, it can be uniquely characterized by its support function.
Let $\mathbb S^\ell=\{p\in \mathbb R^\ell:\|p\|=1\}$ denote the unit sphere in $\mathbb R^\ell$. For a bounded convex set $F$, the \emph{support function} of $F$ is defined by
\begin{align}
 \upsilon(p,F)\equiv\sup_{x\in F}\langle p,x\rangle.	\label{eq:suppdef}
\end{align}
Theorem \ref{thm:thetabounds} is our first main result, which characterizes the identified set through its support function. Let $\Gamma:\mathbb R^3\to\mathbb R$ be defined pointwise by $\Gamma(w_1,w_2,w_3)=1\{w_3\le 0\}w_1+1\{w_3>0\}w_2.$
	\begin{theorem}\label{thm:thetabounds}
			Suppose Assumptions \ref{as:onY}, \ref{as:onw}, and \ref{as:onP1} (i)-(iii)  hold.
	Suppose further that for each $z\in\mathcal Z$, $E[q(Y-\tilde m)|Z=z]=0$ has a unique solution at $m(z)\in D$, and $m$ is differentiable $a.e.$ with a bounded derivative.	 Then, (a) $\Theta_0(P)$ is compact and strictly convex; (b) the support function of $\Theta_0(P)$ is  given pointwise by:
		\begin{align}
	\upsilon(p,\Theta_0(P))&=\int m_p(z)p'l(z)dP(x),
			\label{eq:suppfunc}
		\end{align}
		where $m_p(z)=\Gamma(m_L(z),m_U(z),p'l(z))= 1\{p'l(z)\le 0\}m_L(z)+1\{p'l(z)>0\}m_U(z)$;
(c) If Assumption \ref{as:onP1} (iv) holds, additionally, $\Theta_0(P)$ is sharp. 
	\end{theorem}
Theorem \ref{thm:thetabounds} suggests  that the support function is given by the inner product between $p$ and an extreme point $\theta^*(p)$, a unique point such that $\langle p,\theta^*(p)\rangle=\upsilon(p,\Theta_0(P))$, which can be expressed as:
	\begin{align}
	\theta^*(p)=\int m_p(z)l(z)dP(x),	\label{eq:extreme}
	\end{align}
where $m_p$ switches between $m_L$ and $m_U$ depending on the sign of  $p'l(z)$.	Heuristically, this comes from the fact that the support function of $\Theta_0(P)$ evaluated at $p$ is the maximized value of the map $m\mapsto E[m(Z)'p'l(Z)]$ subject to the constraint $m_L(Z)\le m(Z)\le m_U(Z), P-a.s.$	The maximum is then achieved by setting $m$ to $m_U$ when $p'l(z)>0$ and to $m_L$ otherwise.
	The form of the support function given in \eqref{eq:suppfunc} belongs to the general class of functions of the form $E[\Gamma(\delta_L(Z),\delta_U(Z),h(p,Z))h(p,Z)]$ for some functions $\delta_L,\delta_U,h$.	 This functional form is common in the literature on the best linear predictor of $m$. \citep[See e.g.][]{Stoye2007RC, BeresteanuMolinari2008E,BontempsMagnacMaurin2012E,ChandrasekharChernozhukovMolinari2011}.

Theorem \ref{thm:thetabounds} also gives closed-form bounds on the weighted average derivative with respect to the $j$-th variable. 
Let $\theta^{(j)}\equiv E[w(Z)\partial m(Z)/\partial z^{(j)}]$. The upper bound on $\theta^{(j)}$ can be obtained by setting $p$ to  $\iota_j$, a vector whose $j$-th component is 1 and other components are 0. The lower bound can be obtained similarly with $p=-\iota_j$.
Therefore, the  bounds on $\theta^{(j)}$ are given as $[\theta_L^{(j)},\theta_U^{(j)}]$ with
\begin{align}
\theta_L^{(j)}&=\int\big[1\{l^{(j)}(z)> 0\}m_L^{(j)}(z)+1\{l^{(j)}(z)\le 0\}m_U^{(j)}(z)\big]l^{(j)}(z)dP(x)\\
\theta_U^{(j)}&=\int\big[1\{l^{(j)}(z)\le 0\}m_L^{(j)}(z)+1\{l^{(j)}(z)> 0\}m_U^{(j)}(z)\big]l^{(j)}(z)dP(x).
\end{align}

\subsection{Identification when a regressor is interval-valued}
We now consider the setting where one of the regressors is interval-valued.
Let the vector of covariates be $(Z,V)$, where $ Z$ is fully observed but $V\in\mathcal V\subseteq\mathbb R$ is unobserved. Suppose that there exists a pair $(V_L,V_U)$ of observables such that $V_L\le V\le V_U$ with probability 1.
Our interest lies in the average derivative of the regression function defined by:
\begin{align}
	g(z,v)&\equiv \text{argmin}_{u}E[\varrho(Y-u)|Z=z,V=v].	\label{eq:intcovopt}
\end{align}
Assuming $g$ is differentiable with respect to $z$ $a.e.,$
we define the weighted average derivative pointwise by
\begin{align}
	\theta_{v}\equiv E[w(Z)\nabla_{z}g(Z,v)],\label{eq:weightedAD2}
\end{align}
where the expectation in \eqref{eq:weightedAD2} is with respect to the distribution of $Z$. $\theta_v$ is the average derivative with respect to the observable covariates, fixing  $V$  at a given value $v$. 
 In order to characterize the identified set for $\theta_v$, we make use of the regression function of $Y$ given all observable variables $\tilde Z\equiv(Z',V_L,V_U)'.$ Specifically, for each $(z',v_L,v_U)$, define
\begin{align}
\gamma(z,v_L,v_U)&\equiv \text{argmin}_{u}E[\varrho(Y-u)|Z=z,V_L=v_L,V_U=v_U].	
\end{align}

We make the following assumptions to characterize the identified set.
\begin{assumption}\label{as:mon}
(i) For each $z\in \mathcal Z$, $g(z,v)$ is weakly increasing in $v$. For each $v\in\mathcal V$, $g(z,v)$ is differentiable in $z$ with a bounded derivative;
(ii) For each $v\in\mathcal V$, it holds that
\begin{align}
E[q(Y-g(Z,V))| \tilde Z=\tilde z,V=v]=E[q(Y-g(Z,V))| Z= z,V=v].	
\end{align}
\end{assumption}

Following \cite{Manski_Tamer2002aE}, Assumption \ref{as:mon} (i) imposes a weak monotonicity assumption on the map $v\mapsto g(z,v)$. Without loss of generality, we here assume that $g(z,\cdot)$ is weakly increasing.
Assumption \ref{as:mon} (ii) is a conditional mean independence assumption of the ``regression residual'' $q(Y-g(Z,v))$ from $(V_L,V_U)$, which means that $(V_L,V_U)$ do not provide any additional information if $V$ is observed.
 In the case of mean regression, this condition reduces to the mean independence (MI) assumption in \cite{Manski_Tamer2002aE}.
 
For each $v$, let $\Xi_L(v)\equiv\{(v_L,v_U):v_L\le v_U\le v\}$ and $\Xi_U(v)\equiv\{(v_L,v_U):v\le v_L\le v_U\}$.
Under Assumptions  \ref{as:mon},  one may
 show that the following functional inequalities hold:
\begin{align}
g_L(Z,v)\le g(Z,v)\le g_U(Z,v),~P-a.s.,~ \text{for all } v,\label{eq:intbounds}	
\end{align}
where 
\begin{align}
	g_L(z,v)\equiv\sup_{(v_L,v_U)\in\Xi_L(v)}\gamma( z,v_L,v_U),~~\text{ and }~~g_U(z,v)\equiv \inf_{(v_L,v_U)\in \Xi_U(v)}\gamma( z,v_L,v_U).\label{eq:defgLgU}
\end{align}
We then assume the following regularity conditions.

\begin{assumption}\label{as:onmC}
	(i) There is a compact set $D\subset\mathbb R$ containing the support of $Y$ in its interior.  
(ii) $E[q( Y - u)|\tilde Z=\tilde z] = 0$ has a unique solution at $u=\gamma(\tilde z)\in D$; (iii) For each $v\in \mathcal V$, $g_j(z,v)$ is differentiable in $z$ with a bounded derivative for $j=L,U$. 
\end{assumption}

This assumption is an analog of Assumption \ref{as:onP1}. When Assumption \ref{as:onmC} (iii) holds, we may use it to ensure the sharpness of the identified set. It requires that the functional bounds $g_j,j=L,U$ are differentiable in $z$. Since $g_j,j=L,U$ defined in \eqref{eq:defgLgU} are  optimal value functions of  parametric optimization problems (indexed by ($z,v$)), this means that the value functions are assumed to obey an envelope theorem. Various sufficient conditions for such  results are  known \citep[see for example][]{Bonnans:2000fk,MilgromSegal2002E}, but this condition may not hold for some settings, in which case the obtained identified set gives possibly non-sharp bounds on the average derivatives.

Using an argument similar to the one used to establish Theorem \ref{thm:thetabounds}, we now characterize the identified set $\Theta_{0,v}(P)$ for $\theta_{v}$ through its support function. 
\begin{theorem}\label{thm:thetabounds2}
		Suppose Assumptions \ref{as:onY}-\ref{as:onw}, \ref{as:onP1} (ii),  \ref{as:mon}, and \ref{as:onmC} (i)-(ii) hold.
	Suppose further that for each $z\in\mathcal Z$ and $v\in\mathcal V$, $E[q(Y-u)|Z=z,V=v]=0$ has a unique solution at $u=g(z,v)\in D$.	 Then, (a) $\Theta_{0,v}(P)$ is compact and strictly convex; (b) its support function is  given pointwise by
	\begin{align}
\upsilon(p,\Theta_{0,v}(P))=\int g_{p}(z,v)p'l( z)dP(x),
	\end{align}
	where $g_p(z,v)=\Gamma(g_L(z,v),g_U(z,v),p'l(z))= 1\{p'l(z)\le 0\}g_L(z,v)+1\{p'l(z)>0\}g_U(z,v)$;
(c) If, additionally, Assumption \ref{as:onmC} (iii) holds, $\Theta_{0,v}(P)$ is sharp. 
\end{theorem}

\section{Efficiency Bound}\label{sec:effbound}
In this section, we show that a semiparametric efficiency bound exists for estimation of the support function when $Y$ is interval-valued.
Throughout, we assume that observed data $\{X_i\}_{i=i}^n$ are independently and identically distributed (i.i.d.).
\subsection{Efficiency bound for an infinite dimensional parameter}
We first discuss elements of efficiency analysis.\footnote{The theory of semiparametric efficiency for infinite dimensional parameters is discussed in detail in Chapter 5 in \cite{BickelKlassenRitov1993}. See also \cite{KaidoSantos2011}.}
 Suppose that $P$ is absolutely continuous with respect to a $\sigma$-finite measure $\mu$ on $\mathcal X$, which we denote by $P\ll\mu$. 
Let $\mathbf M$ denote the set of Borel probability measures on $\mathcal X$ and let $\mathbf M_{\mu} \equiv \{P \in \mathbf M : P \ll \mu\}$. 
The set $\mathbf M_{\mu}$  then may be identified with a subset of the space $L^2_\mu$, where:
\begin{align}
L^2_\mu\equiv\{s:\mathcal X\to\mathbf R:\|s\|_{L^2_\mu}<\infty\}~,\hspace{0.5in}\|s\|^2_{L^2_\mu}=\int s^2(x)d\mu(x).
\end{align}
via the mapping $s\equiv\sqrt{dP/d\mu}.$ A model $\mathbf P$ is then a subset of $\mathbf M_\mu$. 
We  define curves and the tangent space in the usual manner.
\begin{definition}\label{def:curve}\rm
A function $h :N \rightarrow L^2_\mu$ is a curve in $L^2_\mu$ if $N\subseteq \mathbb R$ is a neighborhood of zero and $\eta \mapsto h(\eta)$ is continuously Fr\'{e}chet differentiable on $N$. For notational simplicity, we write $h_\eta$ for $h(\eta)$ and let $\dot{h}_\eta$ denote its Fr\'{e}chet derivative at any point $\eta \in N$. 
\end{definition}

The tangent space of a model is then defined as follows.
\begin{definition}\label{def:tangent}\rm
For $\mathbf S \subseteq L^2_\mu$ and a function $s_0 \in \mathbf S$, the tangent set of $\mathbf S$ at $s_0$ is defined as:
\begin{equation}\label{deftangentdisp}
\dot{\mathbf S}^0 \equiv \{\dot h_\eta : \eta \mapsto h_\eta \text{ is a curve in } L^2_\mu \text{ with } h_0 = s_0 \text{ and } h_\eta \in \mathbf S \text{ for all } \eta \} ~.
\end{equation}
The tangent space of $\mathbf S$ at $s_0$, denoted by $\dot {\mathbf S}$, is the closure of the linear span of $\dot{\mathbf S}^0$ (in $L^2_\mu$). 
\end{definition}

The support function $\upsilon(\cdot,\Theta_0(P))$ is a continuous function on the unit sphere. Following Kaido and Santos (2011), we view it as a function-valued parameter taking values in $\mathcal C(\mathbb S^\ell)$, the set of continuous functions on $\mathbb S^\ell$.
A parameter defined on $\mathcal C(\mathbb S^\ell)$ is then a mapping $\rho : \mathbf P \rightarrow \mathcal C(\mathbb S^\ell)$ that assigns to each probability measure $Q \in \mathbf P$ a corresponding function in $\mathcal C(\mathbb S^\ell)$.
In order to derive a semiparametric efficiency bound for estimating $\rho(P)$, we require that the mapping $\rho : \mathbf P \rightarrow \mathcal C(\mathbb S^\ell)$ be smooth in the sense of being pathwise weak-differentiable.

\begin{definition}\label{def:weakpathdiff}\rm
	For a model $\mathbf P \subseteq \mathbf M_\mu$ and a parameter $\rho : \mathbf P \rightarrow \mathcal C(\mathbb S^{\ell})$ we say $\rho$ is pathwise weak-differentiable at $P$ if there is a continuous linear operator $\dot{\rho} :\dot{\mathbf S} \rightarrow \mathcal C(\mathbb S^{\ell})$ such that
	\begin{equation}
	\lim_{\eta\rightarrow 0} |\int_{\mathbb S^{d_\theta}} \{ \frac{\rho(h_\eta)(p) - \rho(h_0)(p)}{\eta} - \dot \rho(\dot h_0)(p)\}dB(p)| = 0 ~,
	\end{equation}
	for any finite Borel measure $B$ on $\mathbb S^{\ell}$ and any curve $\eta \mapsto h_\eta$ with $h_\eta\in\mathbf S$ and $h_0 = \sqrt {dP/d\mu}$.
\end{definition}

Given these definitions, we now present the semiparametric efficiency notion for estimating $\rho(P).$ When $\rho$ is pathwise weak-differentiable at $P$, the H\'{a}jek-LeCam convolution theorem (See  \cite{BickelKlassenRitov1993}) implies that any regular estimator of $\rho(P)$ converges in law to  $\mathbb G$:
$$\mathbb G \stackrel{L}{=} \mathbb G_0 + \Delta_0 ~,$$
where $\stackrel{L}{=}$ denotes equality in law and $\Delta_0$ is a tight Borel measurable random element in $\mathcal C(\mathbb S^\ell)$ independent of $\mathbb G_0$.\footnote{$\{T_n\}$ is regular if there is a tight Borel measurable $\mathbb G$ on $\mathcal C(\mathbb S^\ell)$ such that for every curve $\eta \mapsto h_\eta$ in $\mathbf P$ passing through $p_0 \equiv \sqrt{dP_0/d\mu}$ and every $\{\eta_n\}$ with $\eta_n = O(n^{-\frac{1}{2}})$, $\sqrt{n}(T_n - \rho(h_{\eta_n})) \stackrel{L_n}{\rightarrow} \mathbb G$ where $L_n$ is the law under $P_{\eta_n}^n$.}
 Thus, a regular estimator may be considered efficient if its asymptotic distribution equals that of $\mathbb G_0$.  We characterize the semiparametric efficiency bound by computing the covariance kernel for the Gaussian process $\mathbb G_0$, denoted:
\begin{equation}\label{def:cov}
I^{-1}(p_1,p_2)\equiv \text{Cov}(\mathbb G_0(p_1), \mathbb G_0(p_2))
\end{equation}
and usually termed the \textit{inverse information covariance functional} for $\rho$ in the model $\mathbf P$.

\subsection{Semiparametric efficiency bounds for support functions}
In this section, we derive the semiparametric efficiency bound for estimating the support function obtained in Theorem \ref{thm:thetabounds}.
Toward this end, we introduce some additional notation and add regularity conditions.
For each $z\in\mathcal Z$ and $j=L,U$, let
\begin{align}
	r_j(z)\equiv-\frac{d}{d\tilde m}E[q(Y_j-\tilde m)|Z=z]\Big|_{\tilde m=m_j(z)}.~~
\end{align}
For the mean regression, $r_j(z)$ equals 1, and for the quantile regression, $r_j(z)=f_{Y_j|Z}(m_j(z)|z)$, where $f_{Y_j|Z}$ is the conditional density functions of $Y_j$ given $Z$ for $j=L,U$.  
\begin{assumption}\label{as:onmu}
	(i) $\mu\in\mathbf M$ satisfies $\mu(\{(y_L,y_U,z):y_L\le y_U\})=1$; (ii) $\mu(\{(y_L,y_U,z):F(z)=0\})=0$ for any measurable function $F:\mathbb R^{\ell}\to \mathbb R$.
\end{assumption}

\begin{assumption}\label{as:onP2}
(i)	There exists $\bar \epsilon>0$ such that  $|r_L(z)|>\bar\epsilon$ and $|r_U(z)|>\bar\epsilon$ for all $z\in\mathcal Z$.
(ii) For any $\varphi:\mathcal X\to\mathbb R$ that is bounded and continuously differentiable in $z$ with bounded derivatives, $E[\varphi(X)|Z=z]$ is continuously differentiable in $z$ on $\mathcal Z$ with bounded derivatives; (iii) $E[q(Y_j-\tilde m) \varphi(X)|Z=z]$  is continuously differentiable in $(z,\tilde m)$ on $\mathcal Z\times D$ with bounded derivatives for $j=L,U$. 
\end{assumption}

Since $P\ll \mu$, Assumption \ref{as:onmu} (i) implies $Y_L\le Y_U$, $P-a.s.$ Similarly, $P\ll \mu$ and Assumption \ref{as:onmu} (ii) implies $P(F(Z)=0)=0$, which is used to establish pathwise differentiability of the support function. 
 Assumption \ref{as:onP2} gives additional regularity conditions on $P$.
Assumption \ref{as:onP2} (i) is trivially satisfied for the conditional mean because $r_L(z)=r_U(z)=1$. For the conditional $\alpha$-quantile,   Assumption \ref{as:onP2} (i)  requires  the conditional densities of $Y_L$ and $Y_U$ to be positive on neighborhoods of $m_L(z)$ and $m_U(z)$ respectively. Assumption \ref{as:onP2} (ii)-(iii) are regularity conditions   invoked in \cite{NeweyStoker1993EJES}, which we also impose here. It should be noted that these conditions exclude the setting  where either $Y_L$ or $Y_U$ is discrete and $q$ admits a point of discontinuity on the support of $\epsilon_L$ or $\epsilon_U$.

Given these assumptions,  we now define our model as the set of distributions that satisfy Assumptions \ref{as:onP1} and \ref{as:onP2}:
\begin{align}
\mathbf P\equiv\{P\in\mathbf M:P\ll\mu, P\text{ satisfies Assumptions \ref{as:onP1} and \ref{as:onP2}}\}.
\end{align}

For each $Q\in\mathbf P$,  define $\rho(Q)\equiv \upsilon(\cdot,\Theta_0(Q))$. 
The following theorem characterizes the efficiency bound for the support function.

\begin{theorem}\label{thm:InvCov}
Suppose Assumptions \ref{as:onY}-\ref{as:onw}, and \ref{as:onmu} hold, and suppose $P\in\mathbf P$.  Then, the inverse information covariance functional is given by
	\begin{align}
I^{-1}(p_1,p_2)=E[\psi_{p_1}(X)\psi_{p_2}(X)]~,\label{eq:invcov}
	\end{align}
	where for each $p\in\mathbb S^\ell$, the \emph{efficient influence function} $\psi_p$ is 
	\begin{align}
\psi_p(x)&\equiv 
w(z)p'\nabla_zm_{p}(z)-\upsilon(p,\Theta_0(P_0))+p'l(z)\zeta_{p}(x),\label{eq:infl1}
	\end{align}
and $\nabla_zm_{p}$ and $\zeta_p$ are given by $\nabla_zm_{p}(z)=\Gamma(\nabla_z m_L(z),\nabla_z m_U(z),p'l(z))$ and
	$\zeta_p(x)=\Gamma(r^{-1}_L(z)q(y_L-m_L(z)),r^{-1}_U(z)q(y_U-m_U(z)),p'l(z)).$
\end{theorem}
Theorem \ref{thm:InvCov} naturally extends Theorem 3.1 in \cite{NeweyStoker1993EJES} to the current setting. When there is no interval censoring, i.e. $m_L(Z)=m(Z)=m_U(Z)$, the obtained semiparametric efficiency bound reduces to that of \cite{NeweyStoker1993EJES}, i.e. $\psi_p=p'\psi$, where $\psi$ is the efficient influence function for point identified $\theta$.
Theorem \ref{thm:InvCov} also shows that the variance bound for estimating the support function $\upsilon(p,\Theta_0(P))$ at $p$ is given by
\begin{align}
E[|\psi_p(X)|^2]=Var(w(Z)p'\nabla_zm_{p}(Z))+E[|p'l(Z)\zeta_p(x)|^2].	\label{eq:vardec}
\end{align}
The first term in \eqref{eq:vardec} can be interpreted as the variance bound when $m_p$ is known but $f$ is unknown as this is the asymptotic variance of $\frac{1}{n}\sum_{i=1}^n w(Z_i)p'\nabla_zm_p(Z_i)$, while the second term can be interpreted as the variance bound when $f$ is known but $m_p$ is unknown. \citep[See][page 1205 for a more detailed discussion.]{NeweyStoker1993EJES} 

Theorem \ref{thm:InvCov} establishes that the support function of the identified set has a finite efficiency bound. In the next section, we show that it is possible to construct an estimator that achieves this bound in a leading example. Efficient estimation of the support function also has an important consequence on estimation of the identified set.
Namely, an estimtor of the identified set which is constructed from the efficient estimator of the support function is also asymptotically optimal for a wide class of loss functions based on the Hausdorff distance. (See Remark \ref{rem:setest}.)

\begin{remark}\rm
\cite{PowellStockStoker1989EJES} study the setting where $m$ is the conditional mean, and the weight function is the \textit{density weight}: $w(z)=f(z)$.
The efficiency bound in Theorem \ref{thm:InvCov} can be extended to this setting. For this choice of the weight function, the efficient influence function differs slightly from Eq. \eqref{eq:infl1} due to $f$ being unknown. Taking into account the pathwise derivative of unknown $f$, one can show that the inverse covariance functional for this case is given as in \eqref{eq:invcov} with
	\begin{align}
\psi_{p}(x)\equiv 2\{f(z)p'\nabla_zm_{p}(z)-\upsilon(p,\Theta_0(P))\}-2p'\nabla_z f(z)(y_{p}-m_p(z))~,\label{eq:psspsip}
	\end{align}
	where $y_p=\Gamma(y_L,y_U,p'l(z)).$
\end{remark}

\begin{remark}\rm
 For the setting where an explanatory variable is interval-valued, Theorem \ref{thm:thetabounds2}	shows that the support function of $\Theta_{0,v}(P)$ involves functions that are defined as the supremum (or the infimum) of functions indexed by $(v_L,v_U)$, e.g. $g_L(z,v)=\sup_{(v_L,v_U)\in\Xi_L(v)}\gamma(z,v_L,v_U)$. These types of bounds are known as the \textit{intersection bounds} \citep{Chernozhukov:2009uq}. In particular, for  parametric submodels $\eta\mapsto P_\eta$ passing through $P$, one may show that the support function depends on the intersection bounds in the following way:
\begin{align}\upsilon(p,\Theta_{0,v}(P_\eta)) = \int \big[1\{p'l_\eta(z)\le 0\}g_{L,\eta}(z,v)+1\{p'l_\eta(z)> 0\}g_{U,\eta}(z,v)\big]p'l_\eta(z)dP_\eta(x),
	\end{align}
	where $l_\eta$ is defined as in \eqref{eq:intparts} and $g_{L,\eta},g_{U,\eta}$ are defined as in \eqref{eq:defgLgU} under $P_\eta$.
 \cite{Hirano:2012uq} give general conditions under which intersection bounds are not pathwise differentiable therefore do not  admit regular estimation. When the set of $z$'s on which $g_{L,\eta}$ or $g_{U,\eta}$ is pathwise non-differentiable has a positive probability mass, the support function is pathwise non-differentiable as well. Hence, $\upsilon(p,\Theta_{0,v}(P))$ does not generally admit regular estimation. Therefore, for optimal inference on $\upsilon(p,\Theta_{0,v}(P))$, an alternative optimal criterion would be needed. \citep[See for example][]{Song:2010uq,Chernozhukov:2009uq}. 
 
There is, however, a possibility for regular estimation of a function that approximates $\upsilon(p,\Theta_{0,v}(P))$.
For simplicity, suppose that $V_L$ and $V_U$ have discrete supports so that $\Xi_L(v)$ and $\Xi_U(v)$ are finite sets. Then for a given $\kappa>0$, define 
\begin{align}
	\mathbf g_L(z,v;\kappa)&\equiv\sum_{(v_L,v_U)\in \Xi_L(v)}\gamma(z,v_L,v_U)\frac{\exp(\kappa \gamma(z,v_L,v_U))}{\sum_{(v_L,v_U)\in \Xi_L(v)}\exp(\kappa \gamma(z,v_L,v_U))}\label{eq:smoothmax}\\
	\mathbf g_U(z,v;\kappa)&\equiv\sum_{(v_L,v_U)\in \Xi_U(v)}\gamma(z,v_L,v_U)\frac{\exp(-\kappa \gamma(z,v_L,v_U))}{\sum_{(v_L,v_U)\in \Xi_U(v)}\exp(-\kappa \gamma(z,v_L,v_U))}~,\label{eq:smoothmin}
\end{align}
where the smooth weighted averages on the right hand side of the equations above
conservatively approximate the maximum and  minimum  respectively, where the approximation errors are inversely proportional to $\kappa$   \citep{Chernozhukov:2012qy}. Suppose that the researcher chooses a fixed $\kappa>0$.
Define $\mathbf u(p;\kappa)\equiv\int \mathbf g_{p}(z,v;\kappa)p'l( z)dP(x)$, where $\mathbf g_p(z,v;\kappa)=\Gamma(\mathbf g_L(z,v;\kappa),\mathbf g_U(z,v;\kappa),p'l( z))$. $\mathbf u(p;\kappa)$ is then  a conservative approximation of the support function $\upsilon(p,\Theta_{0,v}(P))$ whose approximation bias can be bounded as follows:
\begin{align}
|\mathbf u(p;\kappa)-\upsilon(p,\Theta_{0,v}(P))|\le C  E[|l(Z)|^2]\kappa^{-2}~\text{ uniformly in }p\in\mathbb S^\ell~,
\end{align} 
where $C$ is a positive constant that depends on the cardinality of  the support of $(V_L,V_U)$.
Note that $\mathbf u(p;\kappa)$ depends smoothly  on the underlying distribution. This is because, as oppose to the maximum and minimum, the smooth weighted averages in \eqref{eq:smoothmax}-\eqref{eq:smoothmin} relate $\mathbf g_L,\mathbf g_U$ and $\gamma$ in a differentiable manner. This suggests that, although regular estimation of $\upsilon(p,\Theta_{0,v}(P))$ is not generally possible, it may be possibile to estimate  $\mathbf u(p;\kappa)$ in a regular manner, which we leave as   future work.
\end{remark}

\section{Estimation of Density Weighted Average Derivative in Mean Regression}\label{sec:est}
In this section, we illustrate efficient estimation by studying a well-known example. We focus on the  case where  $Y$ is interval-valued, and the parameter of interest is the density weighted average derivative of the mean regression function as in \cite{PowellStockStoker1989EJES}. That is, $\theta=E[f(Z)\nabla_zm(Z)]$, where $m$ is the conditional mean of $Y$. The density weight is  attractive for practical purposes as it does not require the choice of a trimming parameter to control for the stochastic denominator $\hat f_{i,h}$ defined below, which would otherwise arise upon estimating the ``score function'' $l$.\footnote{\cite{PowellStockStoker1989EJES} are interested in estimating index coefficients up to scale in a single index model.
The scale normalization implicit in the density weighted average derivative may not yield most easily interpretable estimates. For this reason, \cite{PowellStockStoker1989EJES} also suggest estimating the rescaled parameter $\tilde\theta\equiv \theta/E[f(Z)]$. We consider an estimator of the identified set of $\tilde\theta$ in Section \ref{sec:mc}.}
  
Theorem \ref{thm:thetabounds} and the law of iterated expectations imply that the support function of the identified set in this setting is given by
\begin{align}
	\upsilon(p,\Theta_0(P))=E[Y_pp'l(Z)],\label{eq:upspss}
\end{align}
where $Y_p=\Gamma(Y_L,Y_U,p'l(Z))$ and $l(z)=-2\nabla_z f(z)$. Our estimator of the support function is a natural extension of \cite{PowellStockStoker1989EJES} and replaces unknown objects in \eqref{eq:upspss} with nonparametric kernel estimators and expectations with sample averages.
Let $K:\mathbb R^\ell\to \mathbb R$ be a kernel function. For each $z\in\mathcal Z$, $i$ and bandwidth $h$, define the ``leave-one-out'' kernel density estimator by
\begin{align}
\hat f_{i,h}(z)\equiv\frac{1}{(n-1)h^\ell}\sum_{j=1,j\ne i}^n	K\left(\frac{z-Z_j}{h}\right).
\end{align}
 Our estimator of $l$ is then defined by $\hat l_{i, h}(z)\equiv -2\nabla_z \hat f_{i,h}(z)$.  The support function of $\Theta_0(P)$ is then estimated by
\begin{align}
\hat \upsilon_n(p)\equiv
\frac{1}{n}\sum_{i=1}^n p'\hat l_{i,h}(Z_{i})\hat Y_{p,i},\label{eq:shat}
\end{align}
where $\hat Y_{p,i}$ is an estimator of $Y_{p,i}$, which is not observed. For this, we let $\hat Y_{p,i}=\Gamma(Y_{L,i},Y_{U,i},p'\hat l_{i,\tilde h}),$ where $\tilde h$ is another bandwidth parameter.
Computing  the estimator in \eqref{eq:shat} only involves  kernel density estimation and taking averages. Hence, it can be implemented quite easily.
When the researcher is only interested in the average derivative  with respect to a particular variable, the required computation simplifies further. For example, suppose the parameter of interest is the average derivative $\theta^{(j)}$  with respect to the $j$-th variable.
An estimate of the upper bound on $\theta^{(j)}$ can be obtained by computing the support function in \eqref{eq:shat} only for one direction $p=\iota_j$, i.e., $\hat\theta_{U,n}^{(j)}=\hat\upsilon_n(\iota_j)$. The lower bound can be computed similarly with $p=-\iota_j$.

We now add regularity conditions required for efficient estimation of the support function.
Let $J\equiv(\ell+4)/2$ if $\ell$ is even and $J\equiv (\ell+3)/2$ if $\ell$ is odd.  
\begin{assumption}\label{as:onf_est}
(i)	There exists $M:\mathcal Z\to\mathbb R_+$ such that
\begin{align}
&	\|\nabla_z f(z+e)-\nabla_zf(z)\|<M(z)\|e\|\\
&	\|\nabla_z (f(z+e)\times m_j(z+e))-\nabla_z(f(z) \times m_j(z))\|<M(z)\|e\|,~~j=L,U,
\end{align}
and $E[|M(Z)|^2]<\infty$. 
(ii) 	 All partial derivatives of $f$ of order $J+1$ exist. $E[Y_L(\partial^k f/\partial z^{(j_1)}\cdots\partial z^{(j_k)})]$ and $E[Y_U(\partial^k f/\partial z^{(j_1)}\cdots\partial z^{(j_k)})]$ exist for all $k\le J+1$.
\end{assumption}

\begin{assumption}\label{as:onK}
(i) The support $\mathcal S_K$ of $K$ is a convex subset of $\mathbb R^\ell$ with nonempty interior with the origin as an interior point. Let $\partial \mathcal S_K$ be the boundary of $\mathcal S_K$; (ii) $K$ is a bounded, continuously differentiable function with bounded derivatives, and $\int K(u)du=1$, $\int uK(u)du=0$; (iii) $K(u)=0$ for all $u\in\partial\mathcal S_K$; (iv) $K(u)=K(-u)$ for all $u\in\mathcal S_K$. 	(v) 	$K(u)$ is of order $J$:
		\begin{align}
	\int u^{j_1}u^{j_2}\cdots u^{j_k}K(u)du&=0,~j_1+\cdots+j_k<J\\
	\int u^{j_1}u^{j_2}\cdots u^{j_k}K(u)du&\ne0,~j_1+\cdots+j_k=J.
		\end{align}
		(vi) The $J$ moments of $K$ exist.
\end{assumption}

Assumptions \ref{as:onf_est} and \ref{as:onK} are based on the assumptions in \cite{PowellStockStoker1989EJES}. Assumption \ref{as:onf_est} imposes  suitable smoothness conditions on $ f$, $m_L$ and $m_U$. Assumption \ref{as:onK} then gives standard regularity conditions on the kernel. A higher-order kernel is used to remove an asymptotic bias. With these additional assumptions, the next theorem establishes the asymptotic efficiency of the estimator.

\begin{theorem}\label{thm:estimator}
	Suppose the conditions of Theorem \ref{thm:InvCov} hold and Assumptions \ref{as:onf_est} and \ref{as:onK} hold. 
	Suppose further that $h\to0$, $nh^{\ell+2+\delta}\to \infty$ for some $\delta>0$, $n h^{2J}\to 0$, $\tilde h\to 0$, and $n\tilde h^{4(\ell+1)}\to \infty$ as $n\to\infty$.   Then, (a) $\{\hat \upsilon_n(\cdot)\}$ is a regular estimator for $\upsilon(\cdot,\Theta_0(P))$; (b) Uniformly in $p\in \mathbb S^{\ell}$:
	\begin{multline}\label{th:distdisp1}
	\sqrt{n}\{\hat  \upsilon_n(p) - \upsilon(p,\Theta_0(P))\} 	=\frac{2}{\sqrt n}\sum_{i=1}^n[f(Z_i)p'\nabla_zm_{p}(Z_i)-\upsilon(p,\Theta_0(P))]\\
-\frac{2}{\sqrt n}\sum_{i=1}^np'\nabla_z f(Z_i)(Y_{p,i}-m_{p}(Z_i))+o_p(1);
	\end{multline}
(c) As a process in $\mathcal C(\mathbb S^\ell)$,
	\begin{equation}\label{th:distdisp2}
	\sqrt{n}\{ \hat\upsilon_n(\cdot) -  \upsilon(\cdot,\Theta_0(P))\} \stackrel{L}{\rightarrow} \mathbb G_0 ~,
	\end{equation}
	where $\mathbb G_0$ is a tight mean zero Gaussian process on $\mathcal C(\mathbb S^{\ell})$ with $\emph{Cov}(\mathbb G_0(p_1),\mathbb G_0(p_2)) = I^{-1}(p_1,p_2)$.
\end{theorem}

\begin{remark}\rm\label{rem:setest}
Each extreme point $\theta^*(p)$ can be estimated by its sample analog estimator,
$\hat\theta_n(p)=\frac{1}{n}\sum_{i=1}^n\hat l_{i,n}(Z_{i})\hat Y_{p,i}$.
	Using this estimator, we may also define an estimator  of the identified set as follows:
	\begin{align}
	\hat\Theta_n\equiv\text{co}(\{\theta\in \Theta:\theta=\hat\theta_n(p),~p\in\mathbb S^\ell\}),
	\end{align}
	where $\text{co} (A)$ denotes the  convex hull of $A$.
  Theorem \ref{thm:estimator} has direct consequences on the  consistency and asymptotic optimality of this set estimator. For any two compact convex sets $A,B$, let the Hausdorff distance be $d_H(A,B)\equiv \max\{\sup_{a\in A}\inf_{b\in B}\|a-b\|,\sup_{b\in B}\inf_{a\in A}\|a-b\|\}.$
   Due to the equality of the Hausdorff distance between sets and the supremum distance between the corresponding support functions $d_H(A,B)=\sup_{p\in\mathbb S^\ell}|\upsilon(p,A)-\upsilon(p,B)|$ for any compact convex sets $A,B$, we have $d_H(\hat\Theta_n,\Theta_0(P))=O_p(n^{-1/2})$ by H\"{o}rmander's embedding theorem \citep{Li_Ogura_etal2002aBK}.
Further, \cite{KaidoSantos2011} show that if $\hat\Theta_n$ is associated with the efficient estimator of the support function and $L:\mathbb R_+\to\mathbb R_+$ is a subconvex continuous function,  it holds under regularity conditions that
	\begin{align}
\liminf_{n\to\infty}E[L(\sqrt nd_H(C_n,\Theta_0(P)))]\ge \limsup_{n\to\infty}E[L(\sqrt nd_H(\hat\Theta_n,\Theta_0(P)))]=E[L(\|\mathbb G_0\|_\infty)],
	\end{align}
for any regular convex compact valued set estimator $\{C_n\}$  for $\Theta_0(P)$.
That is, among regular set estimators, $\hat\Theta_n$ asymptotically minimizes a wide class of estimation risks based on the Hausdorff distance.
\end{remark}	

\begin{remark}\rm
Efficient estimators of the support function can also be used to conduct inference for $\Theta_0(P)$ and points inside it using the score-based weighted bootstrap of \cite{KlineSantos2012JEM}. Specifically, let $W_i$ be a mean zero random scalar with variance 1 and let $\{W_i\}$ be a sample independent of $\{X_i\}_{i=1}^n$. 
For each $p\in\mathbb S^\ell$, define the process:
\begin{align}
	G^*_n(p)\equiv \frac{1}{\sqrt n}\sum_{i=1}^nW_i\Big\{\sum_{j=1,j\ne i}\frac{-2}{(n-1)h^{\ell+1}}\nabla_zK\Big(\frac{Z_i-Z_j}{h}\Big)(\hat Y_{p,i}-\hat Y_{p,j})-\hat \upsilon_n(p)\Big\},
\end{align}
where the process $G^*_n$ is a $U$-process which is first-order asymptotically equivalent to the profess $\frac{1}{\sqrt n}\sum_{i=1}^nW_i\psi_p(X_i)$.\footnote{This can be shown following an argument similar to the one in Section 3.4 in \cite{PowellStockStoker1989EJES}. The proof is omitted for brevity.}
In practice, the distribution of $G^*_n$ can be simulated by generating random samples of $\{W_i\}$, which  weakly converges to $\mathbb G_0$ conditional on $\{X_i,i=1,\cdots,n\}$. This distribution, in turn, can be used to make inference. For example, 
a level $1-\alpha$ one-sided confidence set as in \cite{BeresteanuMolinari2008E}  can be constructed as $\mathcal C_{1n}\equiv \hat\Theta_n^{c^*_{1n}/\sqrt n}$, where $\hat\Theta_n^\epsilon\equiv\{\theta\in\Theta :\inf_{\theta'\in\hat\Theta_n}\|\theta-\theta'\|\le \epsilon\}$ and $c^*_{1n}$ is the $1-\alpha$ quantile of $\sup_{p\in\mathbb S^\ell}\{- G^*_n(p)\}_+$ \citep[see also][]{Kaido2010aWP,KaidoSantos2011}. 
\end{remark}

\begin{remark}\rm

One may also construct a confidence set for a particular coordinate $\theta^{(j)}$ of $\theta$ or its identified set $\Theta^{(j)}_0(P)$. For example, a symmetric confidence set for $\Theta^{(j)}_0(P)$ can be constructed as $$\mathcal C^{(j)}_{n}\equiv [\hat \theta^{(j)}_{L,n}-c^{(j)}_{n}/\sqrt n,\hat \theta^{(j)}_{U,n}+c^{(j)}_{n}/\sqrt n],$$ 
where $c^{(j)}_{n}$ is the $1-\alpha$ quantile of $\max_{p\in\{\iota_j,-\iota_j\}}\{- G^*_n(p)\}_+$. 
\end{remark}

\section{Simulation evidence}\label{sec:mc}
In this section, we examine the finite sample performance of an estimator of the support function through Monte Carlo experiments. 
Throughout, we let $Z_i \equiv (Z_{1,i}, Z_{2,i},Z_{3,i})^\prime$, where $Z_{1,i}=1$ is a constant, and $Z_{2,i}$ and $Z_{3,i}$ are continuously distributed. 
For $\beta= (1,1)^\prime$, we generate:
\begin{equation}\label{mc1}
Y_i = Z_i^\prime \beta + \epsilon_i  \hspace{0.3 in} i=1,\ldots,n  ~,
\end{equation}
where $\epsilon_i$ is a standard normal random variable independent of $Z_i$. We then generate $(Y_{L,i},Y_{U,i})$ as:
\begin{align}\label{mc2}
Y_{L,i} &=Y_i - c  -e_{2i} Z_{2i}^2-e_{3i} Z_{3i}^2   \nonumber\\
Y_{U,i} &=Y_i + c + e_{2i} Z_{2i}^2+ e_{3i} Z_{3i}^2 ~,
\end{align}
where $c > 0$ and $e_{2i}$ and $e_{3i}$ are independently uniformly distributed on $[0,0.2]$ independently of $(Y_i,Z_i)$. Here, $c$ is a design parameter that controls the diameter of the identified set. The identified sets under three different values of $c$ are plotted in Figure \ref{fig:identifiedset}.

We report estimation results for two different estimators of the support function. 
Since scale normalization implicit in $\theta$ may not allow a simple interpretation of estimation results,
 we follow \cite{PowellStockStoker1989EJES} and renormalize the weighted average derivative as follows:
\begin{align}
\tilde \theta\equiv E[f(Z)]^{-1}E[f(Z)\nabla_z m(Z)].	
\end{align}
Integrating by parts, it holds that $I_\ell E[f(Z)]=E[\nabla_z Zf(Z)]=E[Zl(Z)]$, where $I_\ell$ is the identity matrix of dimension $\ell.$
Thus, $\tilde\theta$ can be rewritten as $\tilde\theta=E[l(Z)Z]^{-1}E[l(Z)m(Z)]$. Our first estimator of the support function applies this renormalization to the sample counterpart and is defined by
\begin{align}
\hat \upsilon^{IV}_n\equiv p'(\frac{1}{n}\sum_{i=1}^n \hat l_{i,h}(Z_{i})Z_i)^{-1}	\frac{1}{n}\sum_{i=1}^n \hat l_{i,h}(Z_{i})Y_{p,i}~,\label{eq:sIV}
\end{align}
where $\hat l_{i,h}$ uses a Gaussian kernel.
This estimator may be interpreted as  the inner product between $p$ and a boundary point estimated by an instrumental variable (IV) estimator, which regresses $Y_{p,i}$ on $Z_i$ using $\hat l_{i,h}$ as an instrument. Our second estimator replaces the Gaussian kernel with a higher order kernel.\footnote{Detailed description of the construction of the higher-order kernels is in \cite{PowellStockStoker1989EJES} Appendix 2.}

Tables \ref{tab:mc_spec3}-\ref{tab:mc_spec4} report the average losses of these estimators, measured in the Hausdorff distance measures: $R_H\equiv E[d_H(\hat\Theta_n,\Theta_0(P))]$,  $R_{IH}\equiv E[\vec d_H(\hat\Theta_n,\Theta_0(P))]$ and $R_{OH}\equiv E[\vec d_H(\Theta_0(P),\hat\Theta_n)]$. We call them the \textit{Hausdorff risk},  \textit{inner Hausdorff risk} and \textit{outer Hausdorff risk} respectively. The directed Hausdorff distance $\vec d_H$ is defined by $\vec d_H(A,B)\equiv \sup_{a\in A}\inf_{b\in B}\|a-b\|$, which has the property that $\vec d_H(\Theta_0(P),\hat\Theta_n)=0$ if $\Theta_0(P)\subseteq\hat\Theta_n$ but takes a positive value otherwise. Hence, $R_{OH}$ penalizes $\hat\Theta_n$ when it is a  ``small set'' that does not cover $\Theta_0(P)$. On the other hand, $R_{IH}$ penalizes $\hat\Theta_n$ when it is  a ``large set'' that does not fit inside $\Theta_0(P).$ The Hausdorff risk $R_H$ then penalizes $\hat\Theta_n$ for both types of deviations from $\Theta_0(P).$

 Table  \ref{tab:mc_spec3} reports $R_H,R_{IH}, $ and $R_{OH}$ for the first estimator under different values of $c$, $h$, and $n$. Throughout simulations, we have set $h=\tilde h$ for simplicity.
One observation is that, for any value of $n$ when $c=0.5$ or 1, $R_{IH}$ is increasing in $h$, which suggests that a larger bandwidth (oversmoothing) may introduce an outward bias to the set estimator.  This is consistent with the outer hausdorff risk $R_{OH}$ being decreasing in $h$ when identified sets are relatively large ($c=0.5,1$). However, $R_{IH}$ is not increasing in $h$ when the identified set is small ($c=0.1$) suggesting that there may be different sources of risk that could affect $R_{IH}$ in this setting. For example, even if one uses a small bandwidth and the estimated set itself $\hat\Theta_n$ is small, its location may  still be biased so that it does not stay inside $\Theta_0(P)$. The Hausdorff risk $R_H$ takes both errors into account and seems to have a well-defined minimum as a function of the bandwidth. For example, when $c=1$ and $n=1,000$, the Hausdorff risk is minimized when the bandwidth is about 0.6.

Table \ref{tab:mc_spec4} reports results for the bias-corrected (second) estimator.  Again,
for $c=0.5$ and $1$, $R_{IH}$ is increasing in $h$, and $R_{OH}$ is decreasing in $h$, which  suggests an outward bias with oversmoothing, but this tendency is not clear when the identified region is relatively small ($c=0.1$). 
 We also note that the bias correction through the higher-order kernel  improves the lowest Hausdorff risk but not in a significant manner.
In sum, the simulation results show a tradeoff between the inner and outer Hausdorff risks. The optimal bandwidth  in terms of the Hausdorff risk seems to exist, which  makes these two risks roughly of the same order.

\section{Concluding remarks}\label{sec:conclusion}
 This paper studies the identification and estimation of  weighted average derivatives in the presence of interval censoring on either an outcome or on a covariate.
We show that the identified set of average derivatives is compact and convex under general assumptions and further show that it can be represented   by its support function.  This  representation is used to characterize the semiparametric efficiency bound for estimating the identified set when the outcome variable is interval-valued. 
For  mean regression with an interval censored outcome, we  construct a semiparametrically efficient set estimator.  

For practical purposes, an  important avenue for future research is to develop a theory of optimal bandwidth choice.
 The simulation results suggest   the  Hausdorff risks vary with the choice of bandwidth. It is an open question how to trade off different types of biases (inward, outward, and shift) and variance.  Another interesting direction for future research would be to study the  higher order properties of first-order efficient estimators, which would require an
 asymptotic expansion as in \cite{NishiyamaRobinson2000E} extended to the context of interval censoring.

\bibliography{SetEst}

\clearpage
\begin{center}
{\large {\sc Supplemental Appendix}}
\end{center}

{\small
In this supplemental appendix, we include the proofs of results stated in the main text. The contents of the supplemental appendix are organized as follows. Appendix A contains notations and definitions used throughout the appendix.  Appendix B contains the proof of Theorems \ref{thm:thetabounds} and \ref{thm:thetabounds2}. Appendix C contains the proof of Theorem \ref{thm:InvCov} and auxiliary lemmas. Appendix D contains the proof of Theorem \ref{thm:estimator}. Appendix E then reports the Monte Carlo results.
}

\vspace{0.1in}
\begin{center}
{\sc {Appendix A}}: Notation and Definitions
\end{center}

\renewcommand{\thedefinition}{A.\arabic{definition}}
\renewcommand{\theequation}{A.\arabic{equation}}
\renewcommand{\thelemma}{A.\arabic{lemma}}
\renewcommand{\thecorollary}{A.\arabic{corollary}}
\renewcommand{\thetheorem}{A.\arabic{theorem}}
\setcounter{lemma}{0}
\setcounter{theorem}{0}
\setcounter{corollary}{0}
\setcounter{equation}{0}
\setcounter{remark}{0}

{\footnotesize
Let $\Pi:\mathcal X\to\mathcal Z$ be the projection map pointwise defined by $x=(y_L,y_U,z)\mapsto z$. Let $\nu=\Pi_{\#}\mu$ be the pushforward measure of $\mu$ on $\mathcal Z$. We then denote the marginal density of $P$ with respect to $\nu$ by $\phi^2_0(z)$.
 By the disintegration theorem, there exists a family $\{\mu_z:z\in\mathcal Z\}$ of probability measures on $\mathcal X$. Throughout, we assume that $\mu_z$ is absolutely continuous with respect to some $\sigma$-finite measure $\lambda$  for all $z\in\mathcal Z$. We then denote the conditional density function of $P$ with respect to $\lambda$ by $v^2_0(y_L,y_U|z)$.

For any $1\le p\le\infty$, we let $\|\cdot\|_{L^p_\pi}$ be the usual $L^p$-norm with respect to a measure $\pi$, where $\|\cdot\|_{L^\infty_\pi}$ denotes the essential supremum.
}

\vspace{0.1in}
\begin{center}
{\sc {Appendix B}}: Proof of Theorems  \ref{thm:thetabounds} and \ref{thm:thetabounds2}
\end{center}

\renewcommand{\thedefinition}{B.\arabic{definition}}
\renewcommand{\theequation}{B.\arabic{equation}}
\renewcommand{\thelemma}{B.\arabic{lemma}}
\renewcommand{\thecorollary}{B.\arabic{corollary}}
\renewcommand{\thetheorem}{B.\arabic{theorem}}
\setcounter{lemma}{0}
\setcounter{theorem}{0}
\setcounter{corollary}{0}
\setcounter{equation}{0}
\setcounter{remark}{0}

{\footnotesize
\begin{proof}[\rm Proof of Theorem \ref{thm:thetabounds}]
We first show that the  identified set can be written as
\begin{align}
	\Theta_0(P)=\{\theta:\theta=E[m(Z)l(Z)],~P(m_L(Z)\le m(Z)\le m_U(Z))=1\}.\label{eq:idset}
\end{align}
For this, we  note that, by Assumptions \ref{as:onY}-\ref{as:onP1} and arguing as in the proof of
 Theorem 1 in \cite{Stoker1986EJES}, we have
\begin{equation}
	E[w(Z)\nabla_zm(Z)]=E[m(Z)l(Z)].\label{eq:intparts1}
\end{equation}
Further,  the distribution of $Y$ first-order stochastically dominates that of $Y_L$. Similarly, the distribution of $Y_U$ first-order stochastically dominates that of $Y$. Since $q$ is nondecreasing by the convexity of $\varrho$, it then follows that, for each $u\in\mathbb R$,
	\begin{align}
E[q(Y_L-u)|Z]\le E[q(Y-u)|Z] \le E[q(Y_U-u)|Z],~P-a.s.\label{eq:focs}
	\end{align}
Eq. \eqref{eq:bounds} then follows by \eqref{eq:focs}, Assumption \ref{as:onP1} (iii), and the hypothesis that $E[q(Y-u)|Z=z]=0$ has a unique solution at $u=m(z)$ on $D$.

For the  convexity of $\Theta_0(P)$, observe that for any $\theta_1,\theta_2\in\Theta_0(P)$, there exist $m_1,m_2:\mathcal Z\to \mathbb R$ such that $\theta_j=E[m_j(Z)l(Z)]$ and $m_L(Z)\le m_j(Z)\le m_U(Z), P-a.s.$ for $j=1,2.$ Let $\alpha\in [0,1]$ and let $\theta_\alpha\equiv \alpha\theta_1+(1-\alpha)\theta_2$. Then,
\begin{align*}
	\theta_\alpha=E[m_\alpha(Z)l(Z)],
\end{align*}
where $m_\alpha\equiv \alpha m_1+(1-\alpha)m_2$. Since $m_L(Z)\le m_\alpha(Z)\le m_U(Z), P-a.s.$, it follows that $\theta_\alpha\in\Theta_0(P)$.  Therefore, $\Theta_0(P)$ is convex.

	We  show compactness of $\Theta_0(P)$ by showing $\Theta_0(P)$ is bounded and closed. By Assumption  \ref{as:onP1} (i)-(ii), for any $\theta\in\Theta_0(P)$,
	\begin{align}
|\theta^{(j)}|\le \sup_{z\in \mathcal Z}|m(z)|E[|l^{(j)}(Z)|]\le \sup_{x\in D}|x|E[|l^{(j)}(Z)|]  <\infty,~\text{ for }j=1,\cdots, \ell.\label{eq:thetabounds1}
	\end{align}
 Hence, $\Theta_0(P)$ is bounded. To see that	 $\Theta_0(P)$ is closed, consider  the following maximization problem:
	\begin{align}
	\text{maximize}&~~~ E[m(Z)p'l(Z)],\label{eq:supportfunc1}\\
	\text{s.t.}&~~~m_L(Z)\le m(Z)\le m_U(Z), P-a.s.	
	\end{align}
Arguing as in the proof of Proposition 2 in \cite{BontempsMagnacMaurin2012E}, the objective function is maximized by setting $m(z)=m_L(z)$ when $p'l(z)\le 0$ and setting $m(z)=m_U(z)$ otherwise. This and \eqref{eq:intparts1} give the support function of $\Theta_0(P)$ in \eqref{eq:suppfunc} and also shows that, for each $p\in\mathbb S$, there exists $m_p(z)\equiv 1\{p'l(z)\le 0\}m_L(z)+1\{p'l(z)>0\}m_U(z)$ such that $\upsilon(p,\Theta_0(P))=\langle p,\theta^*(p)\rangle$, where $\theta^*(p)=E[m_p(Z)l(Z)]$. 
	Since $m_p$ satisfies $m_L(Z)\le m_p(Z)\le m_U(Z),P-a.s.,$ we have $\theta^*(p)\in \Theta_0(P)$.
	By Proposition 8.29 (a) in \cite{RockafellarWets2005}, the boundary of $\Theta_0(P)$ is $\{\tilde\theta:\langle p,\tilde\theta\rangle=\upsilon(p,\Theta_0(P)),p\in\mathbb S^\ell\}$.
	Therefore, $\Theta_0(P)$ contains its boundary, and hence it is closed. 

For the strict convexity of $\Theta_0(P)$, we show it through the differentiability of the support function. The proof is similar to that of Lemma A.8 in \cite{BeresteanuMolinari2008E} and Lemma 23 in \cite{BontempsMagnacMaurin2012E}. To this end, we extend the support function and define  $s(p,\Theta_0(P))$ as in \eqref{eq:suppfunc} for each $p\in\mathbb R^\ell\setminus\{0\}$.

For each $z\in\mathcal Z$, let $\xi(z)\equiv (m_L(z)-m_U(z))l(z)$. For each $p\in\mathbb R^\ell\setminus\{0\}$, let
$\zeta(p)\equiv E[1\{p'\xi(Z)\ge 0\}p'\xi(Z)]$. Then, since $m_L(Z)-m_U(Z)\le 0$ almost surely, it holds that $\upsilon(p,\Theta_0(P))=\zeta(p)+E[m_U(Z)p'l(Z)]$ for all $p\in\mathbb R^\ell\setminus\{0\}$. For any $p,q\in\mathbb R^\ell\setminus\{0\}$, it then follows by the Cauchy-Shwarz inequality that
\begin{multline}
|\zeta(q)-\zeta(p)-(q-p)'E[\xi(Z)1\{p'\xi(Z)\ge 0\}]|	\\=|E[(1\{q'\xi(Z)\ge 0\}-1\{p'\xi(Z)\ge 0\})q'\xi(Z)]|
\le \|1\{q'\xi\ge 0\}-1\{p'\xi\ge 0\}\|_{L^2_{P}}\|q'\xi\|_{L^2_{P}}.\label{eq:supportfunc3}
\end{multline}
By Assumptions \ref{as:onY} (i), the distribution of $\xi(Z)$ does not assign a  positive measure to any proper subspace of $\mathbb R^\ell$ with dimension $\ell-1$, which ensures $P(p'\xi(Z)=0)=0$. Thus,
for any sequence $\{q_n\}$ such that $q_n\to p$, it follows that $1\{q_n'\xi(Z)\ge 0\}\stackrel{a.s.}{\to}1\{p'\xi(Z)\ge 0\}$ as $n\to\infty$.   Note that $1\{p'\xi(Z)\}$ is bounded for all $p$. Thus, the function class  $\{1^2\{p'\xi(\cdot)\}:p\in\mathbb R^\ell\setminus\{0\}\}$ is uniformly integrable.  These results ensure that $ \|1\{q'\xi\ge 0\}-1\{p'\xi\ge 0\}\|_{L^2_{P}}\to 0$ as $q\to p$. This and $\|q'\xi\|_{L^2_{P}}<\infty$ imply that the right hand side of \eqref{eq:supportfunc3} is $o(1)$.
Hence, $\zeta$ is differentiable at every point on $\mathbb R^\ell\setminus\{0\}$ with the derivative $E[\xi(Z)1\{p'\xi(Z)\ge 0\}]$. Note that, for each $z\in\mathcal Z$, $p\mapsto m_U(z)p'l(z)$ is differentiable with respect to $p$ and $m_Ul$ is integrable with respect to $P$ by Assumption \ref{as:onP1} (i)-(ii). This ensures that $p\mapsto E[m_U(Z)p'l(Z)]$ is differentiable with respect to $p$ at every $p\in\mathbb R\setminus\{0\}$. Therefore, the map $p\mapsto \upsilon(p,\Theta_0(P))$ is differentiable for all $p\in\mathbb R^\ell\setminus\{0\}$.
By Corollary 1.7.3 in \cite{Schneider1993}, the support set $H(p,\Theta_0(P))\equiv\{\theta:\langle p,\theta\rangle=\upsilon(p,\Theta_0(P))\}\cap \Theta_0(P)$ for each $p$ then contains only one point, which ensures the strict convexity of $\Theta_0(P)$. 

Too see that $\Theta_0(P)$ is sharp, take any $\theta\in\Theta_0(P)$. Then, by convexity, there exist $p,q\in\mathbb S^\ell$ and $\alpha\in[0,1]$ such that $\theta=\alpha\theta^*(p)+(1-\alpha)\theta^*(q)$, which further implies
\begin{align}
\theta=E[(\alpha m_p(Z)+(1-\alpha)m_q(Z))l(Z)]=E[w(Z)\nabla_z m_{\alpha,p,q}(Z)],
\end{align}
where $m_{\alpha,p,q}\equiv\alpha m_p+(1-\alpha)m_q$, and the last equality follows from integration by parts and Assumptions \ref{as:onY} (i) and \ref{as:onP1} (iv) ensuring the almost everywhere differentiability of $m_{\alpha,p,q}$. Since $m_{\alpha,p,q}$  satisfies \eqref{eq:bounds} in place of $m$ with $m_{\alpha,p,q}$ and $m_{\alpha,p,q}$ is almost everywhere differentiable, $\theta$ is the weighted average derivative of a regression function consistent with some data generating process. Hence, $\Theta_0(P)$ is sharp.
\end{proof}

\begin{proof}[\rm Proof of Theorem \ref{thm:thetabounds2}]
		We first show \eqref{eq:intbounds}. 
	By the first order condition for \eqref{eq:intcovopt}  and Assumption \ref{as:mon} (ii),  $E[q(Y-g(Z,V))|Z,V]=E[q(Y-g(Z,V))|\tilde Z,V]=0,$ $P-a.s.$ 
	By Assumption \ref{as:mon} (i) and the monotonicity of $q$,  for $v_L\le v\le v_U$, we have
	\begin{align}
	\int q(y-g(z, v_U))dP(y|\tilde z,v)\le 0\le \int q(y-g(z, v_L))dP(y|\tilde z,v), ~P-a.s.	\label{eq:intbounds2}	
	\end{align}
	Taking expectations with respect to $V$, we obtain
	\begin{align}
	\int q(y-g(z, v_U))dP(y|\tilde z)\le 0\le \int q(y-g(z, v_L))dP(y|\tilde z),~P-a.s.	\label{eq:intbounds2a}	
	\end{align}	
	Further,  by Assumption \ref{as:onmC} (ii), 
	\begin{align}
\int q(y-\gamma(z,v_L,v_U))dP(y|\tilde z)=0.\label{eq:intbounds1}
	\end{align}
	By  \eqref{eq:intbounds2a}-\eqref{eq:intbounds1} and the monotonicity of $q$, we then have
	\begin{align}
	g(z,v_L)\le\gamma(z,v_L,v_U)\le g(z,v_U).\label{eq:intbounds3}
	\end{align} 
	Let $\Xi_L(v)\equiv\{(v_L,v_U):v_L\le v_U\le v\}$ and $\Xi_U(v)\equiv\{(v_L,v_U):v\le v_L\le v_U\}$.
	To prove the lower bound on $g(z,v),$ take any $v_U\le v$.
	Then by Assumption \ref{as:mon} (i) and \eqref{eq:intbounds3}, we have $\gamma(z,v_L,v_U)\le g(z,v)$ for any $(v_L,v_U)\in\Xi(v)$.
Hence, it follows that	  $g_L(z,v)\equiv\sup_{(v_L,v_U)\in \Xi_L(v)}\gamma(z,v_L,v_U)\le g(z,v)$. Note that $g_L(z,v)$ is weakly increasing in $v$ by construction and differentiable in $z$ with a bounded derivative by Assumption \ref{as:onmC} (iii). Hence, it is consistent with Assumption \ref{as:mon} (i). Thus, the bound is sharp. A similar argument gives the upper bound. Hence, \eqref{eq:intbounds} holds. 
This and integration by parts imply that the sharp identified set can be written as
\begin{align}
		\Theta_{0,v}(P)=\{\theta:\theta=E[g(Z,v)l(Z)],~P(g_L(Z,v)\le g(Z,v)\le g_U(Z,v))=1\}.\label{eq:idsetv}
\end{align}	
The rest of the proof is then similar to that of Theorem \ref{thm:thetabounds}. It is therefore omitted.
\end{proof}
}

\vspace{0.1in}
\begin{center}
{\sc {Appendix C}}: Proof of Theorem \ref{thm:InvCov}
\end{center}

\renewcommand{\thedefinition}{C.\arabic{definition}}
\renewcommand{\theequation}{C.\arabic{equation}}
\renewcommand{\thelemma}{C.\arabic{lemma}}
\renewcommand{\thecorollary}{C.\arabic{corollary}}
\renewcommand{\thetheorem}{C.\arabic{theorem}}
\setcounter{lemma}{0}
\setcounter{theorem}{0}
\setcounter{corollary}{0}
\setcounter{equation}{0}
\setcounter{remark}{0}

{\footnotesize
This Appendix contains the proof of Theorem \ref{thm:InvCov} and auxiliary lemmas needed to establish the main result.

To characterize the efficiency bound, it proves useful to study a parametric submodel of $\mathbf P$. We define a parametric submodel through a curve in $L^2_\mu$.
Let $h_0\equiv \sqrt{dP/d\mu}.$ 
Let $\tilde v:\mathcal X\to\mathbb R$ and $\tilde \phi:\mathcal Z\to \mathbb R$ be bounded functions that are continuously differentiable in $z$ with bounded derivatives.
We then define
\begin{align}
	\bar v(x)\equiv\tilde v(x)-E[\tilde v(X)|Z=z], ~~~\text{ and }~~~\bar\phi(z)\equiv\tilde\phi(z)-E[\tilde\phi(Z)],
\end{align}
where expectations are with respect to $P$. For each $\eta\in\mathbb R$, define $v_\eta:\mathcal X\to\mathbb R$ and $\phi_\eta:\mathcal Z\to\mathbb R$ by
	\begin{align}
v_\eta^2(y_L,y_U|z)=v_0^2(y_L,y_U|z)(1+2\eta\bar v(x)),~~~\text{ and }~~~\phi_\eta^2(z)=\phi_0^2(z)(1+2\eta\bar\phi(z)).\label{eq:tangent3}
	\end{align}
We then let $h^2_\eta$ be defined pointwise by
\begin{align}
	h^2_\eta(x)\equiv v_\eta^2(y_L,y_U|z)\phi_\eta^2(z)\label{eq:heta}.
\end{align}
	It is straightforward to show that $\eta\mapsto h^2_\eta$ is a curve in $L^2_{\mu}$ with $\dot h_0=\dot v_0\phi_0+v_0\dot\phi_0$, where $\dot v_0(y_L,y_U|z)\equiv\bar v(x) v_0(y_L,y_U|z)$ and $\dot\phi_0(z)=\bar\phi(z)\phi_0(z)$. 
We also note that for any $\eta$ and $\eta_0$ in a neighborhood of 0, it holds that
		\begin{align}
	v_\eta^2(y_L,y_U|z)=v_{\eta_0}^2(y_L,y_U|z)(1+2(\eta-\eta_0)\bar v_{\eta_0}(x)),~~~\text{ and }~~~\phi_\eta^2(z)=\phi_{\eta_0}^2(z)(1+2\eta\bar\phi_{\eta_0}(z)).\label{eq:expectation}
		\end{align}
	where $\bar v_{\eta_0}=\bar vv^2_0/v^2_{\eta_0}$ and $\bar\phi_{\eta_0}=\bar \phi \phi^2_0/\phi^2_{\eta_0}.$
We then define	$\dot v_{\eta_0}(y_L,y_U|z)=\bar v_{\eta_0}(x)v_{\eta_0}(y_L,y_U|z)$ and $\dot\phi_{\eta_0}(z)=\bar\phi_{\eta_0}(z)\phi_{\eta_0}(z)$.
	
We further introduce notation for population objects along this curve.	
 Let $f_\eta(z)\equiv \phi_\eta^2(z)$ and $l_\eta\equiv-\nabla_zw(z)-w(z)\nabla_zf_{\eta}(z)/f_\eta(z)$.	Lemma \ref{lem:onm} will show that 
there exists a neighborhood $N$ of 0 such that  the equations $\int q(y_L-\tilde m)v^2_\eta(y_L,y_U|z)d\lambda(y_L,y_U)=0$ and $\int q(y_U-\tilde m)v^2_\eta(y_L,y_U|z)d\lambda(y_L,y_U)=0$ have unique solutions  on $D$ for all $\eta\in N$. We denote these solutions by $m_{L,\eta}$ and $m_{U,\eta}$ respectively. We then let $m_{p,\eta}(z)\equiv \Gamma(m_{L,\eta}(z),m_{U,\eta}(z),p'l_\eta(z))$. Further, we define
\begin{align}
	r_{j,\eta}(z)&\equiv-\frac{d}{d\tilde m}E_\eta\left[q(y_j-\tilde m)|Z=z\right]\big|_{\tilde m=m_{j,\eta}(z)},~j=L,U,	\label{eq:reta}
\end{align}
where expectations are taken with respect to $P_\eta$.
Finally, define
\begin{align}
\zeta_{p,\eta}\equiv\Gamma(r_{L,\eta}^{-1}(z) q(y_L-m_{L,\eta}(z)), r_{U,\eta}^{-1}(z) q(y_U-m_{U,\eta}(z)),p'l_\eta(z)).\label{eq:zetaeta}	
\end{align}

Given these definitions, we give an outline of the general structure of the proof. The proof of Theorem \ref{thm:InvCov} proceeds by verifying the conditions of Theorem 5.2.1 in \cite{BickelKlassenRitov1993}, which requires (i) the characterization of the tangent space at $P$, which we accomplish in Theorem \ref{thm:tangent} and (ii) the pathwise weak differentiability of the map $Q\mapsto\upsilon(\cdot,\Theta_0(Q))$, which is established by Theorem \ref{thm:pathdiff}. 

\noindent {\sc Tangent Space} (Theorem \ref{thm:tangent})

{\it Step 1:} Lemmas \ref{lem:onm}-\ref{lem:onmLmU} show that for some neighborhood $N$ of 0, Assumptions \ref{as:onP1} and \ref{as:onP2} hold with $P_\eta$ in place of $P$ for all $\eta\in N$, where $\sqrt{dP_\eta/d\mu}=h_\eta$ defined in \eqref{eq:tangent3}-\eqref{eq:heta}. This means that the restrictions on $P$ in Assumptions \ref{as:onP1} and \ref{as:onP2} do not restrict the neighborhood in such a way that affects the tangent space derived in the next step. In this step, we exploit the fact that $\mathcal Z$ is determined by the dominating measure $\mu$ instead of each distribution in the model. 

{\it Step 2:} Theorem \ref{thm:tangent} then establishes that the tangent space $\dot {\mathbf S}$ equals $\mathbf T\equiv\{h\in L^2_\mu:\int h(x)s(x)d\mu(x)=0\}$ by showing that (i) $\dot{\mathbf S}\subseteq\mathbf T$ generally and (ii) due to Step 1,  $\{P_\eta,\eta\in N\}$ is a regular parametric submodel of $\mathbf P$ whose tangent set $\dot{\mathbf U}\subset\dot{\mathbf S}$ is dense in $\mathbf T$ implying $\mathbf T\subseteq \dot{\mathbf S}$.   

\noindent {\sc Differentiability} (Theorem \ref{thm:pathdiff})

{\it Step 1:} Lemmas \ref{lem:onm} and \ref{lem:dmlu} explicitly characterize the pathwise derivatives of $m_{j,\eta},j=L,U$ along the curve $\eta\mapsto h_\eta$ defined in \eqref{eq:tangent3}-\eqref{eq:heta}.

{\it Step 2:} Based on Step 1 and Lemma \ref{lem:intermediate}, Lemma \ref{lem:pathdiff_ptws} then characterizes the pathwise derivative of the support function $\upsilon(p,\Theta_0(P_\eta))$ at a point $p$ along the curve $\eta\mapsto h_\eta$ defined in \eqref{eq:tangent3}-\eqref{eq:heta}. Lemmas \ref{lem:bdd} and \ref{lem:cts} further show that this pathwise derivative  is uniformly bounded and continuous in $(p,\eta)\in\mathbb S^\ell\times N$.

{\it Step 3:} Based on Step 2, Theorem \ref{thm:pathdiff} first characterizes the  pathwise weak derivative of $\rho(P_\eta)=\upsilon(\cdot,\Theta_0(P_\eta))$ on the tangent space of the curve $\eta\mapsto h_\eta$ and further extends it to $\dot{\mathbf S}$.

\begin{lemma}\label{lem:onm}
Let $\eta\mapsto h_\eta$ be a curve in $L^2_\mu$ defined in \eqref{eq:tangent3}-\eqref{eq:heta}. Suppose Assumption \ref{as:onY} holds. Suppose $P\in\mathbf P$. Then, there exists a neighborhood $N$ of 0 such that (i) $\int q(y_j-\tilde m)v^2_\eta(y_L,y_U|z)d\lambda(y_L,y_U)=0$  has a unique solution at $\tilde m=m_{j,\eta}(z)$  on $D$ for $j=L,U$ and for all $\eta\in N$;	(ii) For each $(z,\eta)\in \mathcal Z\times N$, $m_{\eta,L}$ and $ m_{\eta,U}$ are continuously differentiable $a.e.$ on the interior of $\mathcal Z\times N$ with bounded derivative. In particular, it holds that 
\begin{align}
		\frac{\partial}{\partial \eta}m_{L,\eta}(z)\Big|_{\eta=\eta_0}&=2r_{L,\eta_0}^{-1}(z)\int q(y_L-m_{L,\eta_0}(z))\dot v_{\eta_0}(y_L,y_U|z)v_{\eta_0}(y_L,y_U|z)d\lambda(y_L,y_U)\\
		\frac{\partial}{\partial \eta}m_{U,\eta}(z)\Big|_{\eta=\eta_0}&=2r_{U,\eta_0}^{-1}(z)\int q(y_U-m_{U,\eta_0}(z))\dot v_{\eta_0}(y_L,y_U|z)v_{\eta_0}(y_L,y_U|z)d\lambda(y_L,y_U),
\end{align}
for all $\eta_0\in N$.
\end{lemma}

\begin{proof}[\rm Proof of Lemma \ref{lem:onm}]
	The proof builds on the proof of Theorem 3.1 in \cite{NeweyStoker1993EJES}.
By Eq. \eqref{eq:tangent3}, it follows that
	\begin{multline}
\int q(y_L-\tilde m)v^2_\eta(y_L,y_U|z)d\lambda(y_L,y_U)=\int q(y_L-\tilde m)v^2_0(y_L,y_U|z)d\lambda(y_L,y_U)\\
+2\eta\int q(y_L-\tilde m)\bar v(x)v^2_0(y_L,y_U|z)d\lambda(y_L,y_U).\label{eq:tangent4a}
	\end{multline}
Since $P\in\mathbf P$, Assumption \ref{as:onP2} and Lemma C.2 in \cite{Newey1991NSMES} imply that $\int \tilde v(x)v^2_0(y_L,y_U|z)d\lambda(y_L,y_U)$ is continuously differentiable with bounded derivatives, and hence so is $\bar v.$ Therefore, $\int q(y_j-\tilde m)\bar v(x)v^2_0(y_L,y_U|z)d\lambda(y_L,y_U)$ is continuously differentiable in $(z,\tilde m)$ on $\mathcal Z\times D$ for $j=L,U$. Hence, by  \eqref{eq:tangent4a}, there is a neighborhood $N'$ of 0 such that the map 
		 $(z,\tilde m,\eta)\mapsto \int q(y_L-\tilde m)v^2_\eta(y_L,y_U|z)d\lambda(y_L,y_U)$ is continuously differentiable on $\mathcal Z\times D\times N'$  with bounded derivatives. By continuity, we may take a neighborhood $N$ of 0 small enough so that $\int q(y_L-\tilde m)v^2_\eta(y_L,y_U|z)d\lambda(y_L,y_U)=0$ admits a unique solution $m_{L,\eta}(z)$ for all $\eta\in N$.	
	 A similar argument can be made for  $m_{U,\eta}$. 
	
	By the implicit function theorem, there is a neighborhood of $(z,0)$ on which $\nabla_z m_{j,\eta_0}$ and $\frac{\partial}{\partial \eta}m_{j,\eta}(z)|_{\eta=\eta_0}$ exist and are continuous in their arguments on that neighborhood for $j=L,U$. By the compactness of $\mathcal Z$, $N$ can be chosen small enough so that $\nabla_zm_{\eta,L}$  and $\frac{\partial}{\partial \eta}m_{j,\eta}(z)|_{\eta=\eta_0}$ are continuous and bounded on $\mathcal Z\times N$ and
	\begin{align}
		\frac{\partial}{\partial \eta}m_{j,\eta}(z)\Big|_{\eta=\eta_0}=2r_{j,\eta_0}^{-1}(z)\int q(y_j-m_{j,\eta_0}(z))\dot v_{\eta_0}(y_L,y_U|z)v_{\eta_0}(y_L,y_U|z)d\lambda(y_L,y_U),~ j=L,U.\label{eq:dml}
	\end{align}
This completes the proof of the lemma.
\end{proof}

\begin{lemma}\label{lem:onw}
Let $\eta\mapsto h_\eta$ be a curve in $L^2_\mu$ defined in \eqref{eq:tangent3}-\eqref{eq:heta}. Suppose Assumption \ref{as:onY} (i) holds. Suppose further that $P\in\mathbf P.$ Then, there exists a neighborhood $N$ of 0 such that the conditional support of $(Y_L,Y_U)$ given $Z$  is in $D^o\times D^o$, $w(z)f_\eta(z)=0$ on $\partial \mathcal Z$, $P_\eta-a.s.,$ $\nabla_z f_\eta/f_\eta(z)$ is continuous $a.e.$, and $\int \|l_\eta(z)\|^2\phi^2_\eta(z)d\nu(z)<\infty$ for all $\eta\in N$.
\end{lemma}

\begin{proof}[\rm Proof of Lemma \ref{lem:onw}]
By \eqref{eq:tangent3},   $\{(y_L,y_U):v^2_0(y_L,y_U|z)=0\}\subseteq \{(y_L,y_U):v^2_\eta(y_L,y_U|z)=0\}$ for all $z\in\mathcal Z$ and $\eta\in\mathbb R$. This  implies $\{(y_L,y_U):v^2_\eta(y_L,y_U|z)>0\}\subseteq \{(y_L,y_U):v^2_0(y_L,y_U|z)>0\}\subset D^o\times D^o$ for all $z\in\mathcal Z$ and $\eta\in\mathbb R$, where the last inclusion holds by Assumption \ref{as:onP1}. This establishes the first claim.
Similarly, the second claim follows immediately from Eq. \eqref{eq:tangent3} and Assumption \ref{as:onP1} (ii). 

For the third claim, using Eq. \eqref{eq:tangent3}, we write
		\begin{align}
	\frac{\nabla_z f_\eta(z)}{f_\eta(z)}=\frac{\nabla_z f(z)}{f(z)}+\frac{2\eta\nabla_z\bar\phi(z)}{1+2\eta\bar\phi(z)}.\label{eq:tangent5}
		\end{align}
By Assumption \ref{as:onP1} (ii), \eqref{eq:tangent5}, and $\bar\phi$ being bounded and continuously differentiable in $z$, 
 $(\eta,z)\mapsto\nabla_z f_\eta(z)/f_\eta(z)$ is continuous. 
This and Assumption \ref{as:onw} in turn imply that the map $(\eta,z)\mapsto \|l_\eta(z)\|^2$ is continuous. Hence, by Assumption \ref{as:onY} (i), it achieves a finite maximum on   $N\times\mathcal Z$ for some $N$ small enough. Therefore, $\int\|l_\eta(z)\|\phi^2_\eta(z)d\nu(z)<\infty$ for all $\eta\in N$.
\end{proof}

\begin{lemma}\label{lem:ons}
Let $\eta\mapsto h_\eta$ be a curve in $L^2_\mu$ defined in \eqref{eq:tangent3}-\eqref{eq:heta}. Suppose further that $P\in\mathbf P$. Then, there exists a neighborhood $N$ of 0 such that $|r_{L,\eta}(z)|>\bar\epsilon$ and $|r_{U,\eta}(z)|>\bar\epsilon$, for all $z\in\mathcal Z$ and $\eta\in N$.
\end{lemma}
\begin{proof}[\rm Proof of Lemma \ref{lem:ons}]
By \eqref{eq:tangent3} and \eqref{eq:reta}, $r_{L,\eta}$ can be written as
	\begin{align}
	r_{L,\eta}(z)&\equiv-\frac{d}{d\tilde m}E_\eta\left[q(y_L-\tilde m)|Z=z\right]\big|_{\tilde m=m_{L,\eta}(z)}\nonumber\\
	&=	-\frac{d}{d\tilde m}\Big(E\left[q(y_L-\tilde m)|Z=z\right]+2\eta\int q(y_L-\tilde m)\bar v(x)v^2_0(y_L,y_U|z)d\lambda(y_L,y_U)\Big)\big|_{\tilde m=m_{L,\eta}(z)}\nonumber\\
	&=r_L(z)-2\eta \frac{d}{d\tilde m}\int q(y_L-\tilde m)\bar v(x)v^2_0(y_L,y_U|z)d\lambda(y_L,y_U)\big|_{\tilde m=m_{L,\eta}(z)}.\label{eq:tangent7}
	\end{align}
Since $\bar v$ is bounded and continuously differentiable, the second term on the right hand side of  \eqref{eq:tangent7} is well-defined and is bounded because Assumption \ref{as:onP2} (iii) holds for $P\in\mathbf P$. By Assumption \ref{as:onP2} (i) and Eq. \eqref{eq:tangent7}, we may take a neighborhood $N$ of 0 small enough so that $|r_{L,\eta}(z)|>\bar\epsilon$ for all $\eta\in N$ and $z\in\mathcal Z$. A similar argument can be made for $r_{U,\eta}(z)$. Thus, the claim of the lemma follows.
\end{proof}

\begin{lemma}\label{lem:onv}
Let $\eta\mapsto h_\eta$ be a curve in $L^2_\mu$ defined in \eqref{eq:tangent3}-\eqref{eq:heta}. Suppose further that $P\in\mathbf P$. Then, there exists a neighborhood $N$ of 0 such that (i) for any $\varphi:\mathcal X\to\mathbb R$ that is bounded and continuously differentiable in $z$ with bounded derivatives, $\int\varphi(x) v_\eta^2(y_L,y_U|z)d\lambda(y_L,y_U)$ is continuously differentiable in $z$ on $\mathcal Z$ with bounded derivatives; (ii) $\int q(y_L-\tilde m) \varphi(x)v_\eta^2(y_L,y_U|z)d\lambda(y_L,y_U)$ and $\int q(y_U-\tilde m)\varphi(x) v_\eta^2(y_L,y_U|z)d\lambda(y_L,y_U)$ are continuously differentiable in $(z,\tilde m)$ on $\mathcal Z\times D$ with bounded derivatives for all $\eta\in N$.
\end{lemma}

\begin{proof}[\rm Proof of Lemma \ref{lem:onv}]	
	Let $\varphi:\mathcal X\to\mathbb R$ be bounded and continuously differentiable in $z$ with bounded derivatives. By \eqref{eq:tangent3}, we may write
	\begin{equation}
	\int \varphi(x)v^2_\eta(y_L,y_U)d\lambda(y_L,y_U)
	=\int\varphi(x)v^2_0(y_L,y_U|z)d\lambda(y_L,y_U)
	+2\eta\int\varphi(x)\bar v(x)v^2_0(y_L,y_U|z)d\lambda(y_L,y_U)
	\end{equation}
	Note that $\bar v$ is bounded and continuously differentiable in $z$ with bounded derivatives. Thus, 
	 by Assumption \ref{as:onP2} (ii), $\int\varphi(x)v^2_\eta(y_L,y_U)d\lambda(y_L,y_U)$ is bounded and continuously differentiable in $z$ with bounded derivatives. Similarly, we may write
	\begin{multline}
	\int q(y_L-\tilde m)\varphi(x)v^2_\eta(y_L,y_U|z)d\lambda(y_L,y_U)\\
	=\int q(y_L-\tilde m)\varphi(x)v^2_0(y_L,y_U|z)d\lambda(y_L,y_U)+2\eta\int q(y_L-\tilde m)\varphi(x)\bar v(x)v^2_0(y_L,y_U|z)d\lambda(y_L,y_U).\label{eq:onv1}
	\end{multline}
	By Assumption \ref{as:onP2} (iii), $\int q(y_L-\tilde m)\varphi(x)v^2_0(y_L,y_U)d\lambda(y_L,y_U)$ is bounded and continuously differentiable in $z$ with bounded derivatives. Further, since $\bar v$ is bounded and continuously differentiable in $z$ with bounded derivatives, again by Assumption \ref{as:onP2} (iii), the same is true for the second term in the right hand side of \eqref{eq:onv1}.
	The argument for  $\int q(y_U-\tilde m)\varphi(x)v^2_\eta(y_L,y_U)d\lambda(y_L,y_U)$ is similar. Thus the claim of the lemma follows.
\end{proof}

\begin{lemma}\label{lem:onmLmU}
Let $\eta\mapsto h_\eta$ be a curve in $L^2_\mu$ defined in \eqref{eq:tangent3}-\eqref{eq:heta}. Suppose that $P\in\mathbf P$. Then, there exists a neighborhood $N$ of 0 such that	$m_{L,\eta}$ and $ m_{U,\eta}$ are continuously differentiable $a.e.$ on $\mathcal Z$ with bounded derivatives. Further,  the maps $(z,\eta)\mapsto\nabla_z m_{L,\eta}(z)$ and $(z,\eta)\mapsto\nabla_z m_{U,\eta}(z)$ are continuous on $\mathcal Z\times N$.
\end{lemma}
\begin{proof}[\rm Proof of Lemma \ref{lem:onmLmU}]
	We show the claims of the lemma for $m_{L,\eta}$.
By $P\in\mathbf P$,	Assumption \ref{as:onP2} (iii) holds, which implies that the maps $(z,\tilde m)\mapsto\int q(y_L-\tilde m)v^2_0(y_L,y_U|z)d\lambda(y_L,y_U)$ and $(z,\tilde m)\mapsto\int q(y_L-\tilde m)\bar v(x)v^2_{0}(y_L,y_U|z)d\lambda(y_L,y_U)$ are continuously differentiable  on $\mathcal Z\times D$.
By \eqref{eq:onv1} (with $\varphi(x)=1$), it then follows that $(z,\tilde m,\eta)\mapsto\int_{\mathcal X}q(y_L-\tilde m)\varphi(x)v^2_\eta(y_L,y_U|z)d\lambda(y_L,y_U)$
is continuously differentiable on $\mathcal Z\times D\times N$ for some $N$ that contains $0$ in its interior.
	 Following the argument in the proof of Theorem 3.1 in \cite{NeweyStoker1993EJES}, it then follows that  $\nabla_z m_{L,\eta}$ exists and is continuous on $\mathcal Z\times N$. By the compactness of  $\mathcal Z$, $N$ can be chosen small enough so that $\nabla_z m_{L,\eta}$ is bounded on $\mathcal Z\times N$.   The argument for $m_{U,\eta}$ is similar. Hence it is omitted.
\end{proof}

\begin{theorem}\label{thm:tangent}
Let Assumptions \ref{as:onY}-\ref{as:onw} and \ref{as:onmu} hold and $P\in\mathbf P.$ Then, the tangent space of $\mathbf S$ at $s\equiv \sqrt{dP/d\mu}$ is given by $\dot{\mathbf S}=\{h\in L^2_\mu:\int h(x)s(x)d\mu(x)=0.\}$
\end{theorem}

\begin{proof}[\rm Proof of Theorem \ref{thm:tangent}]
	Let $\mathbf T\equiv \{h\in L^2_\mu:\int h(x)s(x)d\mu(x)=0\}$.
 $\dot{\mathbf S}\subseteq \mathbf T$ holds by Proposition 3.2.3 in \cite{BickelKlassenRitov1993}.	
	
For the converse: $\mathbf T\subseteq\dot{\mathbf S}$, it suffices to show that a dense subset of $\mathbf T$ is contained in $\dot{\mathbf S}$. For this, let $\dot{\mathbf U}\equiv\{\dot h_0\in L^2_\mu:\int \dot h_0(x)s(x)d\mu(x)=0\}$ denote the tangent space of the curve defined in \eqref{eq:tangent3}-\eqref{eq:heta}.   
By Lemmas \ref{lem:onm}-\ref{lem:onmLmU}, there is a neighborhood $N$ of 0 for which Assumptions \ref{as:onP1} and \ref{as:onP2} hold for all $\eta\in N$. Therefore, $\eta\mapsto h^2_\eta,\eta\in N$ is a regular parametric submodel of $\mathbf P$ whose Fr\'{e}chet derivative at $\eta=0$ is given by
$\dot h_{0}$. Hence, $\dot{\mathbf U}\subseteq\dot{\mathbf S}.$
Further, by Lemma C.7 in \cite{Newey1991NSMES} and the argument used in the proof of Theorem 3.1 in \cite{NeweyStoker1993EJES}, $\dot{\mathbf U}$ is dense in $\mathbf T$.
Thus, $\mathbf T\subseteq \dot{\mathbf S}$.
\end{proof}

\begin{lemma}\label{lem:support}
	Let $\eta\mapsto h_\eta$ be a curve in $L^2_\mu$ defined in \eqref{eq:tangent3}-\eqref{eq:heta}. Suppose Assumption \ref{as:onY} holds.
Suppose further that $P\in\mathbf P$. Then, there is a compact set $D'$ and a neighborhood $N$ of 0 such that  $D'$ contains the support of $Y_L-m_{L,\eta_0}(Z)$ and $Y_U-m_{U,\eta_0}(Z)$ in its interior for all $\eta_0\in N$.
\end{lemma}

\begin{proof}[\rm Proof of Lemma \ref{lem:support}]
By Lemma \ref{lem:onw}, there exists a neighborhood $N'$ of 0 such that the supports of $Y_L$ and $Y_U$ are contained in the interior of $D$ under $P_\eta$ for all $\eta$ in $N'$. 
Similarly, by Lemma \ref{lem:onm}, there is a neighborhood $N''$ of 0 such that $m_{L,\eta}(Z),m_{U,\eta}(Z)$ are well defined for all $\eta\in N''$ and their supports   are contained in the interior of $D$ respectively.  
Without loss of generality, let $D=[a,b]$ for some $-\infty<a<b<\infty$ and let $N=N'\cap N''.$ Then, the support of $Y_L-m_{L,\eta}(Z)$ is contained in $D'\equiv [a-b,b-a]$ for all $\eta\in N$.
A similar argument ensures that the support of  $Y_U-m_{U,\eta}(Z)$ is contained in $D'$. This completes the proof.
\end{proof}

\begin{lemma}\label{lem:dmlu}
		Let $\eta\mapsto h_\eta$ be a curve in $L^2_\mu$ defined in \eqref{eq:tangent3}-\eqref{eq:heta}. Suppose Assumption \ref{as:onY} holds.
	Suppose further that $P\in\mathbf P$. Then there is a neighborhood $N$ of 0 such that (i) 
 the functions $(z,\eta_0)\mapsto\frac{\partial}{\partial \eta}m_{L,\eta}(z)\big|_{\eta=\eta_0}$ and $(z,\eta_0)\mapsto\frac{\partial}{\partial \eta}m_{U,\eta}(z)\big|_{\eta=\eta_0}$ are bounded on $\mathcal Z\times N$; (ii) For each $z\in\mathcal Z$, the maps $\eta_0\mapsto\frac{\partial}{\partial \eta}m_{L,\eta}(z)\big|_{\eta=\eta_0}$ and $\eta_0\mapsto\frac{\partial}{\partial \eta}m_{U,\eta}(z)\big|_{\eta=\eta_0}$ are continuous on $N$.
\end{lemma}

\begin{proof}[\rm Proof of Lemma \ref{lem:dmlu}]
By Lemmas \ref{lem:onm}, \ref{lem:ons} and \ref{lem:support}, Assumption \ref{as:onY} and $\bar v$ being bounded, it follows that 
\begin{align}
\big|	\frac{\partial}{\partial \eta}m_{j,\eta}(z)\Big|_{\eta=\eta_0}\big|\le 2\bar\epsilon^{-1}\sup_{u\in D'}|q(u)|\times\sup_{x\in\mathcal X}\bar v(x)<\infty,~j=L,U.
\end{align}		
Hence, the first claim follows.
Now let $\eta_n\in N$ be a sequence such that $\eta_n\to\eta_0.$ Then, by the triangle and Cauchy-Schwarz inequalities,
\begin{align}
	\Big| \frac{\partial}{\partial \eta}m_{L,\eta}(z)\Big|_{\eta=\eta_n}&-	\frac{\partial}{\partial \eta}m_{L,\eta}(z)\Big|_{\eta=\eta_0}\Big|\le 2|r_{L,\eta_n}^{-1}(z)-r^{-1}_{L,\eta_0}(z)|\sup_{u\in D'}|q(u)|\times\|\dot v_{\eta_0}\|_{L^2_\lambda}\|v_{\eta_0}\|_{L^2_\lambda}\\
	&+2\bar\epsilon^{-1}\sup_{u\in D'}|q(u)|\big(\| v_{\eta_n}-v_{\eta_0}\|_{L^2_\lambda}\|\dot v_{\eta_n}\|_{L^2_\lambda}+\|\dot v_{\eta_n}-\dot v_{\eta_0}\|_{L^2_\lambda}\|v_{\eta_0}\|_{L^2_\lambda}\big)\\
&+2\bar\epsilon^{-1}	\int |q(y_L-m_{\eta_n}(z))-q(y_L-m_{\eta_0}(z))|\dot v_{\eta_0}(y_L,y_U|z)v_{\eta_0}(y_L,y_Uz)d\lambda(y_L,y_U).\label{eq:mL3}
	\end{align}
Note that $|r_{L,\eta_n}^{-1}-r^{-1}_{L,\eta_0}|\to 0$, $a.e.$ by \eqref{eq:tangent7}. By the continuous Fr\'echet differentiability of $\eta\mapsto v_\eta$, it follows that $\| v_{\eta_n}-v_{\eta_0}\|_{L^2_\lambda}=o(1)$ and $\|\dot v_{\eta_n}-\dot v_{\eta_0}\|_{L^2_\lambda}=o(1)$. Further, \eqref{eq:mL3} tends to 0 as $\eta_n\to\eta_0$ by the dominated convergence theorem, almost everywhere continuity of $q$ ensured by Assumption \ref{as:onY} (ii), and $m_{\eta_n}\to m_{\eta_0}, a.e.$ by  Lemma \ref{lem:onmLmU}. Therefore, $| \frac{\partial}{\partial \eta}m_{L,\eta}(z)|_{\eta=\eta_n}-	\frac{\partial}{\partial \eta}m_{L,\eta}(z)|_{\eta=\eta_0}|=o(1).$ The continuity of $\frac{\partial}{\partial \eta}m_{U,\eta}(z)|_{\eta=\eta_0}$ can be shown in the same way. This completes the proof.
\end{proof}

\begin{lemma}\label{lem:intermediate}
		Let $\eta\mapsto h_\eta$ be a curve in $L^2_\mu$ defined in \eqref{eq:tangent3}-\eqref{eq:heta}. Suppose Assumption \ref{as:onmu} holds.
	Suppose further that $P\in\mathbf P$. Then, there is a neighborhood $N$ of 0 such that
	\begin{align}
\frac{\partial}{\partial \eta}\int p'l_{\eta_0}(z)1\{p'l_{\eta}(z)>0\}(m_{U,\eta_0}(z)-m_{L,\eta_0}(z))\phi^2_{\eta_0}(z)d\nu(z)\Big|_{\eta=\eta_0}=0,
	\end{align}
for all $\eta_0\in N$.
\end{lemma}
\begin{proof}[\rm Proof of Lemma \ref{lem:intermediate}]
By \eqref{eq:tangent5}, there is a neighborhood $N$ of 0 such that for all $\eta_0\in N$,
\begin{align}
\frac{\partial p'l_\eta(z)}{\partial\eta}\Big|_{\eta=\eta_0}=\frac{4p'\nabla_z\bar\phi(z)(1+\eta_0\bar\phi(z))}{(1+2\eta_0\bar\phi(z))^2}.	
\end{align}
Hence, by compactness of $\mathbb S^{\ell}\times\mathcal Z$, the map $(p,z,\eta_0)\mapsto\frac{\partial p'l_\eta(z)}{\partial\eta}|_{\eta=\eta_0}$ is uniformly bounded on $\mathbb S^\ell\times \mathcal Z\times N$ by some constant $M>0$. Let $\eta_n$ be a sequence such that $\eta_n\to\eta_0\in N$. By the mean value theorem, there is $\bar \eta_n(z)$ such that
\begin{align}
\lim_{n\to\infty}&\int p'l_{\eta_0}(z)(1\{p'l_{\eta_n}(z)>0\}-1\{p'l_{\eta_0}(z)>0\})(m_{U,\eta_0}(z)-m_{L,\eta_0}(z))\phi^2_{\eta_0}(z)d\nu(z)\notag	\\
= &\lim_{n\to\infty}\int p'l_{\eta_0}(z)(1\{p'l_{\eta_0}(z)>(\eta_0-\eta_n)\frac{\partial p'l_\eta(z)}{\partial\eta}\Big|_{\eta=\bar\eta_n(z)}\}-1\{p'l_{\eta_0}(z)>0\})(m_{U,\eta_0}(z)-m_{L,\eta_0}(z))\phi^2_{\eta_0}(z)d\nu(z)\notag\\
\le &\lim_{n\to\infty}\int |p'l_{\eta_0}(z)|1\{|p'l_{\eta_0}(z)|\le M|\eta_0-\eta_n|\}|m_{U,\eta_0}(z)-m_{L,\eta_0}(z)|\phi^2_{\eta_0}(z)d\nu(z),\label{eq:dom}
\end{align}
where the last inequality follows  from $\frac{\partial p'l_\eta(z)}{\partial\eta}|_{\eta=\bar\eta_n(z)}$ being  bounded by $M$. Therefore, from \eqref{eq:dom}, we conclude that
\begin{multline}
	\lim_{n\to\infty}\frac{1}{|\eta_n-\eta_0|}|\int p'l_{\eta_0}(z)(1\{p'l_{\eta_n}(z)>0\}-1\{p'l_{\eta_0}(z)>0\})(m_{U,\eta_0}(z)-m_{L,\eta_0})\phi^2_{\eta_0}(z)d\nu(z)|\\
	\le 2\sup_{y\in D}|y|\times  M\times \lim_{n\to\infty}\int 1\{|p'l_{\eta_0}(z)|\le M|\eta_0-\eta_n|\}\phi^2_{\eta_0}(z)d\nu(z)=0,
\end{multline}
where the last equality follows from the monotone convergence theorem and $P$ being in $\mathbf P$ ensuring Assumption \ref{as:onmu} (ii).
\end{proof}

\begin{lemma}\label{lem:pathdiff_ptws}
		Let $\eta\mapsto h_\eta$ be a curve in $L^2_\mu$ defined in \eqref{eq:tangent3}-\eqref{eq:heta}. Suppose Assumptions \ref{as:onY}-\ref{as:onw}, and \ref{as:onmu} hold.
	Suppose further that $P\in\mathbf P$. Then, there is a neighborhood $N$ of 0 such that for all $\eta_0\in N$,
\begin{align}
	\frac{\partial 	\upsilon(p,\Theta_0(P_\eta))}{\partial\eta}\bigg|_{\eta=\eta_0}
	=2\int\{w(z)p'\nabla_zm_{p,\eta_0}(z)-\upsilon(p,\Theta_0(P_{\eta_0}))+ p'l_{\eta_0}(z)\zeta_{p,\eta_0}(y_L,y_U,z)\}\dot h_{\eta_0}(x)h_{\eta_0}d\mu(x).
\end{align}
\end{lemma}
\begin{proof}[\rm Proof of Lemma \ref{lem:pathdiff_ptws}]
We first show $\Gamma(\nabla_z m_{L,\eta}(z),\nabla_z m_{U,\eta}(z),p'l_\eta(z))$ is the gradient  of $m_{p,\eta}(z)$, $\mu-a.e.$
By Assumption \ref{as:onmu} (ii), it suffices to show the equality for $z$ such that $p'l_\eta(z)\ne 0$. Write 
\begin{align}
m_{p,\eta}&(z+h)-m_{p,\eta}(z)-\Gamma(\nabla_z m_{L,\eta}(z),\nabla_z m_{U,\eta}(z),p'l_\eta(z))'h\notag\\
&=1\{p'l_{\eta}(z+h)>0\}[(m_{U,\eta}(z+h)-m_{L,\eta}(z+h))-(m_{U,\eta}(z)-m_{L,\eta}(z))-(\nabla_z m_{U,\eta}(z)-\nabla_z m_{L,\eta}(z))'h]\notag\\
&\qquad+(1\{p'l_{\eta}(z+h)>0\}-1\{p'l_{\eta}(z)>0\})\times [(m_{U,\eta}(z)-m_{L,\eta}(z))-(\nabla_z m_{U,\eta}(z)-\nabla_z m_{L,\eta}(z))'h]\notag\\
&\qquad +(m_{L,\eta}(z+h)-m_{L,\eta}(z)-\nabla_z m_{L,\eta}(z)'h).\label{eq:mpeta}
\end{align}
 $\nabla_z m_{U,\eta}$ and $\nabla_z m_{L,\eta}$ being the gradients of $m_{U,\eta}$ and $m_{L,\eta}$, respectively implies that
\begin{align}
	(m_{U,\eta}(z+h)-m_{L,\eta}(z+h))-(m_{U,\eta}(z)-m_{L,\eta}(z))-(\nabla_z m_{U,\eta}(z)-\nabla_z m_{L,\eta}(z))'h=o(\|h\|),\notag\\
	m_{L,\eta}(z+h)-m_{L,\eta}(z)-\nabla_z m_{L,\eta}(z)'h=o(\|h\|).
\end{align}
By the continuity of $z\mapsto l_\eta(z)$, ensured by Assumption \ref{as:onw}, \eqref{eq:tangent5}, and $\bar\phi$ being continuously differentiable, there exists $\epsilon>0$ such that $1\{p'l_\eta(z+h)>0\}=1\{p'l_\eta(z)>0\}$ for all $h$ such that $\|h\|<\epsilon.$
These results and \eqref{eq:mpeta} ensure that
\begin{align}
m_{p,\eta}&(z+h)-m_{p,\eta}(z)-\Gamma(\nabla_z m_{L,\eta}(z),\nabla_z m_{U,\eta}(z),p'l_\eta(z))'h=o(\|h\|).	
\end{align}
In what follows, we therefore simply write $\nabla_zm_{p,\eta}(z)=\Gamma(\nabla_z m_{L,\eta}(z),\nabla_z m_{U,\eta}(z),p'l_\eta(z)).$

Next, we show that $\eta\mapsto \nabla_z m_{p,\eta}(z)$ is continuous for almost all $z\in\mathcal Z.$ By Lemma \ref{lem:dmlu}, $\eta\mapsto\nabla_zm_{L,\eta}$ and $\eta\mapsto \nabla_z m_{U,\eta}$ are continuous. Further, if $\eta_n\to\eta_0$, then
\begin{align*}
\mu(\{x:\lim_{n\to\infty}1\{p'l_{\eta_n}(z)>0\}=1\{p'l_{\eta_0}(z)>0\}\})=1	
\end{align*}
by the continuity of $\eta\mapsto p'l_\eta(z)$ ensured by \eqref{eq:tangent5} and $\mu(\{x:p'l_{\eta_0}(z)=0\})=0$ by Assumption \ref{as:onmu} (ii). Hence, $\eta\mapsto \nabla_z m_{p,\eta}(z)$ is continuous $a.e.$

By Theorem  \ref{thm:thetabounds}, integration by parts, and \eqref{eq:expectation}, we may write
	\begin{align}
	\upsilon(p,\Theta_0(P_\eta))	&=\int w(z)p'\nabla_z m_{p,\eta}(z)\phi_{\eta_0}^2(z)d\nu(z)+2(\eta-\eta_0)\int w(z)p'\nabla_z m_{p,\eta}(z)\dot\phi_{\eta_0}(z)\phi_{\eta_0}(z)d\nu(z)\\
		&=\int p'l_{\eta_0}(z)m_{p,\eta}(z)\phi^2_{\eta_0}(z)d\nu(z)+2(\eta-\eta_0)\int w(z)p'\nabla_z m_{p,\eta}(z)\dot\phi_{\eta_0}(z)\phi_{\eta_0}(z)d\nu(z).\label{eq:pathdiff2}
	\end{align}
This and $\upsilon(p,\Theta_0(P_{\eta_0}))=\int p'l_{\eta_0}(z)m_{p,\eta_0}(z)\phi^2_{\eta_0}(z)d\nu(z)$ by Theorem \ref{thm:thetabounds} imply
	\begin{align}
		&\frac{\partial 	\upsilon(p,\Theta_0(P_\eta))}{\partial\eta}\bigg|_{\eta=\eta_0}\notag\\
		&\hspace{0.3in}=\lim_{\eta\to\eta_0}\frac{1}{\eta-\eta_0}\int p'l_{\eta_0}(z)(m_{p,\eta}(z)-m_{p,\eta_0}(z))\phi^2_{\eta_0}(z)d\nu(z)+2\lim_{\eta\to \eta_0}\int w(z)p'\nabla_z m_{p,\eta}(z)\dot\phi_{\eta_0}(z)\phi_{\eta_0}(z)d\nu(z),\notag\\
&\hspace{0.3in}=\lim_{\eta\to\eta_0}\frac{1}{\eta-\eta_0}\int p'l_{\eta_0}(z)(m_{p,\eta}(z)-m_{p,\eta_0}(z))\phi^2_{\eta_0}(z)d\nu(z)+2\int w(z)p'\nabla_z m_{p,\eta_0}(z)\dot\phi_{\eta_0}(z)\phi_{\eta_0}(z)d\nu(z),\label{eq:pathdiff3'''}
	\end{align}
	where the second equality follows from $w$, $\nabla_z m_{p,\eta}$ and $\bar\phi_{\eta_0}$ being bounded by Assumption \ref{as:onw} and Lemma \ref{lem:onm}, which allows us to apply the dominated convergence theorem, and the almost everywhere continuity of $\eta\mapsto \nabla_z m_{p,\eta}(z)$.
The first term on the right hand side of \eqref{eq:pathdiff3'''} may be further rewritten as
\begin{align}
\lim_{\eta\to\eta_0}&\frac{1}{\eta-\eta_0}\int p'l_{\eta_0}(z)(m_{p,\eta}(z)-m_{p,\eta_0})\phi^2_{\eta_0}(z)d\nu(z)	\\
&=\lim_{\eta\to\eta_0}\frac{1}{\eta-\eta_0}\int p'l_{\eta_0}(z)1\{p'l_{\eta_0}(z)>0\}(m_{U,\eta}(z)-m_{U,\eta_0}(z))\phi^2_{\eta_0}(z)d\nu(z)\label{eq:decom1}\\
&\quad+\lim_{\eta\to\eta_0}\frac{1}{\eta-\eta_0}\int p'l_{\eta_0}(z)1\{p'l_{\eta_0}(z)\le 0\}(m_{L,\eta}(z)-m_{L,\eta_0}(z))\phi^2_{\eta_0}(z)d\nu(z)\label{eq:decom2}\\
&\quad+\lim_{\eta\to\eta_0}\frac{1}{\eta-\eta_0}\int p'l_{\eta_0}(z)(1\{p'l_{\eta}(z)>0\}-1\{p'l_{\eta_0}(z)>0\})(m_{U,\eta_0}(z)-m_{L,\eta_0}(z))\phi^2_{\eta_0}(z)d\nu(z).\label{eq:pathdiff3''}
\end{align}
For \eqref{eq:decom1}, by the mean value theorem, we have
\begin{align}
\lim_{\eta\to\eta_0}\frac{1}{\eta-\eta_0}\int p'l_{\eta_0}(z)1\{p'l_{\eta_0}(z)>0\}&(m_{U,\eta}(z)-m_{U,\eta_0})\phi^2_{\eta_0}(z)d\nu(z)\notag\\
&=\lim_{\eta\to\eta_0}\int p'l_{\eta_0}(z)1\{p'l_{\eta_0}(z)>0\} \frac{\partial}{\partial \eta}m_{U,\eta}(z)\Big|_{\eta=\bar\eta(z,\eta)}\phi^2_{\eta_0}(z)d\nu(z)\notag\\
&=\int p'l_{\eta_0}(z)1\{p'l_{\eta_0}(z)>0\} \frac{\partial}{\partial \eta}m_{U,\eta}(z)\Big|_{\eta=\eta_0}\phi^2_{\eta_0}(z)d\nu(z),\label{eq:pathdiff3''''}
\end{align}
where the first equality holds for each $p$ for some $\bar\eta(p,\eta)$ a convex combination of $\eta$ and $\eta_0$. The second equality follows from  Lemmas \ref{lem:onw} and \ref{lem:dmlu}, $\|p\|=1$, and Assumption \ref{as:onP1} (ii) justifying the use of the dominated convergence theorem.
Similarly, for \eqref{eq:decom2}, we have
\begin{multline}
\lim_{\eta\to\eta_0}\frac{1}{\eta-\eta_0}\int p'l_{\eta_0}(z)1\{p'l_{\eta_0}(z)\le 0\}(m_{L,\eta}(z)-m_{L,\eta_0})\phi^2_{\eta_0}(z)d\nu(z)\\
=\int p'l_{\eta_0}(z)1\{p'l_{\eta_0}(z)\le 0\} \frac{\partial}{\partial \eta}m_{L,\eta}(z)\Big|_{\eta=\eta_0}\phi^2_{\eta_0}(z)d\nu(z)~.
\end{multline}
Hence, by  \eqref{eq:pathdiff3'''}-\eqref{eq:pathdiff3''''}, integration by parts, and \eqref{eq:pathdiff3''} being 0 by Lemma \ref{lem:intermediate}, we obtain
	\begin{align}
		\frac{\partial 	\upsilon(p,\Theta_0(P_\eta))}{\partial\eta}\bigg|_{\eta=\eta_0}
		&=2\int p'l_{\eta_0}(z)\zeta_{p,\eta_0}(y_L,y_U,z)\dot v_{\eta_0}(y_L,y_U|z)v_{\eta_0}(y_L,y_U,z)d\lambda(y_L,y_U)\phi^2_{\eta_0}(z)d\nu(z)\\
		&\qquad+2\int w(z)p'\nabla_zm_{p,\eta_0}(z)\dot\phi_{\eta_0}(z)\phi_{\eta_0}(z)d\nu(z)~.
	\end{align}
	Using $\dot h_{\eta_0}=\dot v_{\eta_0}\phi_{\eta_0}+v_{\eta_0}\dot\phi_{\eta_0}$, $\int v_{\eta_0}^2d\lambda=1$, and $\int \dot h_{\eta_0}h_{\eta_0}d\mu=0$, we may rewrite this as
	\begin{align}
		\frac{\partial 	\upsilon(p,\Theta_0(P_\eta))}{\partial\eta}\bigg|_{\eta=\eta_0}
		=2\int\{w(z)p'\nabla_zm_{p,\eta_0}(z)-\upsilon(p,\Theta_0(P_{\eta_0}))+ p'l_{\eta_0}(z)\zeta_{p,\eta_0}(y_L,y_U,z)\}\dot h_{\eta_0}(x)h_{\eta_0}d\mu(x)~.
	\end{align}
	Therefore, the conclusion of the lemma follows.
\end{proof}

\begin{lemma}\label{lem:bdd}
Let $\eta\mapsto h_\eta$ be a curve in $L^2_\mu$ defined in \eqref{eq:tangent3}-\eqref{eq:heta}. Suppose Assumptions \ref{as:onY}-\ref{as:onw}, and \ref{as:onmu} hold. Suppose further that $P\in\mathbf P.$ Then,
there is a neighborhood $N$ of $\eta=0$ such that	the map $(p,\eta_0)\mapsto\frac{\partial 	\upsilon(p,\Theta_0(P_\eta))}{\partial\eta}\big|_{\eta=\eta_0}$ is uniformly bounded on $\mathbb S^\ell\times N$.
\end{lemma}
\begin{proof}[\rm Proof of Lemma \ref{lem:bdd}]
 By Lemma \ref{lem:pathdiff_ptws} and the triangle inequality, 
\begin{multline}
	\Big|\frac{\partial 	\upsilon(p,\Theta_0(P_\eta))}{\partial\eta}\Big|_{\eta=\eta_0}\Big|
	=2\Big|\int\{w(z)p'\nabla_zm_{p,\eta_0}(z)-\upsilon(p,\Theta_0(P_{\eta_0}))\}\dot h_{\eta_0}(x)h_{\eta_0}d\mu(x)\Big|\\
+2\Big|\int	p'l_{\eta_0}(z)\zeta_{p,\eta_0}(y_L,y_U,z)\dot h_{\eta_0}(x)h_{\eta_0}(x)d\mu(x)\Big|.\label{eq:bdd1}
\end{multline}
By Assumption \ref{as:onw} and $\|p\|=1$, uniformly on $N$, 
\begin{align}
\|w(z)p'\nabla_{z}m_{p,\eta}(z)\|_{L^\infty_\mu}\le\sup_{z\in\mathcal Z}|w(z)|\times\| \nabla_{z}m_{p,\eta}(z)\|_{L^\infty_\mu}\le \sup_{z\in\mathcal Z}|w(z)|\times(\sup_{z\in\mathcal Z}|\nabla_{z}m_{L,\eta}(z)|+\sup_{z\in\mathcal Z}|\nabla_{z}m_{U,\eta}(z)|)<\infty.\label{eq:bdd1.5}
\end{align}
where the last inequality follows from  Lemma \ref{lem:onm}. 
 This ensures that $(p,\eta)\mapsto\upsilon(p,\Theta_0(P_\eta))$ is uniformly bounded on $\mathbb S^\ell\times N$. We therefore have
\begin{multline}
\Big|\int\{w(z)p'\nabla_zm_{p,\eta_0}(z)-\upsilon(p,\Theta_0(P_{\eta_0}))\}\dot h_{\eta_0}(x)h_{\eta_0}d\mu(x)\Big|\\
\le \sup_{(p,\eta_0,z)\in \mathbb S^\ell\times N\times\mathcal Z}|w(z)p'\nabla_zm_{p,\eta_0}(z)-\upsilon(p,\Theta_0(P_{\eta_0}))|\|\dot h_{\eta_0}\|_{L^2_\mu}\|h_{\eta_0}\|_{L^2_\mu}<\infty.\label{eq:bdd2}
\end{multline}
Further, by  Assumption \ref{as:onY} (ii) and Lemmas \ref{lem:ons} and \ref{lem:support}, it follows that $\sup_{x\in\mathcal X}|\zeta_{p,\eta_0}(x)|\le 2\bar\epsilon^{-1}\sup_{u\in D'}|q(u)|<\infty$ for all $(p,\eta_0)\in \mathbb S^\ell\times N$. Therefore, by the triangle and Cauchy-Schwarz inequalities and Lemma \ref{lem:onw}, we have
\begin{multline}
	\Big|\int	p'l_{\eta_0}(z)\zeta_{p,\eta_0}(y_L,y_U,z)\dot h_{\eta_0}(x)h_{\eta_0}(x)d\mu(x)\Big|\\\le 2\bar\epsilon^{-1}\sup_{u\in D'}q(u)\int \|l_{\eta_0}(z)\|\dot h_{\eta_0}(x)h_{\eta_0}(x)d\mu(x)
	\le 2\bar\epsilon^{-1}\sup_{u\in D'}q(u)\|\dot h_{\eta_0}\|_{L^2_\mu}E_{\eta_0}[\|l_{\eta_0}(z)\|^2]^{1/2}<\infty.\label{eq:bdd3}
\end{multline}
By \eqref{eq:bdd1}, \eqref{eq:bdd2}, and \eqref{eq:bdd3}, the conclusion of the lemma follows.
\end{proof}

\begin{lemma}\label{lem:cts}
Let $\eta\mapsto h_\eta$ be a curve in $L^2_\mu$ defined in \eqref{eq:tangent3}-\eqref{eq:heta}. Suppose Assumptions \ref{as:onY}-\ref{as:onw}, and \ref{as:onmu} hold. Suppose further that $P\in\mathbf P.$
	Then,
	there is a neighborhood $N$ of $0$ such that	the map $(p,\eta_0)\mapsto\frac{\partial 	\upsilon(p,\Theta_0(P_\eta))}{\partial\eta}\big|_{\eta=\eta_0}$ is continuous at all $(p,\eta_0)\in \mathbb S^\ell\times N.$
\end{lemma}
\begin{proof}[\rm Proof of Lemma \ref{lem:cts}]
Let $(p_n,\eta_n)$ be a sequence such that $(p_n,\eta_n)\to(p,\eta_0).$ 
For each $(p,\eta,z)\in\mathbb S^\ell\times N\times\mathcal Z$, let $\gamma_{p,\eta}(z)\equiv w(z)p'\nabla_zm_{p,\eta}(z)-\upsilon(p,\Theta_0(P_{\eta})).$ We first show that  $(p,\eta)\mapsto \gamma_{p,\eta}(z)$  and $(p,\eta)\mapsto \zeta_{p,\eta}(x)$ are continuous $a.e$. By Lemma \ref{lem:dmlu}, $\eta\mapsto\nabla_zm_{L,\eta}(z)$ and $\eta\mapsto \nabla_z m_{U,\eta}(z)$ are continuous for every $z\in\mathcal Z$. Further, if $(p_n,\eta_n)\to(p,\eta_0)$, then
\begin{align*}
\mu(\{x:\lim_{n\to\infty}1\{p_n'l_{\eta_n}(z)>0\}=1\{p'l_{\eta_0}(z)>0\}\})=1	
\end{align*}
by the continuity of $(p,\eta)\mapsto p'l_\eta(z)$ implied by \eqref{eq:tangent5} and $\mu(\{x:p'l_{\eta_0}(z)=0\})=0$ by Assumption \ref{as:onmu} (ii). Hence, $(p,\eta)\mapsto w(z)p'\nabla_zm_{p,\eta}(z)$ is continuous $a.e$. Note that by \eqref{eq:expectation},
$\upsilon(p,\Theta_0(P_\eta))=\int w(z)p'\nabla_zm_{p,\eta}(z)(1+2(\eta-\eta_0)\bar\phi_{\eta_0}(z))\phi_{\eta_0}^2(z)d\nu(z)$. Hence, as $(p_n,\eta_n)\to(p,\eta_0)$, it follows that 
\begin{multline}
\lim_{n\to\infty}	\upsilon(p_n,\Theta_0(P_{\eta_n}))=\lim_{n\to\infty}\int w(z)p_n'\nabla_zm_{p_n,\eta_n}(z)(1+2(\eta_n-\eta_0)\bar\phi_{\eta_0}(z))\phi_{\eta_0}^2(z)d\nu(z)\\
=\int \lim_{n\to\infty}w(z)p_n'\nabla_zm_{p_n,\eta_n}(z)(1+2(\eta_n-\eta_0)\bar\phi_{\eta_0}(z))\phi_{\eta_0}^2(z)d\nu(z)= \upsilon(p,\Theta_0(P_{\eta_0})),
\end{multline}
where the second equality follows from $w(z)p'\nabla_zm_{p,\eta}(z)$ and $\bar\phi(z)$ being bounded on $\mathbb S^\ell\times N\times\mathcal Z$  and  an application of the dominated convergence theorem, while the last equality follows from the continuity of 
$(p,\eta)\mapsto w(z)p'\nabla_zm_{p,\eta}(z)$  for almost all $z$. Hence, $(p,\eta)\mapsto \gamma_{p,\eta}(z)$ is continuous $a.e.$

The maps $(p,\eta)\mapsto r_{j,\eta}^{-1}(z)q(y_j-m_{j,\eta}(z)),j=L,U$ are continuous for almost all $x$ by Assumption \ref{as:onY} (ii), \eqref{eq:tangent7}, and $\eta\mapsto m_{j,\eta}(z)$ being continuous for almost all $z$ for $j=L,U$ by Lemma \ref{lem:dmlu}. 
Since $(p,\eta)\mapsto 1\{p'l_\eta(z)>0\}\}$ is continuous for almost all $z$ as shown above, it then follows that $(p,\eta)\mapsto \zeta_{p,\eta}(x)$ is continuous for almost all $x$.

Given these results, we show $\frac{\partial 	\upsilon(p_n,\Theta_0(P_\eta))}{\partial\eta}|_{\eta=\eta_n}-	\frac{\partial 	\upsilon(p,\Theta_0(P_\eta))}{\partial\eta}|_{\eta=\eta_0}\to 0$ as $\eta_n\to \eta_0.$ Toward this end, we first note that
\begin{align}
\frac{\partial 	\upsilon(p_n,\Theta_0(P_\eta))}{\partial\eta}&\bigg|_{\eta=\eta_n}-	\frac{\partial 	\upsilon(p,\Theta_0(P_\eta))}{\partial\eta}\bigg|_{\eta=\eta_0}\label{eq:cts0}\\
	&=2\int\gamma_{p_n,\eta_n}(z)\dot h_{\eta_n}(x)h_{\eta_n}d\mu(x)-2\int\gamma_{p,\eta_0}(z)\dot h_{\eta_0}(x)h_{\eta_0}d\mu(x)\\
&\qquad+2\int p_n'l_{\eta_n}(z)\zeta_{p_n,\eta_n}(x)\dot h_{\eta_n}(x)h_{\eta_n}d\mu(x)	-2\int p'l_{\eta_0}(z)\zeta_{p,\eta_0}(x)\dot h_{\eta_0}(x)h_{\eta_0}d\mu(x).\label{eq:cts1}
\end{align}
By Lemma \ref{lem:bdd}, the Cauchy-Schwarz and triangle inequalities, 
\begin{align}
|\int\gamma_{p_n,\eta_n}(z)\dot h_{\eta_n}(x)&h_{\eta_n}d\mu(x)-\int\gamma_{p,\eta_0}(z)\dot h_{\eta_0}(x)h_{\eta_0}d\mu(x)	|\\
&\le \sup_{(p,\eta,z)\in\mathbb S^\ell\times N\times\mathcal Z}|\gamma_{p,\eta}(z)|(\|\dot h_{\eta_n}\|_{L^2_\mu}\|h_{\eta_n}-h_{\eta_0}\|_{L^2_\mu}+\|\dot h_{\eta_n}-\dot h_{\eta_0}\|_{L^2_\mu}\| h_{\eta_0}\|_{L^2_\mu})\\
&\qquad+|\int\{\gamma_{p_n,\eta_n}(z)-\gamma_{p,\eta_0}(z)\}\dot h_{\eta_0}(z)h_{\eta_0}(z)\mu(x)=o(1),
\end{align}
where the last equality follows from $\eta\mapsto h_\eta$ being continuously Fr\'{e}chet differentiable, $(p,\eta,z)\mapsto \gamma_{p,\eta}(z)$ being bounded on $\mathbb S^\ell\times N\times\mathcal Z$ as shown in Lemma \ref{lem:bdd},  the dominated convergence theorem, and $(p,\eta)\mapsto \gamma_{p,\eta}(z)$ being continuous $a.e.$

Further, we may write \eqref{eq:cts1} as
\begin{align}
\int p_n'l_{\eta_n}(z)\zeta_{p_n,\eta_n}(x)\dot h_{\eta_n}(x)h_{\eta_n}d\mu(x)	-&\int p'l_{\eta_0}(z)\zeta_{p,\eta_0}(x)\dot h_{\eta_0}(x)h_{\eta_0}d\mu(x)\\
&=\int p_n'l_{\eta_n}(z)\zeta_{p_n,\eta_n}(x)(\dot h_{\eta_n}(x)h_{\eta_n}(x)-\dot h_{\eta_0}(x)h_{\eta_0}(x))d\nu(x)\\
&\quad +\int 	p_n'l_{\eta_n}(z)(\zeta_{p_n,\eta_n}(x)-\zeta_{p,\eta_0}(x))\dot h_{\eta_0}(x)h_{\eta_0}(x)d\mu(x)\\
&\quad+\int (p_n'l_{\eta_n}(z)-p'l_{\eta_0}(z))\zeta_{p,\eta_0}(x)\dot h_{\eta_0}(x)h_{\eta_0}(x)d\mu(x).
\end{align}
By Assumptions \ref{as:onw} and \ref{as:onP1} (ii), \eqref{eq:tangent5}, and $\bar\phi$ being continuously differentiable, $(\eta,z)\mapsto\|l_\eta(z)\|$ is continuous on $N\times\mathcal Z$. Hence, it achieves a finite maximum on $N\times\mathcal Z$. Further, by Lemmas \ref{lem:ons}, \ref{lem:support} and Assumption \ref{as:onY} (ii),  $\sup_{x\in\mathcal X}|\zeta_{p,\eta_0}(x)|\le 2\bar\epsilon^{-1}\sup_{u\in D'}|q(u)|<\infty$ for all $(p,\eta_0)\in \mathbb S^\ell\times N$.
By Cauchy-Schwarz inequality, and $\|p_n\|\le 1$ for all $n$, it then follows that
\begin{multline}
\int p_n'l_{\eta_n}(z)(\dot h_{\eta_n}(x)h_{\eta_n}(x)-\dot h_{\eta_0}(x)h_{\eta_0}(x))d\nu(x)\\
	\le \sup_{(\eta,z)\in N\times\mathcal Z}\|l_\eta(z)\|\sup_{(p,\eta)\in \mathbb S^\ell\times N}\sup_{x\in\mathcal X}|\zeta_{p,\eta}(x)|(\|\dot h_{\eta_n}\|_{L^2_\mu}\|h_{\eta_n}-h_{\eta_0}\|_{L^2_\mu}+\|\dot h_{\eta_n}-\dot h_{\eta_0}\|_{L^2_\mu}\| h_{\eta_0}\|_{L^2_\mu})=o(1),
\end{multline}
where the last equality follows from $\eta\mapsto h_\eta$ being continuously Fr\'echet differentiable.
Further, by the almost everywhere continuity of $(p,\eta)\mapsto p'l_\eta(\zeta_{p,\eta}-\zeta_{p,\eta_0}),$ 
\begin{multline}
\lim_{n\to\infty}\int 	p_n'l_{\eta_n}(z)(\zeta_{p_n,\eta_n}(x)-\zeta_{p,\eta_0}(x))\dot h_{\eta_0}(x)h_{\eta_0}(x)d\mu(x)	\\
=\int \lim_{n\to\infty}	p_n'l_{\eta_n}(z)(\zeta_{p_n,\eta_n}(x)-\zeta_{p,\eta_0}(x))\dot h_{\eta_0}(x)h_{\eta_0}(x)d\mu(x)	=0.
\end{multline}
where the first equality follows from the dominated convergence theorem. Finally, again by the dominated convergence theorem,
\begin{multline}
\lim_{n\to\infty}	\int (p_n'l_{\eta_n}(z)-p'l_{\eta_0}(z))\zeta_{p,\eta_0}(x)\dot h_{\eta_0}(x)h_{\eta_0}(x)d\mu(x)\\
=\int \lim_{n\to\infty}(p_n'l_{\eta_n}(z)-p'l_{\eta_0}(z))\zeta_{p,\eta_0}(x)\dot h_{\eta_0}(x)h_{\eta_0}(x)d\mu(x)=0.\label{eq:ctsend}
\end{multline}
By \eqref{eq:cts0}-\eqref{eq:ctsend}, we conclude that  $\frac{\partial 	\upsilon(p_n,\Theta_0(P_\eta))}{\partial\eta}|_{\eta=\eta_n}-	\frac{\partial 	\upsilon(p,\Theta_0(P_\eta))}{\partial\eta}|_{\eta=\eta_0}\to 0$ as $\eta_n\to \eta_0.$ This establishes the claim of the lemma.
\end{proof}

\begin{theorem}\label{thm:pathdiff}
 Suppose Assumptions \ref{as:onY}-\ref{as:onP1}, and \ref{as:onP2} hold. Then, the mapping $\rho:\mathbf P\rightarrow \mathcal C(\mathbb S^\ell)$ pointwise defined by $\rho(h_\eta)(p) \equiv \upsilon(p,\Theta_0(P_\eta))$ for $h_\eta = \sqrt{dP_\eta/d\mu}$ is then  pathwise weak differentiable at $h_0 \equiv \sqrt{dP_0/d\mu}$. Moreover,  the derivative $\dot{\rho}:\dot{\mathbf P}\to\mathcal C(\mathbb S^\ell)$ satisfies:
	\begin{align}\dot{\rho}(\dot{h}_0)(p) &=   2\int\{ w(z)p'\nabla_zm_{p}(z)-\upsilon(p,\Theta_0(P_0))+p'l(z)\zeta_{p}(x)\}\dot h_0(x)h_0(x)d\mu(x).
	\end{align} 
\end{theorem}

\begin{proof}[\rm Proof of Theorem \ref{thm:pathdiff}]
We first show that $\partial 	\rho(P_\eta)/\partial\eta\big|_{\eta=0}$ is the pathwise weak derivative of $\rho$. For this, 
$\dot\rho(\dot h_0)\in\mathcal C(\mathbb S^\ell)$ for all $\dot h_0\in\dot{\mathbf S}$ as implied by Lemmas \ref{lem:pathdiff_ptws} and \ref{lem:cts}. Linearity of $\dot \rho$ is immediate, while continuity follows by noting that by the Cauchy-Schwarz inequality and $\|p\|=1$,
\begin{multline}
\sup_{\|\dot h_0\|_{L^2_\mu=1}}\|\dot\rho(\dot h_0)\|_\infty\\	
\le 2 \{\sup_{p\in\mathbb S^\ell}\|w(z)p'\nabla_zm_p(z)-\upsilon(p,\Theta_0(P))\|_{L^\infty_\mu}+\sup_{z\in\mathcal Z}\|l(z)\|\times\sup_{u\in\mathcal D'}|q(u)|\}\|\dot h_0\|_{L^2_\mu}\|h_0\|_{L^2_\mu}<\infty,
\end{multline} 
where we exploited \eqref{eq:bdd1.5}, Assumption \ref{as:onY} (ii), and the fact that $z\mapsto\|l(z)\|$ being continuous hence achieves a finite maximum on $\mathcal Z$ by  Assumptions \ref{as:onY} (i), \ref{as:onw}, and $P\in \mathbf P$.
	Let $\eta\mapsto h_\eta$ be a curve in $L^2_\mu$ defined in \eqref{eq:tangent3}-\eqref{eq:heta}. 
For each $p\in\mathbb S^\ell$, by the mean value theorem,
\begin{align}
\lim_{\eta_0 \rightarrow 0} \int_{\mathbb S^\ell} \frac{\upsilon(p,P_{\eta_0}) - \upsilon(p,P_0)}{\eta_0}dB(p) = \lim_{\eta_0 \rightarrow 0} \int_{\mathbb S^\ell} \frac{\partial \upsilon(p,P_\eta)(p)}{\partial\eta} \Big|_{\eta = \bar{\eta}(p,\eta_0)} dB(p) \label{eq:pathdiffrho0}\\ = \int                        _{\mathbb S^\ell}\frac{\partial \upsilon(p,P_\eta)(p)}{\partial\eta} \Big|_{\eta = 0}  dB(p) = \int_{\mathbb S^\ell} \dot{\rho}(\dot h_0)(\tau)dB(p) ~,	\label{eq:pathdiffrho1}
\end{align}
where the first equality holds at each $p$ for some $\bar \eta(p,\eta_0)$ a convex combination of $\eta_0$ and $0$. The second equality in turn follows by Lemma \ref{lem:bdd} justifying the use of the dominated convergence theorem, while the final equality follows by Lemma \ref{lem:cts} and the definition of $\dot \rho :\dot{\mathbf P}\rightarrow \mathcal C(\mathbb S^\ell)$. 
Eqs. \eqref{eq:pathdiffrho0}-\eqref{eq:pathdiffrho1} hold for any $\dot h_0$ in the tangent space $\dot{\mathbf U}$ of the curve defined in \eqref{eq:tangent3}-\eqref{eq:heta}. 
As discussed in the proof of Theorem \ref{thm:tangent}, $\dot{\mathbf U}$ is dense in $\dot{\mathbf S}.$ Since $\dot\rho$ is continuous, Eqs. \eqref{eq:pathdiffrho0}-\eqref{eq:pathdiffrho1} then hold for any $\dot h_{0}\in\dot{\mathbf P}$. This completes the proof.
\end{proof}

\begin{proof}[\rm Proof of Theorem \ref{thm:InvCov}]
Let $\mathbf B\equiv \mathcal C(\mathbb S^\ell)$ and let $\mathbf B^*$ be the set of finite Borel measures on $\mathbb S^\ell$, which is the norm dual of $\mathbf B$ by Corollary 14.15 in \cite{Aliprantis_Border2006a}.	By Theorem \ref{thm:pathdiff}, $\rho$ has pathwise weak derivative $\dot \rho$. For each $B\in\mathbf B^*$, define
\begin{align}
	\dot\rho^T(B)(x)\equiv \int_{\mathbb S^\ell }2\{ w(z)p'\nabla_zm_{p}(z)-\upsilon(p,\Theta_0(P))+p'l(z)\zeta_{p}(x)\}h_0(x)dB(p).\label{eq:dotrhoT}
\end{align}
	We show that (i) $\dot \rho^T$ is well defined for any $B\in \mathbf B^*$, (ii) $\dot\rho^T(B) \in \dot{\mathbf S}$ and finally (iii) $\dot\rho^T$ is the adjoint operator of $\dot\rho$. 
	
	We first note that $(p,z)\mapsto p'l(z)$ is continuous in $z$ for each $p$ by Assumption \ref{as:onw} and measurable in $p$ for each $z$. Thus, $(p,x)\mapsto p'l(z)$ is jointly measurable by  Lemma 4.51 in \cite{Aliprantis_Border2006a}. This implies the joint measurability of $(p,x)\mapsto 1\{p'l(z)>0\}$. A similar argument also ensures the joint measurability of $p'\nabla_zm_L(z)$ and $p'\nabla_zm_U(z)$. By the joint measurability of $(p,x)\mapsto (w(z),1\{p'l(z)>0\},p'\nabla_zm_L(z),p'\nabla_zm_U(z))$ and Assumption \ref{as:onw}, $(p,x)\mapsto w(z)p'\nabla_z m_p(z)$ is jointly measurable. By the proof of Theorem \ref{thm:thetabounds}, $\upsilon(p,\Theta_0(p))$ is differentiable in $p$ and is therefore continuous, implying $(p,x)\mapsto \upsilon(p,\Theta_0(p))$ is jointly measurable. 
	Further, $r_L$ and $r_U$ are measurable by $P\in\mathbf P$ satisfying Assumption \ref{as:onP2} (iii). $q(y_L-m_L(z)),q(y_U-m_U(z))$ are measurable by Assumption \ref{as:onY} and $P\in\mathbf P$ satisfying Assumption \ref{as:onP1} (iv). Hence, $(p,x)\mapsto \zeta_p(x)$ is jointly measurable.
	Therefore, the map $(p,x)\mapsto (w(z)p'\nabla_z m_p(z),\upsilon(p,\Theta_0(p),p'l(z),\zeta_p(x),h_0(x))'$ is jointly measurable by Lemma 4.49 in \cite{Aliprantis_Border2006a}. Hence, the map
	\begin{align}
(p,x)\mapsto 2\{w(z)p'\nabla_zm_{p}(z)-\upsilon(p,\Theta_0(P_0))+p'l(z)\zeta_{p}(x)\}
	\end{align}
is jointly measurable by the measurability of the composite map.
	
Moreover, for $|B|$ the total variation of the measure $B$, by \eqref{eq:bdd1.5}, we have
\begin{multline}
\int(	\int_{\mathbb S^\ell}2\{w(z)p'\nabla_zm_{p}(z)-\upsilon(p,\Theta_0(P_0))\}h_0(z)dB(p))^2d\mu(x)\\
\le 16\times \sup_{p\in\mathbb S^\ell}\|w(z)p'\nabla_{z}m_{p,\eta}(z)\|_{L^\infty_\mu}^2\times |B|^2<\infty.
\end{multline}
Further,
\begin{align}
\int(	\int_{\mathbb S^\ell}2p'l(z)\zeta_{p}(x)h_0(x))dB(p))^2d\mu(x)\le 16 \int |p'l(z)|^2h_0(x)^2d\mu(x)\times \bar\epsilon^{-2}\times\sup_{u\in D'} |q(u)|^2 \times |B|^2<\infty,
\end{align}
by Assumption \ref{as:onY} and $p\in\mathbf P$ satisfying Assumptions \ref{as:onP1} and \ref{as:onP2}.
Therefore, $\dot\rho^T(B)\in L^2_\mu$ for each $B\in\mathbf B^*$. 

By Fubini's theorem and  Assumption \ref{as:onY} (iv), we  have
\begin{multline}
	\int\int_{\mathbb S^\ell} 2\{ w(z)p'\nabla_zm_{p}(z)-\upsilon(p,\Theta_0(P_0))+p'l(z)\zeta_{p}(x)\}h_0(x)dB(p)h_0(x)d\mu(x)\\
	=	\int_{\mathbb S^\ell}\int 2\{ w(z)p'\nabla_zm_{p}(z)-\upsilon(p,\Theta_0(P_0))+p'l(z)\zeta_{p}(x)\}h^2_0(x)d\mu(x)dB(p)=0,
\end{multline}
where we exploited $\upsilon(p,\Theta_0(P_0))=E[w(Z)p'\nabla_zm_{p}(Z)]$ and $E[\zeta_{p}(x)|Z=z]=1\{p'l(z)\le 0\}E[q(Y_L-m_L(Z))|Z=z]+1\{p'l(z)>0\}E[q(Y_U-m_U(Z))|Z=z]=0$, $P-a.s$.
 Thus, by Theorem \ref{thm:tangent} and \eqref{eq:dotrhoT}, $\dot\rho^T(B)\in \dot{\mathbf S}$ for all $B\in\mathbf B^*$. Further, for any $\dot h_0\in\dot{\mathbf S}$, again by interchanging the order of integration
\begin{align}
	\int_{\mathbb S^\ell}\dot \rho(\dot h_0)(p)dB(p)=\int_{\mathcal X}\dot h_0(x)\dot\rho^T(B)(x)d\mu(x),
\end{align}
which ensures that $\dot\rho^T:\mathbf B^*\to\dot{\mathbf P}$ is the adjoint of $\dot\rho:\dot{\mathbf P}\to\mathbf B$.

Since $\dot{\mathbf S}$ is linear by Theorem \ref{thm:tangent}, Theorem \ref{thm:pathdiff} and Theorem 5.2.1 in \cite{BickelKlassenRitov1993} establishes that
\begin{multline}
\text{Cov}(\int_{\mathbb S^\ell}\mathbb G(p)dB_1(p),\int_{\mathbb S^\ell}\mathbb G(p)dB_2(p))=\frac{1}{4}\int_{\mathcal X}\dot\rho^T(B_1)(x)	\dot\rho^T(B_2)(x)	d\mu(x)\\
=\int_{\mathbb S^\ell}\int_{\mathbb S^\ell}E[\{ w(z)p'\nabla_zm_{p}(z)-\upsilon(p,\Theta_0(P_0))+p'l(z)\zeta_{p}(x)\}\{ w(z)q'\nabla_zm_{q}(z)-\upsilon(q,\Theta_0(P_0))+q'l(z)\zeta_{q}(x)\}]dB_1(p)dB_2(q),\label{eq:covker}
\end{multline}
for any $B_1,B_2\in\mathbf B^*$ by Fubini's theorem. Letting $B_1$ and $B_2$ be the degenerate  measures at $p$ and $q$ in \eqref{eq:covker} gives the desired result.
\end{proof}
}

\vspace{0.1in}
\begin{center}
{\sc {Appendix D}}: Proof of Theorem 4.1.
\end{center}
\renewcommand{\thedefinition}{D.\arabic{definition}}
\renewcommand{\theequation}{D.\arabic{equation}}
\renewcommand{\thelemma}{D.\arabic{lemma}}
\renewcommand{\thecorollary}{D.\arabic{corollary}}
\renewcommand{\thetheorem}{D.\arabic{theorem}}
\setcounter{lemma}{0}
\setcounter{theorem}{0}
\setcounter{corollary}{0}
\setcounter{equation}{0}
\setcounter{remark}{0}

{\footnotesize
In this appendix, we establish Theorem \ref{thm:estimator}.
Throughout, let $Y_{p,i}\equiv 1\{p' l(Z_i)\le 0\}Y_{L,i}+1\{p' l(Z_i)> 0\}Y_{U,i}$ and let $\bar \upsilon_n(p)\equiv \frac{1}{n}\sum_{i=1}^n p'\hat l_{i,h}(Z_{i})Y_{p,i}.$ The proof of Theorem \ref{thm:estimator} proceeds by decomposing $\sqrt n(\hat \upsilon_n(p)-\upsilon(p,\Theta_0(P_0)))$ as follows:
	\begin{multline} 
\sqrt n(\hat \upsilon_n(p)-\upsilon(p,\Theta_0(P_0)))=\sqrt n(\hat \upsilon_n(p)-\bar \upsilon_n(p))+\sqrt n(\bar \upsilon_n(p)-E[\bar \upsilon_n(p)])+\sqrt n(E[\bar \upsilon_n(p)]-\upsilon(p,\Theta_0(P_0)))\\
\equiv G_{1n}(p)+G_{2n}(p)+G_{3n}(p).\label{eq:effdecomp}
	\end{multline}
$G_{1n}$ is the difference between $\hat\upsilon_n$ and the infeasible estimator $\bar \upsilon_n$ which requires the knowledge of $Y_{p,i}$. $G_{2n}$ represents a properly centered version of $\bar\upsilon_n$, and $G_{3n}$ is the asymptotic bias of $\bar\upsilon_n$. The auxiliary lemmas are then used to show the following results:

\noindent
{\it Step 1:} Lemma \ref{lem:G1} shows $G_{1n}=o_p(1)$ uniformly in $p\in\mathbb S^\ell$ using the result of Lemma \ref{lem:euclidean}, while Lemma \ref{lem:G3} shows $G_{3n}=o(1)$ uniformly in $p\in\mathbb S^\ell$.

\noindent
{\it Step 2:} Lemmas \ref{lem:pss3.1} and \ref{lem:G2} then establish that $G_{2n}=\frac{1}{\sqrt n}\sum_{i=1}^n\psi_p(X_i)+o_p(1)$  uniformly in $p\in\mathbb S^\ell$, and Lemma \ref{lem:donsker} establishes that $\{\psi_p:p\in\mathbb S^\ell\}$ is a $P$-Donsker class.

\noindent
{\it Step 3:} Combining Steps 1-2 and \eqref{eq:effdecomp} gives the main claim of Theorem \ref{thm:estimator}.

Before proceeding further, we introduce one more piece of notation.
For each $p\in\mathbb S^\ell$, define
\begin{align}
p_n(x_i,x_j;p)\equiv -\Big(\frac{1}{h}\Big)^{\ell+1}p'\nabla_z K\Big(\frac{z_i-z_j}{h}\Big)(y_{p,i}-y_{p,j}),~\text{ and }~
r_n(x_i;p)\equiv E[p_n(X_i,X_j;p)|X_i=x_i].\label{eq:pnrn}
\end{align}

\begin{lemma}\label{lem:euclidean}
Let $\mathcal H_{n}\equiv\{\tilde p_n:\mathcal X\times\mathcal X\to\mathbb R:\tilde p_n(x,x';p)=-p'\nabla_z K(\frac{z-z'}{h})(y_{p}-y_{p}'),p\in\mathbb B^\ell\}$ and $\mathcal G_n\equiv\{\tilde q_n:\mathcal X\times\mathcal X\to\mathbb R:\tilde q_n(x,x';p)=\tilde p_n(x,x';p)-\tilde r_n(x,p)-\tilde r_n(x';p)-E[\tilde r_n(x;p)],p\in\mathbb B^\ell\}$ where $\tilde r_n(x_i;p)\equiv E[\tilde p_n(X_i,X_j;p)|X_i=x_i]$ and $\mathbb B^\ell\equiv\{p\in\mathbb R^\ell:\|p\|\le 1\}$ is the unit ball in $\mathbb R^\ell$.  Suppose Assumption \ref{as:onK} holds. Then, $\mathcal H_n$ and $\mathcal G_n$ are Euclidean in the sense of \cite{Sherman:1994uq} with envelope functions $H:\mathcal X\times\mathcal X\to \mathbb R$ and $G:\mathcal X\times\mathcal X\to \mathbb R$ such that $E[H(X_i,X_j)^2]<\infty$ and $E[G(X_i,X_j)^2]<\infty$.	
\end{lemma}

\begin{proof}[\rm Proof of Lemma \ref{lem:euclidean}]
Let $\mathcal F_g\equiv\{f:\mathcal X\times\mathcal X\to\mathbb R:f(x,x,p)=p'g(x,x'),p\in\mathbb B^\ell\}$, where $g:\mathcal X\times\mathcal X\to\mathbb R^\ell$ is a known function.  By the Cauchy-Schwarz inequality, for any $p,q\in\mathbb B^\ell$, we therefore have $|f(x,x',p)-f(x,x',p')|\le \|g(x,x')\|\|p-q\|.$ Hence, by Lemma 2.13 in \cite{Pakes:1989fk}, $\mathcal F_g$ is Euclidean with the envelope function $F_g(x,x')=g(x,x')'p_0+M\|g(x,x')\|$ for some $p_0\in\mathbb B^\ell$, where $M=2\sqrt \ell\sup_{p\in\mathbb B^\ell}\|p-p_0\|$, which can be further bounded from above by $4\sqrt \ell$. Hence, we may take the envelope function as $F_g(x,x')=(1+4\sqrt \ell)\|g(x,x')\|.$

By Lemma 2.4 in \cite{Pakes:1989fk}, the class of sets $\{x\in\mathcal X:p'l(z)>0\}$ is a VC-class, which in turn implies that the function classes $\mathcal F_{\phi_1}\equiv\{f:\mathcal X\times\mathcal X\to\mathbb R:f(x,x',p)=1\{p'l(z)>0\}\}$ and $\mathcal F_{\phi_2}\equiv\{f:\mathcal X\times\mathcal X\to\mathbb R:f(x,x',p)=1\{p'l(z')>0\}\}$ are Euclidean, where $x=(y_L,y_U,z)$ and $x'=(y_L',y_U',z')$ with envelope function $F_{\phi_j}(x,x')=1,j=1,2$.

Note that we may write
$$	\tilde p_n(x,x';p)=-p'\nabla_z K\left(\frac{z-z'}{h}\right)\{(y_{U}-y_{L})1\{p'l(z)>0\}-y_{L}-(y_{U}'-y_{L}')1\{p'l(z')>0\}+y_{L}'\}.$$
Hence, $\mathcal H_n$ can be written as the combination of classes of functions as $\mathcal H_n=\mathcal F_{g_1}\cdot\mathcal F_{\phi_1}+\mathcal F_{g_2}+\mathcal F_{g_3}\cdot\mathcal F_{\phi_2}+\mathcal F_{g_4},$ where
\begin{align*}
g_1(x,x')&=	-\nabla_z K\Big(\frac{z-z'}{h}\Big)(y_{U}-y_{L}),~~g_2(x,x')=	\nabla_z K\Big(\frac{z-z'}{h}\Big)y_{L}\\
g_3(x,x')&=\nabla_z K\Big(\frac{z-z'}{h}\Big)(y_{U}'-y_{L}'),~~g_4(x,x')=	-\nabla_z K\Big(\frac{z-z'}{h}\Big)y_{L}'.
\end{align*}
By Lemma 2.14 in \cite{Pakes:1989fk} and $\mathcal F_{\phi_1}$ and $\mathcal F_{\phi_2}$ having constant envelope functions, $\mathcal H_n$ is Euclidean with the envelope function $F_{g_1}+ F_{g_2}+F_{g_3}+F_{g_4}$. Hence, $\mathcal H_n$ is Euclidean with the envelope function $H(x,x')\equiv 8(1+4\sqrt \ell)\sup_{y\in D}|y|\sup_{h>0}\|\nabla_zK(\frac{z-z'}{h})\|$. By Assumption \ref{as:onK},  $E[\sup_{h>0}\|\nabla_z K(\frac{Z_i-Z_j}{h})\|^2]<\infty$, which in turn implies $E[H(X_i,X_j)^2]<\infty$. This shows the claim of the lemma for $\mathcal H_n$.
Showing $\mathcal G_n$ is Euclidean is similar. Hence, the rest of the proof is omitted. 
\end{proof}

\begin{lemma}\label{lem:G1}
Suppose Assumptions \ref{as:onP1}, \ref{as:onmu}, and \ref{as:onK} hold.  Suppose further that $\tilde h\to 0$ and $ n\tilde h^{4(\ell+1)}\to\infty$. Then, uniformly in $p\in \mathbb S^\ell$, $\hat \upsilon_n(p)-\bar \upsilon_n(p)=o_p(n^{-1/2}).$
\end{lemma}

\begin{proof}[\rm Proof of Lemma \ref{lem:G1}]
By \eqref{eq:shat}, we may write 
\begin{multline}
	\hat \upsilon_n(p)-\bar \upsilon_n(p)=\frac{1}{n}\sum_{i=1}^np'\hat l_{i,h}(Z_i)(\hat Y_{p,i}-Y_{p,i})\\
	=\frac{1}{n}\sum_{i=1}^np'\hat l_{i,h}(Z_i)[(1\{p'\hat l_{i,\tilde h}(Z_i)>0\}-1\{p' l(Z_i)>0\})(Y_{U,i}-Y_{L,i})]=\frac{1}{n} \sum_{i=1}^nu_{i,n}(p),
\end{multline}
where $u_{i,n}(p)=p'\hat l_{i,h}(Z_i)[(1\{p'\hat l_{i,\tilde h}(Z_i)>0\}-1\{p' l(Z_i)>0\})(Y_{U,i}-Y_{L,i})].$ For the conclusion of the lemma, it therefore suffices to show $E[\sup_{p\in\mathbb S^\ell}|u_{i,n}(p)|^2]=o(n^{-1})$.

By $\|p\|=1$, the Cauchy-Schwarz inequality, and Assumption \ref{as:onP1}, 
\begin{multline}
	E[\sup_{p\in\mathbb S^\ell}|u_{i,n}(p)|^2]\le 2\sup_{y\in D}|y|E[\|\hat l_{i,h}(Z_i)\|^2\sup_{p\in\mathbb S^\ell}|1\{p'\hat l_{i,\tilde h}(Z_i)>0\}-1\{p' l(Z_i)>0|^2]\\
	\le 2\sup_{y\in D}|y|E[\|\hat l_{i,h}(Z_i)\|^4]^{1/2}P(\text{\text{sgn}}(p'\hat l_{i,\tilde h}(Z_i))\ne \text{sgn}(p' l(Z_i)),\exists p\in\mathbb S^\ell),\label{eq:G1_1}
\end{multline}
where $E[\|\hat l_{i, h}(Z_i)\|^4]<\infty$ under our choice of $h$. Hence, for the desired result, it suffices to show that $P(\text{sgn}(p'\hat l_{i,\tilde h}(Z_i))\ne \text{sgn}(p' l(Z_i)),\exists p\in\mathbb S^\ell)=o(n^{-1})$. By Assumption \ref{as:onmu}, it follows that
\begin{multline}
	P(\text{sgn}(p'\hat l_{i,\tilde h}(Z_i))	\ne \text{sgn}(p' l(Z_i)),\exists p\in\mathbb S^\ell)
	\\\le P(p'\hat l_{i,\tilde h}(Z_i))>0\text{ and }p' l(Z_i)<0,\exists p\in\mathbb S^\ell)+P(p'\hat l_{i,\tilde h}(Z_i))<0\text{ and }p' l(Z_i)>0,\exists p\in\mathbb S^\ell).\label{eq:G1_2}
\end{multline}
Without loss of generality, suppose that $p'\hat l_{i,\tilde h}(Z_i))>0\text{ and }p' l(Z_i)<0$ for some $p\in\mathbb S^\ell$. Then, there must exist $\epsilon>0$ such that
$\sup_{p\in\mathbb S^\ell}|p'\hat l_{i,\tilde h}(Z_i)-E[p'\hat l_{i,\tilde h}(Z_i)]+E[p'\hat l_{i,\tilde h}(Z_i))]+p' l(Z_i)|>\epsilon$. This is also true if $p'\hat l_{i,\tilde h}(Z_i))<0\text{ and }p' l(Z_i)>0$.
Therefore, by the triangle inequality and the law of iterated expectations, we may write
\begin{multline}
P(\text{sgn}(p'\hat l_{i,\tilde h}(Z_i))	\ne \text{sgn}(p' l(Z_i)),\exists p\in\mathbb S^\ell)\\
\le 2\Big\{E\Big[P\big(\sup_{p\in\mathbb S^\ell}|p'\hat l_{i,\tilde h}(Z_i)-E[p'\hat l_{i,\tilde h}(Z_i)|Z_{i}])|>\epsilon/2|Z_i\big)	+P\big(\sup_{p\in\mathbb S^\ell}|E[p'\hat l_{i,\tilde h}(Z_i)|Z_{i}]-p' l(Z_i)|>\epsilon/2|Z_i\big)\Big]\Big\},\label{eq:G1_3}	
\end{multline}
where the second term in \eqref{eq:G1_3} vanishes for all $n$ sufficiently large because   the bias satisfies $\|E[\hat l_{i,\tilde h}(Z_i)|Z_{i}]- l(Z_i)\|\to 0$ with probability 1 as $\tilde h\to0$. Hence, we focus on controlling the first term in \eqref{eq:G1_3} below.

Let $\bar M\equiv \sup_{z\in\mathcal Z}\|\nabla_z K(z)\|$ and define 
\begin{align}
W_n(p)&\equiv \frac{\tilde h^{(\ell+1)}}{2\bar M}(	p'\hat l_{i,\tilde h}(z)-E[p'\hat l_{i,\tilde h}(Z_i)|Z_{i}=z])
=\frac{1}{2\bar M(n-1)}\sum_{j=1,j\ne i}^np'	\Big\{\nabla_z K\Big(\frac{z-Z_j}{\tilde h}\Big)-E[\nabla_z K\Big(\frac{z-Z_j}{\tilde h}\Big)]\Big\}\\
\bar\sigma^2&\equiv E\Big[\sup_{p\in\mathbb S^\ell}(\frac{1}{2\bar M(n-1)}\sum_{j=1,j\ne i}^np'	\Big\{\nabla_z K\Big(\frac{z-Z_j}{\tilde h}\Big)-E[\nabla_z K\Big(\frac{z-Z_j}{\tilde h}\Big)]\Big\})^2\Big].\label{eq:barsigma}
\end{align}
Define $\mathcal W\equiv\{f:\mathcal X\to\mathbb R:f(z_j)=\frac{1}{2\bar M}p'	\{\nabla_z K(\frac{z-z_j}{\tilde h})-E[\nabla_z K(\frac{z-Z_j}{\tilde h})]\},p\in\mathbb S^\ell\}.$ Then by $\mathbb S^\ell$ being finite dimensional and Lemma 2.6.15 in \cite{Vaart_Wellner2000aBK}, $\mathcal W$ is a VC-subgraph class, which in turn implies that $\sup_{Q}N(\epsilon,\mathcal W,L_2(Q))\le (\frac{K}{\epsilon})^{V}$ for all $0<\epsilon<K$ for some positive constants $V$ and $K$ by  Lemman 2.6.7 in \cite{Vaart_Wellner2000aBK}.
Then, by $W_n$ being independent of $Z_i$ and  Theorem 2.14.16 in \cite{Vaart_Wellner2000aBK}, we have
\begin{multline}
P(\sup_{p\in\mathbb S^\ell}|p'\hat l_{i,\tilde h}(Z_i)-E[p'\hat l_{i,\tilde h}(Z_i)|Z_{i}])|>\epsilon/2|Z_i=z)= P(\|W_n\|_{\mathcal W}>\frac{\epsilon \tilde h^{(\ell+1)}}{4\bar M})	\\
\le 
C\Big(\frac{1}{\bar\sigma}\Big)^{2V}\Big(1\vee \frac{\epsilon \tilde h^{(\ell+1)}}{4\bar M\bar\sigma}\Big)^{3V+1}\exp\Big(-\frac{1}{2}\frac{(\frac{\epsilon \tilde h^{(\ell+1)}}{4\bar M})^2}{\bar\sigma^2+(3+\frac{\epsilon \tilde h^{(\ell+1)}}{4\bar M})/\sqrt n}\Big),\label{eq:G1_4}
\end{multline}
where  $C$ is a constant that depends on $V$ and $K$. Note that under the imposed conditions on $\tilde h$, we have
\begin{align}
\frac{(\frac{\epsilon \tilde h^{(\ell+1)}}{4\bar M})^2}{\bar\sigma^2+(3+\frac{\epsilon \tilde h^{(\ell+1)}}{4\bar M})/\sqrt n}=\frac{1}{S_{1,n}+S_{2,n}},	
\end{align}
where
$S_{1,n}\equiv \bar\sigma^2/(\frac{\epsilon \tilde h^{(\ell+1)}}{4\bar M})^2$ and $S_{2,n}\equiv (3+\frac{\epsilon \tilde h^{(\ell+1)}}{4\bar M})/[(\frac{\epsilon \tilde h^{(\ell+1)}}{4\bar M})^2\sqrt n]$.
By \eqref{eq:barsigma}, $\bar\sigma^2=o(n^{-2}),$ which implies that $S_{1,n}=o(1/(n\tilde h^{(\ell+1)})^2)=o(1)$ under our choice of $\tilde h$.
Further, under our assumption, $\sqrt n\tilde h^{2(\ell+1)}\to\infty,$ which in turn implies $S_{2,n}=o(1)$.
This ensures that, by \eqref{eq:G1_4}, $P(\sup_{p\in\mathbb S^\ell}|p'\hat l_{i,\tilde h}(Z_i)-E[p'\hat l_{i,\tilde h}(Z_i)|Z_{i}])|>\epsilon/2|Z_i=z)$ decays exponentially  as $n\to\infty$. Hence, by \eqref{eq:G1_1}, \eqref{eq:G1_2}, and \eqref{eq:G1_3}, we have $E[\sup_{p\in\mathbb S^\ell}|u_{i,n}(p)|^2]=o(n^{-1})$ as desired. This establishes the claim of the Lemma.
\end{proof}

\begin{lemma}\label{lem:pss3.1}
Let	$U_n(p)\equiv {n \choose 2}^{-1}\sum_{i=1}^{n-1}\sum_{j=i+1}^np_n(X_i,X_j;p)$ and $\hat U_n(p)=\frac{2}{n}\sum_{i=1}^n r_n(X_i;p).$ 
Suppose Assumptions \ref{as:onY} and \ref{as:onK}  hold.
Suppose further that $nh^{\ell+2+\delta}\to\infty$ for some $\delta>0$ as $h\to0$.
 Then, $\sqrt n(\hat U_n(p)-U_n(p))=o_p(1)$ uniformly in $p\in\mathbb S^\ell.$ 
\end{lemma}

\begin{proof}[\rm Proof of Lemma \ref{lem:pss3.1}]
 Following the same argument as in the proof of Lemma 3.1 in \cite{PowellStockStoker1989EJES}, we may write
\begin{align}
\hat U_n(p)-U_n(p)={n\choose 2}^{-1}\sum_{i=1}^{n-1}\sum_{j=i+1}^nq_n(X_i,X_j;p),	\label{eq:uhatun}
\end{align}
 where $q_n(x_i,x_j;p)=p_n(x_i,x_j;p)-r_n(x_i,p)-r_n(x_j;p)-E[r_n(X_i;p)]$. Recall that $\tilde q_n=h^{(\ell+1)}q_n.$
By \eqref{eq:pnrn}, we may then obtain the following bound:
\begin{multline}
	E[\sup_{p\in\mathbb S^\ell}|\tilde q_n(X_i,X_j;p)|^2]\le 16 E[\sup_{p\in\mathbb S^\ell}|\tilde p_n(X_i,X_j;p)|^2]\le 64(\sup_{y\in D}|y|)^2E\Big[\|\nabla_z K((Z_i-Z_j)/h) \|^2\Big]\\
	\le 64(\sup_{y\in D}|y|)^2 h^{\ell}\int \|\nabla_z K(u) \|^2f(z_i)f(z_i+hu)dz_idu=O(h^{\ell}),
\end{multline}
where the second inequality follows from Assumption \ref{as:onY}, $\|p\|=1$ for all $p$, and the Cauchy-Schwarz inequality, while the third inequality uses the change of variables from $(z_i,z_j)$ to $(z_i,u=(z_i-z_j)/h)$ with Jacobian $h^{-\ell}$.  
By Lemma \ref{lem:euclidean}, $\mathcal G_n$ is Euclidean. 
By Theorem 3 in \cite{Sherman:1994fk} applied with $\delta_n=1$ and $\gamma_n^2=h^{\ell}$, it then follows that for some $0<\alpha<1$, which can be made arbitrarily close to 1, we have
\begin{multline}
	{n\choose 2}^{-1}\sum_{i=1}^{n-1}\sum_{j=i+1}^nq_n(X_i,X_j;p)=h^{-(\ell+1)}{n\choose 2}^{-1}\sum_{i=1}^{n-1}\sum_{j=i+1}^n\tilde q_n(X_i,X_j;p)\\
	\le O(h^{-(\ell+1)})O_p(h^{\alpha\ell/2}/n)=O_p(h^{\ell(\frac{\alpha}{2}-1)-1}/n)
\end{multline}
uniformly over $\mathbb B^\ell$. Since $\alpha$ can be made arbitrarily small and $nh^{\ell+2+\delta}\to\infty$, we have $h^{\ell(\frac{\alpha}{2}-1)-1}\le O(h^{-\frac{\ell}{2}-1-\frac{\delta}{2}})=o(\sqrt n)$. Therefore, $	{n\choose 2}^{-1}\sum_{i=1}^{n-1}\sum_{j=i+1}^nq_n(X_i,X_j;p)=o(n^{-1/2}).$
This, together with $\mathbb S^\ell\subset\mathbb B^\ell$ and \eqref{eq:uhatun}, establishes the claim of the lemma.
\end{proof}

\begin{lemma}\label{lem:G2}
Suppose Assumptions \ref{as:onY}-\ref{as:onP1}, \ref{as:onmu}, and \ref{as:onf_est}-\ref{as:onK} hold.	
	Suppose further that $nh^{\ell+2+\delta}\to\infty$ for some $\delta>0$ as $h\to0$. Then, uniformly in $p\in \mathbb S^\ell$, $\sqrt n(\bar \upsilon_n(p)-E[\bar \upsilon_n(p)])=\frac{1}{\sqrt n}\sum_{i=1}^n\psi_p(Z_i)+o_p(1)$.
\end{lemma}

\begin{proof}[\rm Proof of Lemma \ref{lem:G2}]
	Note that, for each $p\in\mathbb S^\ell$, one may write	\begin{multline}
\bar \upsilon_n(p)= \frac{1}{n}\sum_{i=1}^n p'\hat l_{i,h}(Z_{i})Y_{p,i}=\frac{-2}{n}\sum_{i=1}^n p'\nabla_z\hat f_{i,h}(Z_i)Y_{p,i}=\frac{-2}{n(n-1)h^{\ell+1}}\sum_{i=1}^n\sum_{j=1,j\ne i}^n	p'\nabla_z K\Big(\frac{Z_i-Z_j}{h}\Big)Y_{p,i}\\
=-{n\choose 2}^{-1}\sum_{i=1}^{n-1}\sum_{j=i+1}^n\Big(\frac{1}{h}\Big)^{\ell+1}p'\nabla_zK\Big(\frac{Z_i-Z_j}{h}\Big)(Y_{p,i}-Y_{p,j})={n \choose 2}^{-1}\sum_{i=1}^{n-1}\sum_{j=i+1}^np_n(X_i,X_j;p).\label{eq:G2_0}
	\end{multline}
 By Lemma \ref{lem:pss3.1} and $nh^{\ell+2+\delta}\to\infty$, it then follows that
	\begin{multline}
		\sqrt n(\bar \upsilon_n(p)-E[\bar \upsilon_n(p)])=\sqrt n({n \choose 2}^{-1}\sum_{i=1}^{n-1}\sum_{j=i+1}^np_n(X_i,X_j;p)-E[p_n(X_i,X_j;p)])\\
		=\frac{2}{\sqrt n}\sum_{i=1}^nr_n(X_i;p)-E[r_n(X_i;p)]+o_p(1),\label{eq:est1}
	\end{multline}
	uniformly over $\mathbb S^\ell$. By Eq. (3.15) in \cite{PowellStockStoker1989EJES}, we may write
$r_n(x;p)=\psi_p(x)+t_n(x;p)$ with
	\begin{align}
	t_n(x;p)\equiv 	\int \{\nabla_z (m_{p} f)(z+hu)- \nabla_z (m_{p} f)(z)\}K(u)du-y_{p}\int\{\nabla_z f(z+hu)- \nabla_z f(z)\}K(u)du,
	\end{align}
Hence, for the conclusion of the lemma, it suffices to show $t_n(x,\cdot)$ converges in probability to 0 uniformly in $p\in\mathbb S^\ell$.	
By Assumptions \ref{as:onmu}, \ref{as:onf_est} and $m_p(z)=1\{p'l(z)\le 0\}m_L(z)+1\{p'l(z)\le 0\}m_U(z)$, we have uniformly in $p\in\mathbb S^\ell$,
\begin{align}
|\nabla_z (m_{p} f)(z+hu)- \nabla_z (m_{p} f)(z)|\le 2 M(z)\|hu\|.	\label{eq:G2_1}
\end{align}
Further, by Assumptions \ref{as:onY} and \ref{as:onf_est},  uniformly in $p\in\mathbb S^\ell$,
\begin{align}
|	y_p||\nabla_z f(z+hu)- \nabla_z f(z)|\le \sup_{y\in \mathcal D}|y|\times M(z)\|hu\|.\label{eq:G2_2}
\end{align}
Assumption \ref{as:onf_est} and \eqref{eq:G2_1}-\eqref{eq:G2_2} then imply $E[\sup_{p\in\mathbb S^\ell}|	t_n(x;p)|^2]\le |h|^2(2+\sup_{y\in D}|y|)^2 E[|M(Z)|^2](\int \|u\|^|K(u)|du)^2=O(h^2),$ which in turn implies $\frac{1}{\sqrt n}\sum_{i=1}^nt_n(x,\cdot)-E[t_n(x,\cdot)]$ converges in probability to 0 uniformly in $p\in\mathbb S^\ell$. This establishes the claim of the lemma. 
\end{proof}

\begin{lemma}\label{lem:G3}
Suppose Assumptions \ref{as:onf_est}-\ref{as:onK} hold. Suppose that $nh^{2J}\to 0$. Then, uniformly in $p\in \mathbb S^\ell$, $E[\bar \upsilon_n(p)]-\upsilon(p,\Theta_0(P_0))=o(n^{-1/2}).$
\end{lemma}

\begin{proof}[\rm Proof of Lemma \ref{lem:G3}]
	The proof is based on the proof of Theorem 3.2 in \cite{PowellStockStoker1989EJES}. Hence, we briefly sketch the argument.
	By \eqref{eq:G2_0}, the law of iterated expectations, and arguing as in (3.19) in \cite{PowellStockStoker1989EJES}, we obtain
	\begin{multline}
E[\bar\upsilon_n(p)]=-2E\Big[\Big(\frac{1}{h}\Big)^{\ell+1}	p'\nabla_z K\Big(\frac{Z_i-Z_j}{h}\Big)Y_{p,i}\Big]\\
=\frac{-2}{h}\int\int p'\nabla_z K(u)m_p(z)f(z)f(z+hu)dzdu=\frac{-2}{h}\int\int  K(u)m_p(z)f(z)p'\nabla_zf(z+hu)dzdu,\label{eq:G3_1}
	\end{multline}
	where the second equality follows from change of variables.
By Assumptions  \ref{as:onf_est}, \ref{as:onK}, and Young's version of Taylor's theorem, for each $p\in\mathbb S^\ell$, we then obtain the expansion:
\begin{align}
	\sqrt n(E[\bar\upsilon_n(p)]-\upsilon(p,\Theta_0(P_0)))=b_1(p)\sqrt n h+b_2(p)\sqrt nh^2+\cdots+b_{J-1}(p)\sqrt n h^{J-1}+O(\sqrt n h^{J}),\label{eq:G3_2}
\end{align}
where $b_k$ is given by
\begin{align}
b_k(p)=	\frac{-2}{k!}\sum_{j_1,\cdots,j_k}^k\int u^{j_1}\cdots u^{j_k}K(u)du\times \int m_p(z)\sum_{i=1}^\ell p^{(i)}\frac{\partial^{k+1}f(z)}{\partial z_{j_1}\cdots\partial z_{j_k}\partial z_i}f(z)dz,~p\in\mathbb S^\ell,~k=1,\cdots,J,\label{eq:G3_3}
\end{align}
which shows that the map $p\mapsto b_k(p)$ is continuous on $\mathbb S^\ell$ for $k=1,2,\cdots,J.$ This implies that the expansion in \eqref{eq:G3_2} is valid uniformly over the compact set $\mathbb S^\ell$. By Assumption \ref{as:onK} (v) and \eqref{eq:G3_3}, $b_k(p)=0$ for all $k\le J$ but $b_k\ne 0$ for $k=J$. By the hypothesis that $nh^{2J}\to 0$, we obtain $\sqrt n(E[\bar\upsilon_n(p)]-\upsilon(p,\Theta_0(P_0)))=O(\sqrt n h^{J})=o(1)$. This establishes the claim of the lemma.
\end{proof}

\begin{lemma}\label{lem:donsker}
Suppose Assumptions \ref{as:onY}-\ref{as:onP1}, and \ref{as:onP2} hold. Then, $\mathcal F\equiv \{\psi_p:\mathcal X\to\mathbb R:\psi_p(x)=w(z)p'\nabla_zm_{p}(z)-\upsilon(p,\Theta_0(P))+p'l(z)\zeta_{p}(x)\}$ is Donsker in $\mathcal C(\mathbb S^\ell).$
\end{lemma}

\begin{proof}[\rm Proof of Lemma \ref{lem:donsker}]
Let $\mathcal F_g\equiv\{f:\mathcal X\to\mathbb R:f(x)=p'g(x),p\in\mathbb S^\ell\}$, where $g:\mathcal X\to\mathbb R^\ell$ is a known function. Then by $\mathbb S^\ell$ being finite dimensional and Lemma 2.6.15 in \cite{Vaart_Wellner2000aBK}, $\mathcal F_g$ is a VC-subgraph class of index $\ell+2$ with an envelope function $F(x)\equiv\|g(x)\|$. Define
\begin{align}
g_1(x)&\equiv w(z)(\nabla_zm_U(z)-\nabla_z m_L(z)),~~g_2(x)\equiv w(z)\nabla_z m_L(z),\\
g_3(x)&\equiv l(z)\{r^{-1}_U(z)q(y_U-m_U(z))-r^{-1}_L(z)q(y_L-m_L(z))\},~~
g_4(x)\equiv l(z)r^{-1}_L(z)q(y_L-m_L(z)),~~g_5(x)\equiv l(z).
\end{align}
Then $\mathcal F_{g_j},j=1,\cdots, 5$ are VC-subgraph classes. Further, let $\mathcal F_{\upsilon}\equiv\{f:\mathcal X\to\mathbb R:f(x)=\upsilon(p,\Theta_0(P)),p\in\mathbb S^\ell\}$. This is also finite dimensional. Hence, $\mathcal F_{\upsilon}$ is a VC-subgraph class. Finally, let $\mathcal F_{\phi}\equiv\{f:\mathcal X\to\mathbb R:1\{p'l(z)>0\},p\in\mathbb S^\ell\}$. Then, $F_{\phi}=\phi\circ\mathcal F_{g_5},$ where $\phi:\mathbb R\to\mathbb R$ is the monotone map $\phi(w)=1\{w>0\}.$ By Lemma 2.6.18 in \cite{Vaart_Wellner2000aBK}, $\mathcal F_{\phi}$ is also a VC-subgraph class.

Note that $\psi_p$ can be written as
\begin{multline}
	\psi_p(x)=w(z)p'\{1\{p'l(z)> 0\}(\nabla_zm_U(z)-\nabla_z m_L(z))+\nabla_z m_L(z)\}-\upsilon(p,\Theta_0(P))\\
	+p'l(z)\{1\{p'l(z)> 0\}\{r^{-1}_U(z)q(y_U-m_U(z))-r^{-1}_L(z)q(y_L-m_L(z))\}+r^{-1}_L(z)q(y_L-m_L(z))\}.
\end{multline}	
Therefore, $\mathcal F=\mathcal F_{g_1}\cdot\mathcal F_\phi+\mathcal F_{g_2}+(-\mathcal F_\upsilon)+\mathcal F_{g_3}\cdot\mathcal F_\phi+\mathcal F_{g_4}$, which is again a VC-subgraph class with some index $V(\mathcal F)$ by Lemma 2.6.18 in \cite{Vaart_Wellner2000aBK}. By Assumptions \ref{as:onY}-\ref{as:onP1} and \ref{as:onP2}, we may take $F(x)\equiv \sup_{p\in\mathbb S^\ell}\|w(z)p'\nabla_{z}m_{p,\eta}(z)\|_{L^\infty_\mu}+ \|l(z)\|\times \bar\epsilon^{-1}\times\sup_{u\in D'} |q(u)|$ as an envelope function such that  $E[F(x)^2]<\infty.$
Then, by Theorems 2.6.7 and 2.5.1 in \cite{Vaart_Wellner2000aBK}, $\mathcal F$ is a Donsker class. This establishes the claim of the lemma.
\end{proof}

\begin{proof}[\rm Proof of Theorem \ref{thm:estimator}]
	For each $p\in\mathbb S^\ell$, we have the following decomposition:
	\begin{align} 
\sqrt n(\hat \upsilon_n(p)-\upsilon(p,\Theta_0(P_0)))&=\sqrt n(\hat \upsilon_n(p)-\bar \upsilon_n(p))+\sqrt n(\bar \upsilon_n(p)-E[\bar \upsilon_n(p)])+\sqrt n(E[\bar \upsilon_n(p)]-\upsilon(p,\Theta_0(P_0)))\\
&\equiv G_{1n}(p)+G_{2n}(p)+G_{3n}(p).\label{eq:effinfexp}
	\end{align} 
By Lemmas \ref{lem:G1}-\ref{lem:G3},   uniformly in $p\in\mathbb S^\ell$, $G_{1n}(p)=G_{3n}(p)=o_p(1)$, and $G_{2n}(p)=	\frac{1}{\sqrt n}\sum_{i=1}^n\psi_p(Z_i)+o_p(1)$. This establishes the second claim of the Theorem. 
By Theorem \ref{thm:InvCov}, $\psi_p$ is the efficient influence function, and hence regularity of $\{\hat\upsilon_n(\cdot)\}$ follows from Lemma \ref{lem:donsker} and Theorem 18.1 in \cite{Kosorok2008}, which establishes the first claim. 
The stated convergence in distribution is then immediate from \eqref{eq:effinfexp} and Lemma \ref{lem:donsker}, while the limiting process having the efficient covariance kernel is a direct result of the characterization of $I(p_1,p_2)$ obtained in Theorem \ref{thm:InvCov}, which establishes the third claim.
\end{proof}
}

\vspace{0.1in}
\begin{center}
{\sc {Appendix E}}: Figures and Tables
\end{center}

\renewcommand{\thedefinition}{E.\arabic{definition}}
\renewcommand{\theequation}{E.\arabic{equation}}
\renewcommand{\thelemma}{E.\arabic{lemma}}
\renewcommand{\thecorollary}{E.\arabic{corollary}}
\renewcommand{\thetheorem}{E.\arabic{theorem}}
\setcounter{lemma}{0}
\setcounter{theorem}{0}
\setcounter{corollary}{0}
\setcounter{equation}{0}
\setcounter{remark}{0}

\begin{figure}[htbp]
\begin{center}
\includegraphics[scale=.45]{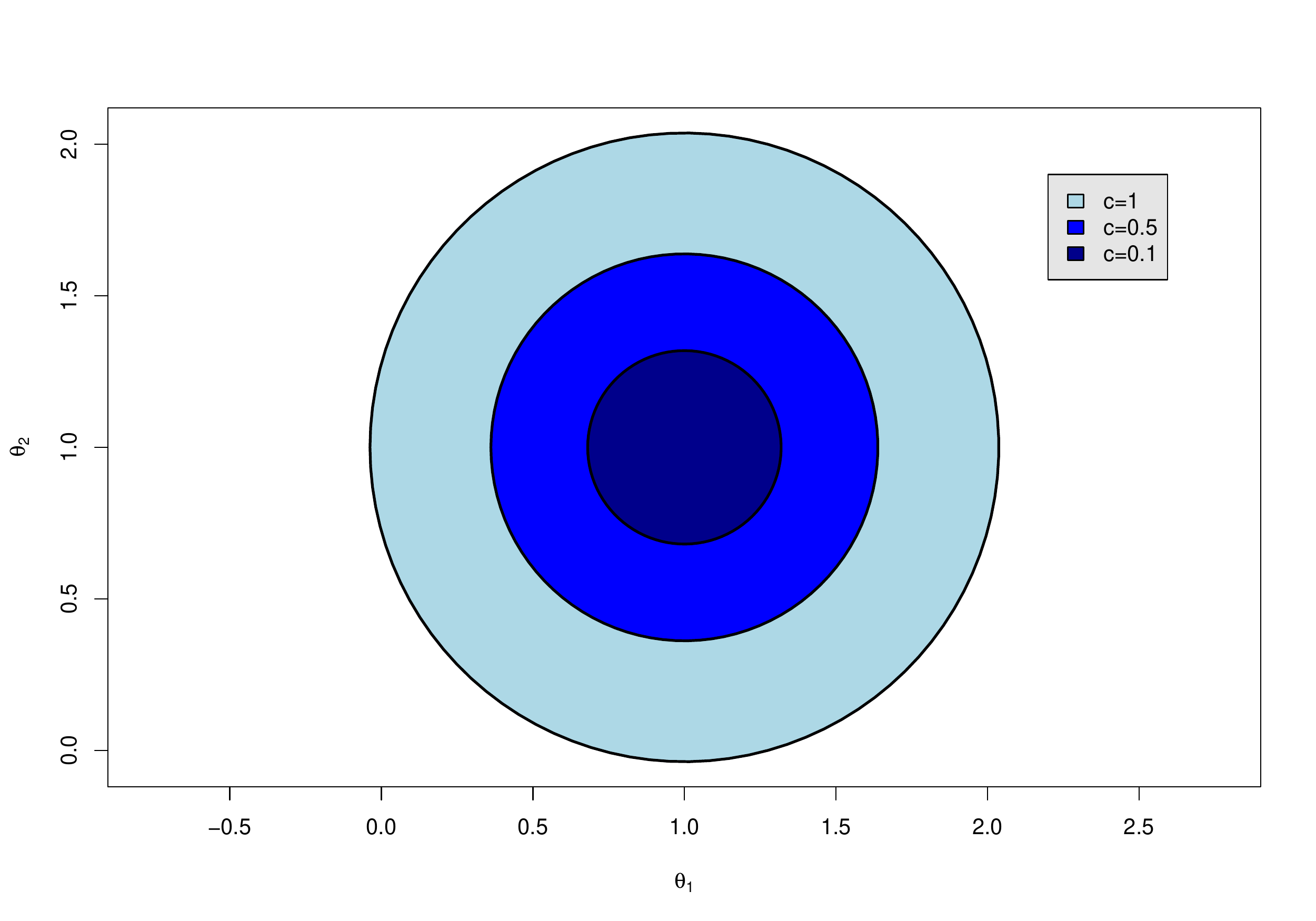}
\end{center}
\caption{Identified sets for the density weighted average derivatives}
\label{fig:identifiedset}
\end{figure}

\begin{table}[htbp]\small
\caption{Risk of $\hat \upsilon^{IV}_n$ (Gaussian kernel)}
\begin{center}
\begin{tabular}{lrrrrlrrrlrrr}
\hline
\hline
 & \multicolumn{1}{l}{} & \multicolumn{ 3}{c}{$c$=0.1} &  & \multicolumn{ 3}{c}{$c$=0.5} &  & \multicolumn{ 3}{c}{$c$=1} \\ 
Sample Size & \multicolumn{1}{l}{$h$} & \multicolumn{1}{l}{$R_H$} & \multicolumn{1}{l}{$R_{IH}$} & \multicolumn{1}{l}{$R_{OH}$} &  & \multicolumn{1}{l}{$R_H$} & \multicolumn{1}{l}{$R_{IH}$} & \multicolumn{1}{l}{$R_{OH}$} &  & \multicolumn{1}{l}{$R_H$} & \multicolumn{1}{l}{$R_{IH}$} & \multicolumn{1}{l}{$R_{OH}$} \\ 
\hline
$n$=1000 & \multicolumn{1}{l}{} & \multicolumn{1}{l}{} & \multicolumn{1}{l}{} & \multicolumn{1}{l}{} &  & \multicolumn{1}{l}{} & \multicolumn{1}{l}{} & \multicolumn{1}{l}{} &  & \multicolumn{1}{l}{} & \multicolumn{1}{l}{} & \multicolumn{1}{l}{} \\ 
 & 0.4 & 0.0608 & 0.0477 & 0.0600 &  & 0.0834 & 0.0709 & 0.0673 &  & 0.1229 & 0.1037 & 0.0801 \\ 
 & 0.5 & 0.0588 & 0.0468 & 0.0578 &  & 0.0785 & 0.0749 & 0.0485 &  & 0.1212 & 0.1185 & 0.0437 \\ 
 & 0.6 & 0.0572 & 0.0452 & 0.0564 &  & 0.0809 & 0.0804 & 0.0351 &  & 0.1305 & 0.1304 & 0.0229 \\ 
 & 0.7 & 0.0567 & 0.0416 & 0.0563 &  & 0.0844 & 0.0844 & 0.0263 &  & 0.1416 & 0.1416 & 0.0086 \\ 
 & 0.8 & 0.0555 & 0.0386 & 0.0553 &  & 0.0882 & 0.0882 & 0.0195 &  & 0.1556 & 0.1556 & 0.0026 \\ 
 & \multicolumn{1}{l}{} & \multicolumn{1}{l}{} & \multicolumn{1}{l}{} & \multicolumn{1}{l}{} &  & \multicolumn{1}{l}{} & \multicolumn{1}{l}{} & \multicolumn{1}{l}{} &  & \multicolumn{1}{l}{} & \multicolumn{1}{l}{} & \multicolumn{1}{l}{} \\ 
$n$=500 & \multicolumn{1}{l}{} & \multicolumn{1}{l}{} & \multicolumn{1}{l}{} & \multicolumn{1}{l}{} &  & \multicolumn{1}{l}{} & \multicolumn{1}{l}{} & \multicolumn{1}{l}{} &  & \multicolumn{1}{l}{} & \multicolumn{1}{l}{} & \multicolumn{1}{l}{} \\ 
 & 0.4 & 0.0929 & 0.0703 & 0.0919 &  & 0.1185 & 0.0877 & 0.1072 &  & 0.1731 & 0.1203 & 0.1437 \\ 
 & 0.5 & 0.0836 & 0.0684 & 0.0817 &  & 0.1091 & 0.0979 & 0.0839 &  & 0.1555 & 0.1414 & 0.0873 \\ 
 & 0.6 & 0.0799 & 0.0646 & 0.0786 &  & 0.1038 & 0.0999 & 0.0640 &  & 0.1555 & 0.1520 & 0.0530 \\ 
 & 0.7 & 0.0774 & 0.0607 & 0.0762 &  & 0.1060 & 0.1051 & 0.0512 &  & 0.1679 & 0.1677 & 0.0297 \\ 
 & 0.8 & 0.0775 & 0.0592 & 0.0769 &  & 0.1098 & 0.1096 & 0.0410 &  & 0.1785 & 0.1785 & 0.0173 \\ 
 & \multicolumn{1}{l}{} & \multicolumn{1}{l}{} & \multicolumn{1}{l}{} & \multicolumn{1}{l}{} &  & \multicolumn{1}{l}{} & \multicolumn{1}{l}{} & \multicolumn{1}{l}{} &  & \multicolumn{1}{l}{} & \multicolumn{1}{l}{} & \multicolumn{1}{l}{} \\ 
$n$=250 & \multicolumn{1}{l}{} & \multicolumn{1}{l}{} & \multicolumn{1}{l}{} & \multicolumn{1}{l}{} &  & \multicolumn{1}{l}{} & \multicolumn{1}{l}{} & \multicolumn{1}{l}{} &  & \multicolumn{1}{l}{} & \multicolumn{1}{l}{} & \multicolumn{1}{l}{} \\ 
 & 0.4 & 0.1357 & 0.0960 & 0.1349 &  & 0.1820 & 0.1061 & 0.1770 &  & 0.2480 & 0.1256 & 0.2339 \\ 
 & 0.5 & 0.1189 & 0.0941 & 0.1169 &  & 0.1517 & 0.1231 & 0.1289 &  & 0.2013 & 0.1638 & 0.1446 \\ 
 & 0.6 & 0.1133 & 0.0914 & 0.1112 &  & 0.1413 & 0.1299 & 0.1053 &  & 0.1954 & 0.1818 & 0.1084 \\ 
 & 0.7 & 0.1121 & 0.0910 & 0.1098 &  & 0.1365 & 0.1317 & 0.0890 &  & 0.1974 & 0.1949 & 0.0725 \\ 
 & 0.8 & 0.1086 & 0.0864 & 0.1068 &  & 0.1374 & 0.1360 & 0.0737 &  & 0.2069 & 0.2061 & 0.0500 \\ 
\hline
\end{tabular}
\end{center}
\label{tab:mc_spec3}
\end{table}

\begin{table}[htbp]\small
\caption{Risk of $\hat \upsilon^{IV}_n$ (Higher-order kernel)}
\begin{center}
\begin{tabular}{lrrrrlrrrlrrr}
\hline
\hline
 & \multicolumn{1}{l}{} & \multicolumn{ 3}{c}{$c$=0.1} &  & \multicolumn{ 3}{c}{$c$=0.5} &  & \multicolumn{ 3}{c}{$c$=1} \\ 
Sample Size & \multicolumn{1}{l}{$h$} & \multicolumn{1}{l}{$R_H$} & \multicolumn{1}{l}{$R_{IH}$} & \multicolumn{1}{l}{$R_{OH}$} &  & \multicolumn{1}{l}{$R_H$} & \multicolumn{1}{l}{$R_{IH}$} & \multicolumn{1}{l}{$R_{OH}$} &  & \multicolumn{1}{l}{$R_H$} & \multicolumn{1}{l}{$R_{IH}$} & \multicolumn{1}{l}{$R_{OH}$} \\ 
\hline
$n$=1000 & \multicolumn{1}{l}{} & \multicolumn{1}{l}{} & \multicolumn{1}{l}{} & \multicolumn{1}{l}{} &  & \multicolumn{1}{l}{} & \multicolumn{1}{l}{} & \multicolumn{1}{l}{} &  & \multicolumn{1}{l}{} & \multicolumn{1}{l}{} & \multicolumn{1}{l}{} \\ 
 & 0.5 & 0.0722 & 0.0549 & 0.0714 &  & 0.1267 & 0.0461 & 0.1256 &  & 0.2038 & 0.0494 & 0.2017 \\ 
 & 0.6 & 0.0654 & 0.0551 & 0.0637 &  & 0.0912 & 0.0532 & 0.0872 &  & 0.1384 & 0.0636 & 0.1312 \\ 
 & 0.7 & 0.0600 & 0.0511 & 0.0583 &  & 0.0760 & 0.0631 & 0.0645 &  & 0.1020 & 0.0835 & 0.0745 \\ 
 & 0.8 & 0.0564 & 0.0470 & 0.0553 &  & 0.0759 & 0.0741 & 0.0444 &  & 0.1093 & 0.1085 & 0.0370 \\ 
 & 0.9 & 0.0565 & 0.0446 & 0.0559 &  & 0.0802 & 0.0801 & 0.0313 &  & 0.1302 & 0.1302 & 0.0134 \\ 
 & \multicolumn{1}{l}{} & \multicolumn{1}{l}{} & \multicolumn{1}{l}{} & \multicolumn{1}{l}{} &  & \multicolumn{1}{l}{} & \multicolumn{1}{l}{} & \multicolumn{1}{l}{} &  & \multicolumn{1}{l}{} & \multicolumn{1}{l}{} & \multicolumn{1}{l}{} \\ 
$n$=500 & \multicolumn{1}{l}{} & \multicolumn{1}{l}{} & \multicolumn{1}{l}{} & \multicolumn{1}{l}{} &  & \multicolumn{1}{l}{} & \multicolumn{1}{l}{} & \multicolumn{1}{l}{} &  & \multicolumn{1}{l}{} & \multicolumn{1}{l}{} & \multicolumn{1}{l}{} \\ 
 & 0.5 & 0.1104 & 0.0744 & 0.1101 &  & 0.1867 & 0.0587 & 0.1861 &  & 0.2887 & 0.0604 & 0.2870 \\ 
 & 0.6 & 0.0947 & 0.0753 & 0.0930 &  & 0.1308 & 0.0745 & 0.1267 &  & 0.1993 & 0.0857 & 0.1914 \\ 
 & 0.7 & 0.0869 & 0.0737 & 0.0846 &  & 0.1080 & 0.0843 & 0.0970 &  & 0.1453 & 0.1085 & 0.1184 \\ 
 & 0.8 & 0.0802 & 0.0668 & 0.0783 &  & 0.1019 & 0.0958 & 0.0747 &  & 0.1373 & 0.1308 & 0.0683 \\ 
 & 0.9 & 0.0772 & 0.0635 & 0.0758 &  & 0.1042 & 0.1034 & 0.0564 &  & 0.1513 & 0.1508 & 0.0372 \\ 
 & \multicolumn{1}{l}{} & \multicolumn{1}{l}{} & \multicolumn{1}{l}{} & \multicolumn{1}{l}{} &  & \multicolumn{1}{l}{} & \multicolumn{1}{l}{} & \multicolumn{1}{l}{} &  & \multicolumn{1}{l}{} & \multicolumn{1}{l}{} & \multicolumn{1}{l}{} \\ 
$n$=250 & \multicolumn{1}{l}{} & \multicolumn{1}{l}{} & \multicolumn{1}{l}{} & \multicolumn{1}{l}{} &  & \multicolumn{1}{l}{} & \multicolumn{1}{l}{} & \multicolumn{1}{l}{} &  & \multicolumn{1}{l}{} & \multicolumn{1}{l}{} & \multicolumn{1}{l}{} \\ 
 & 0.5 & 0.1788 & 0.1038 & 0.1787 &  & 0.2832 & 0.0630 & 0.2831 &  & 0.4316 & 0.0511 & 0.4313 \\ 
 & 0.6 & 0.1374 & 0.1034 & 0.1359 &  & 0.1959 & 0.0979 & 0.1925 &  & 0.2802 & 0.1061 & 0.2716 \\ 
 & 0.7 & 0.1212 & 0.1001 & 0.1187 &  & 0.1571 & 0.1147 & 0.1460 &  & 0.2063 & 0.1401 & 0.1773 \\ 
 & 0.8 & 0.1143 & 0.0948 & 0.1118 &  & 0.1385 & 0.1231 & 0.1133 &  & 0.1811 & 0.1620 & 0.1201 \\ 
 & 0.9 & 0.1107 & 0.0910 & 0.1085 &  & 0.1342 & 0.1292 & 0.0903 &  & 0.1865 & 0.1820 & 0.0819 \\ 
\hline
\end{tabular}
\end{center}
\label{tab:mc_spec4}
\end{table}

\end{document}